\documentclass[11pt,twoside,a4paper]{article}
\usepackage{amsmath,amssymb,mathrsfs}
\usepackage{amsthm}

\usepackage{inputenc}
\usepackage{avant}
\usepackage[english]{babel}
\usepackage{thmtools,thm-restate}
\usepackage[nottoc]
{tocbibind}
\usepackage{graphicx}
\usepackage{enumitem}
\usepackage{tikz}
\usepackage{bbm}
\usepackage{esint}
\usepackage{authblk}
\usepackage{mathtools}

\usepackage{nicematrix}
\NiceMatrixOptions{cell-space-limits = 1pt}

\usepackage[inner=2.4cm,outer=2.4cm,top=2.8cm,bottom=2.8cm]{geometry}
\usepackage[all]{xy}
\usepackage{todonotes}

\usepackage[bookmarks=true]{hyperref}
\usepackage{xcolor}
\hypersetup{
    colorlinks,
    linkcolor={red!50!black},
    citecolor={blue!50!black},
    urlcolor={blue!80!black},
}

\mathtoolsset{showonlyrefs}  

\DeclareMathOperator{\id}{id}

\newcommand{\scal}[2]{\ensuremath{\langle #1 , #2 \rangle}} 
\newcommand{\norm}[1]{\left\lVert#1\right\rVert}
\newcommand{\normdot}{{\|\!\cdot\!\|}}

\newcommand{\Leb}{\mathscr{L}}

\newcommand{\N}{\mathbb{N}}
\newcommand{\R}{\mathbb{R}}

\newcommand{\p}{\mathtt p} 
\newcommand{\de}{\ensuremath{\, \mathrm d}} 

\newcommand\restr[2]{{
  \left.\kern-\nulldelimiterspace 
  #1 
  \right|_{#2} 
  }}

\newcommand{\cd}{\mathsf{CD}}
\newcommand{\bm}{\mathsf{BM}}
\newcommand{\BE}{\mathsf{BE}}

\newcommand{\MCP}{\mathsf{MCP}}
\newcommand{\Ric}{\mathrm{Ric}}

\newcommand{\X}{\mathsf{X}}
\newcommand{\I}{\mathcal{I}}
\newcommand{\M}{\mathcal{M}}

\newcommand{\Lip}{\mathsf {Lip}}

\newcommand{\di}{\mathsf d} 
\newcommand{\m}{\mathfrak m} 

\DeclareMathOperator{\Geo}{Geo}

\newcommand{\Prob}{\mathscr{P}}

\newcommand{\dis}{\mathcal D}

\newcommand{\sF}{sub-Finsler }
\newcommand{\sr}{sub-Riemannian }
\newcommand{\ariem}{almost-Riemannian }
\newcommand{\Ltwo}{L^2\big([0,1];(\R^k,\normdot)\big)}
\DeclareMathOperator{\ann}{Ann}
\newcommand{\sinom}{\sin_\Omega}
\newcommand{\cosom}{\cos_\Omega}
\newcommand{\sinomp}{\sin_{\Omega^\circ}}
\newcommand{\cosomp}{\cos_{\Omega^\circ}}
\newcommand{\Sbb}{\mathbb{S}}
\newcommand{\cut}{{\rm Cut}}
\newcommand{\e}{{\rm e}}

\newcommand{\hei}{\mathbb{H}}

\usepackage{titlesec}
\titleformat{\section}
  {\scshape\large\centering}
  {\thesection}{0.5em}{}
\titleformat{\subsection}
  {\scshape\centering}
  {\thesubsection}{0.4em}{}

\title{\textbf{Failure of the curvature-dimension condition in \sF manifolds}}
\date{\today}

\author{Mattia Magnabosco\footnote{Institut f\"ur Angewandte Mathematik, Universit\"at Bonn. \textit{E-mail}:  \href{mailto:magnabosco@iam.uni-bonn.de}{magnabosco@iam.uni-bonn.de}} \ and Tommaso Rossi\footnote{Institut f\"ur Angewandte Mathematik, Universit\"at Bonn. \textit{E-mail}: \href{mailto:rossi@iam.uni-bonn.de}{rossi@iam.uni-bonn.de}}}

\newtheoremstyle{remark}
        {10pt}
        {10pt}
        {}
        {}
        {\itshape}
        {.}
        {.4em}
        {}

\newtheoremstyle{proof}
        {10pt}
        {10pt}
        {}
        {}
        {\itshape}
        {.}
        {.4em}
        {}
        
\newtheoremstyle{definition}
        {10pt}
        {10pt}
        {}
        {}
        {\bfseries}
        {.}
        {.4em}
        {}

\newtheoremstyle{theorem}
        {10pt}
        {10pt}
        {\slshape}
        {}
        {\bfseries}
        {.}
        {.4em}
        {}

\theoremstyle{theorem}
\newtheorem{theorem}{Theorem}[section]
\newtheorem{prop}[theorem]{Proposition}
\newtheorem{corollary}[theorem]{Corollary}
\newtheorem{lemma}[theorem]{Lemma}
\newtheorem{conj}[theorem]{Conjecture}

\theoremstyle{definition}
\newtheorem{definition}[theorem]{Definition}

\theoremstyle{remark}
\newtheorem{remark}[theorem]{Remark}
\newtheorem{notation}[theorem]{Notation}

\theoremstyle{proof}
\newtheorem*{pro}{Proof}
 {\popQED\end{pro}}

\makeatletter
\renewcommand\xleftrightarrow[2][]{%
  \ext@arrow 9999{\longleftrightarrowfill@}{#1}{#2}}
\newcommand\longleftrightarrowfill@{%
  \arrowfill@\leftarrow\relbar\rightarrow}
\makeatother

\makeindex
\begin{document}

\maketitle

\begin{abstract}
The Lott--Sturm--Villani curvature-dimension condition $\cd(K,N)$ provides a synthetic notion for a metric measure space to have curvature bounded from below by $K$ and dimension bounded from above by $N$. It has been recently proved that this condition does not hold in \sr geometry for every choice of the parameters $K$ and $N$. In this paper, we extend this result to the context \sF geometry, showing that the $\cd(K,N)$ condition is not well-suited to characterize curvature in this setting. Firstly, we show that this condition fails in (strict) \sF manifolds equipped with a smooth strongly convex norm and with a positive smooth measure. Secondly, we focus on the \sF Heisenberg group, proving that curvature-dimension bounds can not hold also when the reference norm is less regular, in particular when it is of class $C^{1,1}$. The strategy for proving these results is a non-trivial adaptation of the work of Juillet \cite{MR4201410}, and it requires the introduction of new tools and ideas of independent interest. Finally, we demonstrate the failure of the (weaker) measure contraction property $\MCP(K,N)$ in the \sF Heisenberg group, equipped with a singular strictly convex norm and with a positive smooth measure. This result contrasts with what happens in the \sr Heisenberg group, which instead satisfies $\MCP(0,5)$.\vspace{4pt} \vspace{4pt} 
\end{abstract}

\tableofcontents

\section{Introduction}

In the present paper, we address the validity of the Lott--Sturm--Villani curvature-dimension (in short $\cd(K,N)$) condition in the setting of \sF geometry. In particular, we prove that this condition can not hold in a large class of \sF manifolds. Thus, on the one hand, this work shows that the $\cd(K,N)$ condition is not well-suited to characterize curvature in \sF geometry. On the other hand, we discuss how our results could provide remarkable insights about the geometry of $\cd(K,N)$ spaces.

\subsection{Curvature-dimension conditions}

In their groundbreaking works, Sturm \cite{MR2237206,MR2237207} and Lott--Villani \cite{MR2480619} introduced independently a synthetic notion of curvature-dimension bounds for non-smooth spaces, using Optimal Transport. Their theory stems from the crucial observation that, in the Riemannian setting, having a uniform lower bound on the Ricci curvature and an upper bound on the dimension, can be equivalently characterized in terms of a convexity property of suitable entropy functionals in the Wasserstein space. In particular, it was already observed in \cite{MR2142879} that the Ricci bound $\Ric \geq K\cdot g$ holds if and only if the Boltzmann--Shannon entropy functional is $K$-convex in the Wasserstein space. More generally, let $(M,g)$ be a complete Riemannian manifold, equipped with a measure of the form $\m=e^{-V}{\rm vol}_g$, where ${\rm vol}_g$ is the Riemannian volume and $V\in C^2(M)$. Given $K\in\R$ and $N\in (n,+\infty]$, Sturm \cite{MR2237207} proved that the (generalized) Ricci lower bound 
\begin{equation}
\label{eq:intro_ricci_bound}
       \Ric_{N,V} := \Ric  + \nabla^2 V - 
        \frac
            {\nabla V \otimes \nabla V}
            {N-n} \geq K \cdot g,
\end{equation}

holds if and only if a $(K,N)$-convexity inequality holds for R\'enyi entropy functionals, defined with respect to the reference measure $\m$. While \eqref{eq:intro_ricci_bound} involves a differential object, the Ricci tensor, entropy convexity can be formulated relying solely upon a reference distance and a reference measure, without the need of the underlying smooth structure of the Riemannian manifold. Therefore, it can be introduced in the non-smooth setting of metric measure spaces and taken
as definition of curvature-dimension bound. This condition is called $\cd(K,N)$ and represents a synthetic lower bound on the (Ricci) curvature by $K\in\R$ and a synthetic upper bound on the dimension by $N\in (1,\infty]$, see Definition \ref{def:CD}. In this sense, according to the discussion above, the $\cd(K,N)$ condition is coherent with the Riemannian setting. Moreover, it was proved by Ohta \cite{MR2546027} that the relation between curvature and $\cd(K,N)$ condition holds also in the context of Finsler manifolds.

Remarkably, $\cd(K,N)$ spaces (i.e.\ spaces satisfying the $\cd(K,N)$ condition) enjoy several geometric properties which hold in the smooth setting.
Some of them are expected (and in a way necessary) for a reasonable curvature-dimension bound, such as the scaling \cite{MR2237207}, tensorization \cite{MR2610378} and globalization \cite{MR4309491} properties or the monotonicity with respect to the parameters \cite{MR2237207}, i.e.
\begin{equation*}
    \cd(K',N') \implies \cd(K,N) \qquad \text{if }K'\geq K \text{ and }N'\leq N.
\end{equation*}
Others are completely non-trivial and highlight some notable geometric features. Among them, we mention the Bonnet--Myers diameter bound and the Bishop--Gromov inequality, that provides an estimate on the volume growth of concentric balls. Particularly interesting in the context of this work is the Brunn--Minkowski inequality $\bm(K,N)$, which, given two sets $A$ and $B$ in the reference metric measure space $(\X,\di,\m)$, provides a lower estimate on the measure of the set of $t$-midpoints
\begin{equation}
    M_t(A,B)=\left\{ x \in \X\,:\, \di(a,x)=t \di(a,b), \, \di(x,b)=(1-t) \di(a,b)\,\text{ for some }a\in A,b\in B\right \},
\end{equation}
in terms of $\m(A)$ and $\m(B)$, for every $t\in[0,1]$, cf. \eqref{eq:bm}. The notable feature of the $\bm(K,N)$ inequality is that its formulation does not invoke optimal transport, or Wasserstein interpolation, and because of that, it is easier to handle than the $\cd(K,N)$ condition. Nonetheless, it contains a strong information about the curvature of the underlying space, to the extent that it is equivalent to the $\cd(K,N)$ condition in the Riemannian setting,  cf. \cite{seminalpaper}. In particular, in the proof of Theorem \ref{thm:intro1} and Theorem \ref{thm:intro3}, we show the failure of the $\cd(K,N)$ condition by contradicting the Brunn--Minkowski inequality $\bm(K,N)$.

Finally, another fundamental property of the $\cd(K,N)$ condition is its stability with respect to the (pointed) measured Gromov--Hausdorff convergence \cite{MR2237207,MR2480619,GigMonSav}. This notion of convergence for metric measure spaces essentially combines the Hausdorff convergence for the metric side and the weak convergence for the reference measures. 
As in a metric measure space, the tangent spaces at a point are identified with a measured Gromov--Hausdorff limit procedure of suitably rescalings of the original space, the stability of the curvature-dimension condition implies that the metric measure tangents of a $\cd(K,N)$ space is a $\cd(0,N)$ space.

In the setting of metric measure spaces, it is possible to define other curvature-dimension bounds, such as the so-called measure contraction property (in short $\MCP(K,N)$), introduced by Ohta in \cite{MR2341840}. In broad terms, the $\MCP(K,N)$ condition can be interpreted as the Brunn--Minkowski inequality where one of the two sets degenerates to a point. In particular, it is implied by (and strictly weaker than) the $\bm(K,N)$ inequality, and therefore it is also a consequence of the $\cd(K,N)$ condition. 

\subsection{The curvature-dimension condition in \sr geometry}
\label{sec:CD-SR}

While in the Riemannian setting the $\cd(K,N)$ condition is equivalent to having bounded geometry, a similar result does not hold for \sr manifolds. Sub-Riemannian geometry is a generalization of Riemannian geometry where, given a smooth manifold $M$, we define a smoothly varying scalar product only on a subset of \emph{horizontal} directions $\dis_p\subset T_pM$ (called distribution) at each point $p\in M$. Under the so-called H\"ormander condition, $M$ is horizontally-path connected, and the usual length-minimization procedure yields a well-defined distance $\di_{SR}$. In particular, differently from what happens in Riemannian geometry, the rank of the distribution $r(p):=\dim \dis_p$ may be strictly less than the dimension of the manifold and may vary with the point. This may influence the behavior of geodesics, emphasizing singularities of the distance $\di_{SR}$. For this reason, we can not expect the $\cd(K,N)$ condition to hold for \emph{truly} \sr manifolds. This statement is confirmed by a series of papers, most notably \cite{MR4201410,MR4562156,rizzi2023failure}, that contributed to the proof of the following result (see also \cite{MR4061989} for the analogous result in Carnot groups).

\begin{theorem}
\label{thm:intro0}
    Let $M$ be a complete truly \sr manifold, equipped with a positive smooth measure $\m$. Then, the metric measure space $(M,\di_{SR},\m)$ does not satisfy the $\cd(K,N)$ condition, for any $K\in\R$ and $N\in(1,\infty)$.  
\end{theorem}

In \cite{MR4201410}, Juillet proved Theorem \ref{thm:intro0} for \sr manifolds where the rank of the distribution $r(p)$ is strictly smaller than the topological dimension $n:=\dim M$, for every $p\in M$. His strategy relies on the construction of two Borel subsets for which the Brunn--Minkowski inequality $\bm(K,N)$ does not hold. Namely, for all $R,\varepsilon>0$, one can find $A,B\subset M$ such that $\mathrm{diam}(A\cup B)<R$, $\m(A)\approx\m(B)$, and such that there exists $t\in (0,1)$ for which
\begin{equation}
\label{eq:bm_ineq}
    \m(M_t(A,B))\leq \frac{1}{2^{\mathcal{N}-n}}\,\m(B)(1+\varepsilon),
\end{equation}
where $\mathcal{N}$ is the so-called \emph{geodesic dimension} of $M$, see \cite[Def.\ 5.47]{MR3852258} for a precise definition. The sets $A$ and $B$ are metric balls of small radius, centered at the endpoints of a short segment of an \emph{ample geodesic}, see \cite{MR3852258} for details. The inequality \eqref{eq:bm_ineq} allows to contradict the Brunn--Minkowski inequality $\bm(K,N)$ if and only if the geodesic dimension $\mathcal N$ is strictly greater than $n$, which is the case if $r(p)<n$, for every $p\in M$. 

While Julliet's result is quite general, it does not include \ariem geometry. Roughly speaking, an \ariem manifold is a \sr manifold where the rank of the distribution coincides with the dimension of $M$, at almost every point. In \cite{MR4562156}, we addressed this issue, proposing a new strategy for proving Theorem \ref{thm:intro0} in this setting. Our idea is to 
exploit the following one-dimensional characterization of the $\cd(K,N)$ condition: 
\begin{equation}\label{eq:camon}
    \cd(K,N)\quad\Rightarrow\quad \cd^1(K,N),
\end{equation}
proved by Cavalletti and Mondino in \cite{MR3648975}, and contradict the $\cd^1(K,N)$ condition. On a metric measure space $(\X,\di,\m)$, given a $1$-Lipschitz function $u\in \Lip(\X)$, it is possible to partition $\X$ in one-dimensional transport rays, associated with $u$, and disintegrate the measure $\m$ accordingly. Then, the $\cd^1(K,N)$ condition asks for the validity of the $\cd(K,N)$ condition along the transport rays of the disintegration associated with $u$, for any choice of $u\in\Lip(\X)$. In \cite[Thm.\ 1.2]{MR4562156}, when $M$ is either strongly regular or $\dim M=2$, we are able to explicitly build a $1$-Lipschitz function, and compute the associated disintegration, showing that the $\cd(K,N)$ condition along the rays does not hold for any $K\in\R$ and $N\in (1,\infty)$. 


Most recently, Rizzi and Stefani \cite{rizzi2023failure} proposed yet another strategy to prove Theorem \ref{thm:intro0}. Differently from the strategies presented above, they pursue the ``Eulerian'' approach to curvature-dimension bounds, based on a suitable Gamma calculus, see \cite{MR3155209} for details. This approach can be adopted for metric measure spaces that satisfy the \emph{infinitesimal Hibertian} condition (cf. \cite{AmbrosioGigliSavare11-2,Gigli12}) which forces the space to be \emph{Riemannian-like} and ensures the linearity of the heat flow. According to \cite{AmbrosioGigliSavare12}, an infinitesimally Hilbertian $\cd(K,N)$ space supports the so-called Bakry--\'Emery inequality $\BE(K,\infty)$, which, in the \sr setting reads as 
\begin{equation}
\label{eq:be_lla}
    \frac{1}{2} \Delta\left(\|\nabla f\|^{2}\right) \geq g (\nabla f, \nabla \Delta f)+K\|\nabla f\|^{2}, \qquad \forall\,f\in C^\infty_c(M),
\end{equation}
where $\nabla$ is the horizontal gradient and $\Delta$ is the sub-Laplacian. In \cite{rizzi2023failure}, the authors show that \eqref{eq:be_lla} implies the existence of enough isometries on the metric tangent to force it to be Euclidean at \emph{each point}, proving Theorem \ref{thm:intro0} (including also the case $N=\infty$). 

\subsection{Other curvature-dimension bounds in \sr geometry}\label{sec:Illpayforthatmycourtesy}

Given that the $\cd(K,N)$ condition does not hold in \sr geometry, considerable efforts have been undertaken to explore potential curvature-dimension bounds that may hold in this class. A first observation in this direction is that the weaker $\MCP(K,N)$ condition does hold in many examples of \sr manifolds. In particular, it was proved by Juillet \cite{MR2520783} that the \sr Heisenberg group satisfies the $\MCP(0,5)$ condition, where the curvature-dimension parameters can not be improved. Moreover, in \cite{MR3852258} it was observed that the optimal dimensional parameter for the measure contraction property coincides with the geodesic dimension of the \sr Heisenberg group (i.e.\ $\mathcal N=5$). This result has been subsequently extended to a large class of \sr manifolds, including ideal Carnot groups \cite{MR3110060}, corank-$1$ Carnot groups \cite{MR3502622}, generalised H-type Carnot groups \cite{MR3848070} and two-step
analytic \sr structures \cite{MR4245620}. In all these cases, the $\MCP(0,N)$ condition holds with the dimensional parameter $N$ greater than or equal to the geodesic dimension $\mathcal N$. 

Another attempt is due to Milman \cite{MR4373164}, who introduced the quasi curvature-dimension condition, inspired by the interpolation inequalities along Wasserstein geodesics in ideal sub-Riemannian manifolds, proved by Barilari and Rizzi \cite{MR3935035}. Finally, these efforts culminated in the recent work by Barilari, Mondino and Rizzi \cite{barilari2022unified}, where the authors propose a unification of Riemannian and sub-Riemannian geometries in a comprehensive theory of synthetic Ricci curvature lower bounds. In the setting of gauge metric measure spaces, they introduce the $\cd(\beta,n)$ condition, encoding in the distortion coefficient $\beta$ finer geometrical information of the underlying structure. Moreover they prove that the $\cd(\beta,n)$ condition holds for compact fat \sr manifolds, thus substantiating the definition.

\subsection{Sub-Finsler manifolds and Carnot groups}\label{sec:babygoat}

In the present paper, we focus on \sF manifolds, which widely generalize both \sr and Finsler geometry. Indeed, in this setting, given a smooth manifold $M$, we prescribe a smoothly varying \emph{norm} (which needs not be induced by a scalar product) on the distribution $\dis_p\subset T_pM$, at each point $p\in M$. As in the \sr setting, $\dis$ must satisfy the H\"ormander condition, and consequently the length-minimization procedure among admissible curves gives a well-defined distance $\di_{SF}$. Note that, on the one hand, if the $\dis_p=T_pM$ for every $p\in M$, we recover the classical Finsler geometry. On the other hand, if the norm on $\dis_p$ is induced by a scalar product for every $p\in M$, we fall back into \sr geometry.  

Replacing the scalar product with a (possibly singular) norm is not merely a technical choice, as the metric structure of a \sF manifold reflects the singularities of the reference norm. Indeed, even though \sF manifolds can still be investigated by means of classical control theory \cite{AS-GeometricControl}, deducing finer geometrical properties is more delicate compared to what happens in the \sr setting, as the Hamiltonian function has a low regularity, cf. Section \ref{sec:sub-Finsler_geometry}. In this regard, \sF manifolds provide an interesting example of smooth structures which present both the typical \sr and Finsler singular behavior. A particularly relevant class of \sF manifolds is the one of \emph{\sF Carnot groups}.

\begin{definition}[Carnot group]
    A Carnot group is a connected, simply connected Lie group $G$ with nilpotent Lie algebra $\mathfrak{g}$, admitting a stratification
    \begin{equation*}
        \mathfrak g= \mathfrak g_1 \oplus \cdots \oplus \mathfrak g_k,
    \end{equation*} 
    where $\mathfrak g_{i+1}= [\mathfrak g_1,\mathfrak g_i]$, for every $i=1,\dots, k-1$, and $[\mathfrak g_1,\mathfrak g_k]=\{0\}$.
\end{definition}

\noindent Given a Carnot group $G$, if we equip the first layer $\mathfrak g_1$ of its Lie algebra with a norm, we naturally obtain a left-invariant \sF structure on $G$. We refer to the resulting manifold as a \sF Carnot group. 

Motivated from the results presented in the previous section, cf. Theorem \ref{thm:intro0}, and especially from the ones obtained in the present work (see Section \ref{sec:coldweather}), we formulate the following conjecture.

\begin{conj}\label{conj:carnot}
    Let $G$ be a \sF Carnot group, endowed with a positive smooth measure $\m$. Then, the metric measure space $(G,\di_{SF},\m)$ does not satisfy the $\cd(K,N)$ condition for any $K\in\R$ and $N\in(1,\infty)$.
\end{conj}

\noindent Our interest in Carnot groups stems from the fact that they are the only metric spaces that are locally compact, geodesic, isometrically homogeneous and self-similar (i.e. admitting a dilation) \cite{MR3283670}. According to this property, sub-Finsler Carnot groups naturally arise as metric tangents of metric measure spaces.

\begin{theorem}[Le Donne \cite{MR2865538}]
    Let $(\X,\di,\m)$ be a geodesic metric measure space, equipped with a doubling measure $\m$. Assume that, for $\m$-almost every $x\in \X$, the set ${\rm Tan}(\X, x)$ of all metric tangent spaces at $x$ contains only one element. Then, for $\m$-almost every $x \in \X$, the element in ${\rm Tan}(\X, x)$ is a \sF Carnot group $G$.    
\end{theorem}

\noindent In particular, this result applies to $\cd(K,N)$ spaces, where the validity of the doubling property is guaranteed by the Bishop--Gromov inequality. Moreover, as already mentioned, the metric measure tangents of a $\cd(K,N)$ space are $\cd(0,N)$. Therefore, the study of the $\cd(K,N)$ condition in \sF Carnot groups, and especially the validity of Conjecture \ref{conj:carnot}, has the potential to provide deep insights on the structure of tangents of $\cd(K,N)$ spaces. This could be of significant interest, particularly in connection with Bate's recent work \cite{MR4506771}, which establishes a criterion for rectifiability in metric measure spaces, based on the structure of metric tangents.

\subsection{Main results}\label{sec:coldweather}

The aim of this paper is to show the failure of the $\cd(K,N)$ condition in the \sF setting, with a particular attention to Conjecture \ref{conj:carnot}. Our results offer an advance into two different directions: on the one hand we deal with general \sF structures, where the norm is smooth, cf. Theorem \ref{thm:intro1} and Theorem \ref{thm:introtremezzino}, and, on the other hand, we deal with the \sF Heisenberg group, equipped with more general norms, cf. Theorem \ref{thm:intro2} and Theorem \ref{thm:intro3}. 

In order to extend the \sr result of Theorem \ref{thm:intro0} to the \sF setting, one can attempt to adapt the strategies discussed in Section \ref{sec:CD-SR}, however this can present major difficulties. Specifically, the argument developed in \cite{rizzi2023failure} has little hope to be generalized, because the infinitesimal Hilbertianity assumption does not hold in Finsler-like spaces, see \cite{MR2917125}. It is important to note that this is not solely a ``regularity" issue, in the sense that it also occurs when the norm generating the \sF structure is smooth, but not induced by a scalar product. Instead, the approach proposed in \cite{MR4562156} could potentially be applied to \sF manifolds as it relies on tools developed in the non-smooth setting, see \eqref{eq:camon}. However, adapting the \sr computations that led to a contradiction of the $\cd^1(K,N)$ condition seems non-trivial already when the reference norm is smooth. Finally, the strategy illustrated in \cite{MR4201410} hinges upon geometrical constructions and seems to be well-suited to generalizations to the \sF setting. In this paper, we build upon this observation and adapt the latter strategy to prove our main theorems. 

Our first result is about the failure of the $\cd(K,N)$ condition in \emph{smooth} \sF manifolds, cf. Theorem \ref{thm:casosmooth}.

\begin{theorem}\label{thm:intro1}
    Let $M$ be a complete \sF manifold with $r(p)<n:=\dim M$ for every $p\in M$, equipped with a smooth, strongly convex norm $\normdot$ and with a positive smooth measure $\m$. Then, the metric measure space $(M,\di_{SF},\m)$ does not satisfy the $\cd(K,N)$ condition, for any $K\in\R$ and $N\in (1,\infty)$.
\end{theorem}

\noindent This result is the \sF analogue of \cite[Cor.\ 1.2]{MR4201410}. Although the strategy of its proof follows the blueprint of \cite{MR4201410}, the adaptation to our setting is non-trivial and requires many intermediate results of independent interest. First of all, we establish the existence of geodesics \emph{without abnormal sub-segments}, cf. Theorem \ref{thm:existence_good_geod}, proposing a construction that is new even in the \sr framework and relies on the regularity properties of the distance function from the boundary of an open set. 
Note that, while these properties are well-known in the \sr context (cf. \cite[Prop. 3.1]{FPR-sing-lapl}), inferring them in the \sF setting becomes more challenging due to the low regularity of the Hamiltonian, which affects the regularity of the normal exponential map. Nonetheless, we settle a weaker regularity result that is enough for our purposes, cf. Theorem \ref{thm:E_diffeo_delta_smooth}. Second of all, we prove an analogue of the \sr theorem, indicating that the volume contraction along ample geodesics is governed by the geodesic dimension, see \cite[Thm.\ D]{MR3852258}. Indeed, in a smooth \sF manifold, we establish that the volume contraction rate along geodesics without abnormal sub-segments is bigger than $\dim M+1$, cf. Theorem \ref{thm:finite_order_jacobian}. Finally, we mention that these technical challenges lead us to a simplification of Juillet's argument (cf. Theorem \ref{thm:casosmooth}), which revealed itself to be useful also in the proof of Theorem \ref{thm:intro3}.

Observe that, since sub-Finsler Carnot groups are equiregular (and thus $r(p)<n$, for every $p\in G$) and complete, we immediately obtain the following consequence of Theorem \ref{thm:intro1}, which constitutes a significant step forward towards the proof of Conjecture \ref{conj:carnot}.

\begin{theorem}\label{thm:introtremezzino}
    Let $G$ be a \sF Carnot group, equipped with a smooth, strongly convex norm $\normdot$ and with a positive smooth measure $\m$. Then, the metric measure space $(G,\di_{SF},\m)$ does not satisfy the $\cd(K,N)$ condition, for any $K\in\R$ and $N\in (1,\infty)$.
\end{theorem}

In the proof of Theorem \ref{thm:intro1}, the smoothness of the norm plays a pivotal role in establishing the correct volume contraction rate along geodesics. When the norm is less regular, it is not clear how to achieve an analogue behavior in full generality. Nonetheless, we are able to recover such a result in the context of the \emph{\sF Heisenberg group} $\hei$, equipped with a possibly singular norm (see Section \ref{sec:Heisenberg}). Working in this setting is advantageous since, assuming strict convexity of the norm, the geodesics and the cut locus are completely described \cite{Bereszynski} and there exists an explicit expression for them in terms of convex trigonometric functions \cite{Nestogol} (see also \cite{MR3657277} for an example of the non-strictly convex case).

For the \sF Heisenberg group, we prove two different results, with the first addressing the case of $C^{1,1}$ and strictly convex reference norms and thus substantially relaxing the assumptions of Theorem \ref{thm:intro1}, cf. Theorem \ref{thm:noCDC11}.

\begin{theorem}\label{thm:intro3}
    Let $\hei$ be the \sF Heisenberg group, equipped with a strictly convex and $C^{1,1}$ norm and with a positive smooth measure $\m$. Then, the metric measure space $(\hei,\di_{SF},\m)$ does not satisfy the $\cd(K,N)$, for any $K\in\R$ and $N\in (1,\infty)$.
\end{theorem}

The proof of this statement follows the same lines of \cite[Cor.\ 1.2]{MR4201410}. However, the low regularity of the norm, and thus of geodesics, prevent us to exploit the same differential tools developed for Theorem \ref{thm:intro1}. Nonetheless, using the explicit expression of geodesics and of the exponential map, we can still recover an analogue result. In particular, guided by the intuition that a contraction rate along geodesics, similar to the one appearing in the smooth case, should still hold, we thoroughly study the Jacobian determinant of the exponential map. Building upon a fine analysis of convex trigonometric functions, cf. Section \ref{sec:convex_trigtrig} and Proposition \ref{prop:diffCcirc} where the regularity of the norm plays a pivotal role, we obtain an estimate on the contraction rate of the Jacobian determinant of the exponential map. However, this behavior is determined only for a large (in a measure-theoretic sense) set of covectors in the cotangent space. This poses additional challenges that we are able to overcome with a delicate density-type argument, together with an extensive use of the left-translations of the group, cf. Theorem \ref{thm:noCDC11} and also Remark \ref{rmk:flexiamo}. Remarkably, for every $C^{1,1}$ reference norm, we obtain the exact same contraction rate, equal to the geodesic dimension $\mathcal N=5$, that characterizes the \sr Heisenberg group. Finally, let us mention that, recently, Borza and Tashiro \cite{borzatashiro} showed that the Heisenberg group equipped with the $l^p$-norm can not satisfy the $\MCP(K,N)$ condition if $p>2$. In particular, $l^p$-norms with $p>2$ are strictly convex $C^{1,1}$ norms, hence their result implies Theorem \ref{thm:intro3} for those specific cases. However, we stress that their technique is based on explicit computations that can not be carried out in our general setting.

Our second result in the \sF Heisenberg group deals with the case of singular (i.e. non-$C^1$) reference norms, cf. Theorem \ref{thm:noCDnonC1}.

\begin{theorem}\label{thm:intro2}
    Let $\hei$ be the \sF Heisenberg group, equipped with a strictly convex norm $\normdot$ which is not $C^1$, and let $\m$ be a positive smooth measure on $\hei$. Then, the metric measure space $(\hei, \di_{SF}, \m)$ does not satisfy the measure contraction property $\MCP(K,N)$ for any $K\in \R$ and $N\in (1,\infty)$.
\end{theorem}

Observe that this theorem also shows the failure of the $\cd(K,N)$ condition, which is stronger than the measure contraction property $\MCP(K,N)$. However, Theorem \ref{thm:intro2} has an interest that goes beyond this consequence, as it reveals a phenomenon that stands in contrast to what typically happens in the \sr setting. In fact, as already mentioned in section \ref{sec:Illpayforthatmycourtesy}, the $\MCP(K,N)$ condition holds in many \sr manifolds, and, in the particular case of the \sr Heisenberg group, holds with (sharp) parameters $K=0$ and $N=5$. Therefore, Theorem \ref{thm:intro2} shows that a singularity of the reference norm can cause the failure of the measure contraction property $\MCP(K,N)$. 

Our strategy to show Theorem \ref{thm:intro2} consists in finding a set $A\subset\hei$, having positive $\m$-measure, such that the set of $t$-midpoints $M_t(\{\e\},A)$ (where $\e$ denotes the identity in $\hei$) is $\m$-null for every $t$ sufficiently small. This construction is based on a remarkable geometric property of the space $(\hei, \di_{SF}, \m)$, where geodesics can branch, even though they are unique. This has independent interest, as examples of branching spaces usually occur when geodesics are not unique.

We conclude this section highlighting that the combination of Theorem \ref{thm:intro3} and Theorem \ref{thm:intro2} proves Conjecture \ref{conj:carnot} for a large class of \sF Heisenberg groups. This is particularly interesting as the \sF Heisenberg groups are the unique \sF Carnot groups with Hausdorff dimension less than $5$ (or with topological dimension less than or equal to $3$), up to isometries. 

\subsection*{Structure of the paper}

In Section \ref{sec:gocheckthenumbers} we introduce all the necessary preliminaries. In particular, we present the precise definition of the $\cd(K,N)$ condition with some of its consequences, and we introduce the notion of \sF structure on a manifold. Section \ref{sec:sub-Finsler_geometry} is devoted to the study of the geometry of \sF manifolds. For the sake of completeness, we include generalizations of various \sr results, especially regarding the characterizations of normal and abnormal extremals and the exponential map. In Section \ref{sec:smoothSF}, we present the proof of Theorem \ref{thm:intro1}. We start by developing the building blocks for it, namely the existence of a geodesic without abnormal sub-segments and the regularity of the distance function. Then, we estimate the volume contraction rate, along the previously selected geodesic. Finally, in Section \ref{sec:argument_of_the_mago}, we adapt Juillet's strategy to obtain our first main theorem. Section \ref{sec:Heisenberg} collects our results about the failure of the $\cd(K,N)$ condition in the \sF Heisenberg group. After having introduced the convex trigonometric functions in Section \ref{sec:convex_trigtrig}, we use them to provide the explicit explicit expression of geodesics, cf. Section \ref{sec:saymyname}. We conclude by proving Theorem \ref{thm:intro3} in Section \ref{sec:CDheisenberg} and Theorem \ref{thm:intro2} in Section \ref{sec:noMCP}.

\subsection*{Acknowledgments} 
T.R. acknowledges support from the Deutsche
Forschungsgemeinschaft (DFG, German Research Foundation) through the collaborative research centre “The mathematics of emerging effects” (CRC 1060, Project-ID 211504053). The autors wish to thank Luca Rizzi for stimulating discussions regarding the regularity of the distance function in \sF manifolds and Lorenzo Portinale for his careful reading of the introduction.  

\section{Preliminaries}\label{sec:gocheckthenumbers}

\subsection{The \texorpdfstring{$\cd(K,N)$}{CD(K,N)} condition}

A metric measure space is a triple $(\X,\di,\m)$ where $(\X,\di)$ is a complete and separable metric space and $\m$ is a locally finite Borel measure on it. In the following $C([0, 1], \X)$ will stand for the space of continuous curves from $[0, 1]$ to $\X$. A curve $\gamma\in C([0, 1], \X)$ is called \textit{geodesic} if 
\begin{equation}
    \di(\gamma(s), \gamma(t)) = |t-s| \cdot  \di(\gamma(0), \gamma(1)) \quad \text{for every }s,t\in[0,1],
\end{equation}
and we denote by $\Geo(\X)$ the space of geodesics on $\X$. The metric space $(\X,\di)$ is said to be geodesic if every pair of points $x,y \in \X$ can be connected with a curve $\gamma\in \Geo(\X)$. 
For any $t \in [0, 1]$ we define the evaluation map $e_t \colon C([0, 1], \X) \to \X$ by setting $e_t(\gamma) := \gamma(t)$.
We denote by $\Prob(\X)$ the set of Borel probability measures on $\X$ and by $\Prob_2(\X) \subset \Prob(\X)$ the set of those having finite second moment. We endow the space $\Prob_2(\X)$ with the Wasserstein distance $W_2$, defined by
\begin{equation}
\label{eq:defW2}
    W_2^2(\mu_0, \mu_1) := \inf_{\pi \in \mathsf{Adm}(\mu_0,\mu_1)}  \int \di^2(x, y) \, \de \pi(x, y),
\end{equation}
where $\mathsf{Adm}(\mu_0, \mu_1)$ is the set of all the admissible transport plans between $\mu_0$ and $\mu_1$, namely all the measures in $\Prob(\X\times\X)$ such that $(\p_1)_\sharp \pi = \mu_0$ and $(\p_2)_\sharp \pi = \mu_1$. The metric space $(\Prob_2(\X),W_2)$ is itself complete and separable, moreover, if $(\X,\di)$ is geodesic, then $(\Prob_2(\X),W_2)$ is geodesic as well. In particular, every geodesic $(\mu_t)_{t\in [0,1]}$ in $(\Prob_2(\X),W_2)$ can be represented with a measure $\eta \in \Prob(\Geo(\X))$, meaning that $\mu_t = (e_t)_\# \eta$.

We are now ready to introduce the $\cd(K,N)$ condition, pioneered by Sturm and Lott--Villani \cite{MR2237206,MR2237207,MR2480619}. As already mentioned, this condition aims to generalize, to the context metric measure spaces, the notion of having Ricci curvature bounded from below by $K\in\R$ and dimension bounded above by $N>1$. In order to define the $\cd(K,N)$ condition, let us introduce the following distortion coefficients: for every $K \in \R$ and $N\in (1,\infty)$,
\begin{equation}\label{eq:tau}
    \tau_{K,N}^{(t)}(\theta):=t^{\frac{1}{N}}\left[\sigma_{K, N-1}^{(t)}(\theta)\right]^{1-\frac{1}{N}},
\end{equation}
where
\begin{equation}
\sigma_{K,N}^{(t)}(\theta):= 
\begin{cases}

\displaystyle  \frac{\sin(t\theta\sqrt{K/N})}{\sin(\theta\sqrt{K/N})} & \textrm{if}\  N\pi^{2} > K\theta^{2} >  0, \crcr
t & \textrm{if}\ 
K =0,  \crcr
\displaystyle   \frac{\sinh(t\theta\sqrt{-K/N})}{\sinh(\theta\sqrt{-K/N})} & \textrm{if}\ K < 0.
\end{cases}
\end{equation}

\begin{remark}\label{rmk:limit_dist_coeff}
    Observe that for every $K\in \R$, $N\in (1,\infty)$ and $t\in [0,1]$ we have 
    \begin{equation*}
        \lim_{\theta\to 0} \sigma_{K,N}^{(t)} (\theta) = t \qquad \text{ and } \qquad \lim_{\theta\to 0} \tau_{K,N}^{(t)} (\theta) = t.
    \end{equation*}
\end{remark}

\begin{definition}\label{def:CD}
A metric measure space $(\X,\di,\m)$ is said to be a $\cd(K,N)$ space (or to satisfy the $\cd(K,N)$ condition) if for every pair of measures $\mu_0=\rho_0\m,\mu_1= \rho_1 \m \in \Prob_2(\X)$, absolutely continuous with respect to $\m$, there exists a $W_2$-geodesic $(\mu_t)_{t\in [0,1]}$ connecting them and induced by $\eta \in \Prob(\Geo(\X))$, such that for every $t\in [0,1]$, $\mu_t =\rho_t \m \ll \m$ and the following inequality holds for every $N'\geq N$ and every $t \in [0,1]$
\begin{equation}\label{eq:CDcond}
    \int_\X \rho_t^{1-\frac 1{N'}} \de \m \geq \int_{\X \times \X} \Big[ \tau^{(1-t)}_{K,N'} \big(\di(x,y) \big) \rho_{0}(x)^{-\frac{1}{N'}} +    \tau^{(t)}_{K,N'} \big(\di(x,y) \big) \rho_{1}(y)^{-\frac{1}{N'}} \Big]    \de\pi( x,y),
\end{equation}
where $\pi= (e_0,e_1)_\# \eta$.
\end{definition}


One of the most important merits of the $\cd(K,N)$ condition is that it is sufficient to deduce geometric and functional inequalities that hold in the smooth setting. An example which is particularly relevant for this work is the so-called Brunn--Minkowski inequality, whose definition in the metric measure setting requires the following notion.

\begin{definition}
     Let $(\X,\di)$ be a metric space and let $A,B \subset \X$ be two Borel subsets. Then for $t\in (0,1)$, we defined the set of $t$-\emph{midpoints} between $A$ and $B$ as
    \begin{align}
        M_t(A,B) := 
        \{
            x \in \X \, : \, x = \gamma(t)
                \, , \, 
            \gamma \in \Geo(\X)
                \, , \, 
            \gamma(0) \in A \, , 
                \ \text{and} \  
            \gamma(1) \in B
        \} 
            \, .
    \end{align}
\end{definition}


\noindent We can now introduce the metric measure version of the Brunn--Minkowski inequality, whose formulation is stated in terms of the distortion coefficients \eqref{eq:tau}.

\begin{definition}
    \label{def:bruno}
    Given $K \in \R$ and $N\in (1,\infty)$, we say that a metric measure space $(\X,\di, \m)$ satisfies the \emph{Brunn--Minkowski inequality} $\bm(K,N)$ if, for every nonempty $A,B \subset \text{spt}(\m)$ Borel subsets, $t \in (0,1)$, we have
        \begin{align}   \label{eq:bm}
            \m \big(M_t(A,B)\big) \big)^ \frac{1}{N} 
                \geq 
            \tau_{K,N}^{(1-t)} (\Theta(A,B)) \cdot \m(A)^ \frac{1}{N} 
                + 
            \tau_{K,N}^{(t)} (\Theta(A,B)) \cdot \m(B)^ \frac{1}{N}
                \, ,
        \end{align}
     where 
    \begin{align}   \label{eq:def_Theta}
        \Theta(A,B):=
        \left\{
        \begin{array}{ll}\displaystyle    \inf_{x \in A,\, y \in B} \di(x, y) 
                & \text { if } K \geq 0 \, , \\ 
        \displaystyle
            \sup _{x \in A,\, y \in B} \di(x, y) 
                & \text { if } K<0 \, .
        \end{array}
        \right. 
    \end{align}
\end{definition}

As already mentioned, the Brunn--Minkowski inequality is a consequence of the $\cd(K,N)$ condition, in particular we have that 
\begin{equation*}
    \cd(K,N) \implies \bm(K,N),
\end{equation*}
for every $K\in \R$ and every $N\in (1,\infty)$. 
In Sections \ref{sec:smoothSF} and \ref{sec:Heisenberg}, we are going to disprove the $\cd(K,N)$ condition for every choice of the parameters $K\in \R$ and $N\in (1,\infty)$, by contradicting the Brunn--Minkowski inequality $\bm(K,N)$. A priori, this is a stronger result than the ones stated in Theorem \ref{thm:intro1} and Theorem \ref{thm:intro3}, since the Brunn--Minkowski inequality is (in principle) weaker than the $\cd(K,N)$ condition. However, recent developments (cf. \cite{seminalpaper,Magnabosco-Portinale-Rossi:2022b}) suggest that the Brunn--Minkowski $\bm(K,N)$ could be equivalent to the $\cd(K,N)$ condition in a wide class of metric measure spaces.

Another curvature-dimension bound, which can be defined for metric measure spaces, is the so-called measure contraction property (in short $\MCP(K,N)$), that was introduced by Ohta in \cite{MR2341840}. The idea behind it is basically to require the $\cd(K,N)$ condition to hold when the first marginal degenerates to $\delta_x$, a delta-measure at $x\in\text{spt}(\m)$, and the second marginal is $\frac{\m|_A}{\m(A)}$, for some Borel set $A\subset\X$ with $0<\m(A)<\infty$.

\begin{definition} [$\mathsf{MCP}(K,N)$ condition]
\label{def:mcp}
    Given $K\in\R$ and $N\in (1,\infty)$, a metric measure space $(\X,\di,\m)$ is said to satisfy the \emph{measure contraction property} $\mathsf{MCP}(K,N)$ if for every $x\in\text{spt}(\m)$ and a Borel set $A\subset\X$ with $0<\m(A)<\infty$, there exists a Wasserstein geodesic induced by $\eta \in \Prob(\Geo(\X))$ connecting $\delta_x$ and $\frac{\m|_A}{\m(A)}$ such that, for every $t\in[0,1]$,
    \begin{equation}
    \label{eq:mcp_def}
        \frac{1}{\m(A)}\m\geq(e_t)_\#\Big(\tau_{K,N}^{(t)}\big(\di(\gamma(0),\gamma(1))\big)^N\eta(\text{d}\gamma)\Big).
    \end{equation}
\end{definition}

\begin{remark}
\label{rmk:SIUUUUUUU}
For our purposes, we will use an equivalent formulation of the inequality \eqref{eq:mcp_def}, which holds whenever geodesics are unique, cf. \cite[Lemma 2.3]{MR2341840} for further details. More precisely, let $x\in\text{spt}(\m)$ and a Borel set $A\subset\X$ with $0<\m(A)<\infty$. Assume that for every $y\in A$, there exists a unique geodesic $\gamma_{x,y}:[0,1]\to \X$ joining $x$ and $y$. Then, \eqref{eq:mcp_def} is verified for the measures $\delta_x$ and $\frac{\m|_A}{\m(A)}$ if and only if 
\begin{equation}
\label{eq:tj_pantaloncini}
    \m\big(M_t(\{x\},A'))\big)\geq \int_{ A'}\tau_{K,N}^{(t)}(\di(x,y))^N \de\m(y),\qquad\text{for any Borel }A'\subset A.
\end{equation}
\end{remark}

\noindent The $\MCP(K,N)$ condition is weaker than the $\cd(K,N)$ one, i.e. 
\begin{equation*}
    \cd(K,N) \implies \MCP(K,N),
\end{equation*}
for every $K\in \R$ and every $N\in (1,\infty)$. In Theorem \ref{thm:casosmooth} (for the case of non-ample geodesics) and Theorem \ref{thm:noCDnonC1} (cf. Theorem \ref{thm:intro2}) in the Heisenberg group, equipped with singular norms, we contradict the $\MCP(K,N)$ condition. More precisely, we find a counterexample to \eqref{eq:tj_pantaloncini}.

\subsection{Strictly and strongly convex norms on \texorpdfstring{$\R^k$}{Rk}}

We say that a function $f:\R^k\to \R$ is \emph{strictly convex} if, for any $x,y\in\R^k$ such that $x\neq y$, 
\begin{equation}
    f(tx+(1-t)y)<tf(x)+(1-t)f(y),\qquad\forall\, t\in (0,1).
\end{equation}
Furthermore, let $\normdot:\R^k\to \R_{\geq 0}$ be a norm, then we say that $f$ is \emph{strongly convex} with respect to $\normdot$ if there exists $\alpha>0$ such that, for every $x,y\in\R^k$, 
\begin{equation}
    f(tx+(1-t)y)\leq tf(x)+(1-t)f(y)-\frac{\alpha}{2}t(1-t)\norm{x-y}^2,\qquad\forall\,t\in[0,1].
\end{equation}

\begin{definition}[Strictly and strongly convex norm]
    Let $\normdot:\R^k\to\R_{\geq 0}$ be a (symmetric) norm on $\R^k$. We say that $\normdot$ is \emph{strictly convex} if the function $f:\R^k\to \R$ defined by $f(x):=\norm{x}^2$ is strictly convex. Similarly, we say that $\normdot$ is \emph{strongly convex} if $f$ is strongly convex with respect to $\normdot$.
\end{definition}

Note that, according to \cite[Prop.\ 1.6]{MR1079061}, $\normdot$ is strictly convex if and only if the associated unit ball $B^{\norm{\cdot}}_1(0)$ is a strictly convex set. Whereas, \cite[Prop.\ 2.11]{MR1079061} implies that $\normdot$ is strongly convex if and only if the associated unit ball $B^{\norm{\cdot}}_1(0)$ is a uniformly convex set (see also \cite[Thm.\ 4.1(f)]{kerdreux2021local}). We recall below a well-known result on the relation between a norm $\normdot$ and its dual norm $\normdot_*$, cf. \cite[Prop.\ 2.6]{MR1623472} or \cite[Def.\ 3.1]{kerdreux2021local}.


\begin{prop}\label{prop:propunderduality}
    Let $\normdot$ be a norm on $\R^k$, and let $\normdot_*$ be its dual norm, then:
    \begin{itemize}
        \item [(i)] $\normdot$ is a strictly convex norm if and only if $\normdot_*$ is a $C^1$ norm, i.e. $\normdot_*\in C^1(\R^k\setminus\{0\})$;
        \item [(ii)] $\normdot$ is a strongly convex norm if and only if $\normdot_*$ is a $C^{1,1}$ norm, i.e. $\normdot_*\in C^{1,1}(\R^k\setminus\{0\})$.
    \end{itemize}
\end{prop}


We conclude this section by providing an explicit description of the dual element of $v\in(\R^k,\normdot)$, when $\normdot$ is strictly convex and $C^1$. Recall that its dual vector $v^*\in(\R^k,\normdot)^*$ is uniquely characterized by
    \begin{equation}\label{eq:dualvector}
        \norm{v^*}_*= \norm{v} \qquad \text{and} \qquad \scal{v^*}{v}= \norm{v}^2,
    \end{equation}
where $\scal{\cdot}{\cdot}$ is the dual coupling.

\begin{lemma}\label{lem:banachduality}
Let $(\R^k,\normdot)$ be a normed space and assume $\normdot:\R^k \to \R_+$ is a strictly convex $C^1$ norm. Then, for every non-zero vector $v\in \R^k$, it holds that
\begin{equation}
    v^*= \norm{v} \cdot d_v \normdot.
\end{equation}
\end{lemma}

\begin{proof}
     Set $\lambda:= d_v \normdot\in (\R^k,\normdot)^*$, where we recall that 
    \begin{equation}
        d_v\normdot(u):=\lim_{t\to0} \frac{ \norm{v+tu}-\norm{u}}{t},\qquad\forall\,u,v\in\R^k.
    \end{equation}
    Then, on the one hand, it holds that 
    \begin{equation}
        \scal{\lambda}{v}=d_v\normdot(v)=\lim_{t\to0} \frac{ \norm{v+tv}-\norm{v}}{t}= \norm{v}.
    \end{equation}
    On the other hand, we have
    \begin{equation*}
        \norm{\lambda}_* = \sup_{u \in B_1} \scal{\lambda}{u} = \sup_{u \in B_1}d_v\normdot(u) = \sup_{u \in B_1} \lim_{t\to0} \frac{1}{t} \big( \norm{v+tu}-\norm{v}\big) \leq \sup_{u \in B_1} \norm{u} = 1, 
    \end{equation*}
    where $B_1:=B_1^{\|\cdot\|}(0)\subset\R^k$ is the ball of radius $1$ and centered at $0$, with respect to the norm $\normdot$. The converse inequality can be obtained by taking $u=\frac{v}{\norm{v}}$. Finally, the conclusion follows by homogeneity of the dual norm.
\end{proof}

\subsection{Sub-Finsler structures}
\label{sec:prelim_sf}

Let $M$ be a smooth manifold of dimension $n$ and let $k\in\N$. A \emph{\sF structure} on $M$ is a couple $(\xi,\normdot)$ where $\normdot: \R^k \to \R_+$ is a norm on $\R^k$ and $\xi: M\times\R^k\rightarrow TM$ is a morphism of vector bundles such that:
\begin{enumerate}[label=(\roman*)]
    \item each fiber of the (trivial) bundle $M\times\R^k$ is equipped with the norm $\normdot$;
    \item  The set of horizontal vector fields, defined as 
    \begin{equation}
    \dis := \big\{ \xi\circ\sigma \, :\, \sigma \in \Gamma(M\times\R^k) \big\} \subset \Gamma(TM),
    \end{equation}
    is a \emph{bracket-generating} family of vector fields (or it satisfies the H\"ormander condtion), namely setting
    \begin{equation}
        {\rm Lie}_q(\dis):=\big\{X(q)\,:\, X\in\text{span}\{[X_1,\ldots,[X_{j-1},X_j]]\,:\, X_i\in\dis,j\in\N\} \big\},\qquad\forall\,q\in M,
    \end{equation}
    we assume that ${\rm Lie}_q(\dis)=T_qM$, for every $q\in M$.
\end{enumerate}

\begin{definition}[Smooth \sF manifold]
   Let $(\xi,\normdot)$ be a \sF structure on $M$. We say that $M$ is a \emph{smooth \sF manifold}, if the norm $\normdot$ is strongly convex and smooth, i.e.\ $\normdot\in C^\infty(\R^k\setminus\{0\})$.
\end{definition}

\begin{remark}
    Although this definition is not completely general in \sF context, since it does not allow the norm to vary on the fiber of $M\times\R^k$, it includes \sr geometry (where $\normdot$ is induced by a scalar product), as every \sr structure is equivalent to a free one, cf. \cite[Sec.\ 3.1.4]{ABB-srgeom}.  
\end{remark}

At every point $q\in M$ we define the \emph{distribution} at $q$ as 
\begin{equation}
\label{eq:distribution}
    \dis_q := \big\{ \xi(q,w) \, :\, w \in \R^k \big\}=\big\{X(q)\,:\,X\in\dis\big\} \subset T_qM. 
\end{equation}
This is a vector subspace of $T_qM$ whose dimension is called \emph{rank} (of the distribution) and denoted by $r(q):=\dim\dis_q\leq n$. Moreover, the distribution is described by a family of horizontal vector fields. Indeed, letting $\{e_i\}_{i=1, \dots, k}$ be the standard basis of $\R^k$, the \emph{generating frame} is the family $\{X_i\}_{i=1, \dots, k}$, where 
\begin{equation*}
    X_i(q) := \xi(q, e_i) \qquad\forall\,q\in M, \quad \text{for }i= 1, \dots ,k.
\end{equation*}
Then, according to \eqref{eq:distribution}, $\dis_q= \text{span} \{X_1(q), \dots, X_k(q)\}$. On the distribution we define the \emph{induced norm} as
\begin{equation*}
    \norm{v}_q := \inf\big\{ \norm{w} \, :\, v = \xi(q,w) \big\} \qquad \text{for every }v\in \dis_q.
\end{equation*}
Since the infimum is actually a minimum, the function $\normdot_q$ is a norm on $\dis_q$, so that $(\dis_q, \normdot_q)$ is a normed space. Moreover, the norm depends smoothly on the base point $q\in M$. A curve $\gamma: [0,1]\to M$ is \emph{admissible} if its velocity $\dot\gamma(t)$ exists almost everywhere and there exists a function $u=(u_1,\ldots,u_k)\in\Ltwo$ such that
\begin{equation}
\label{eq:admissible_curve}
    \dot\gamma(t)=\sum_{i=1}^k u_i(t)X_i(\gamma(t)),\qquad\text{for a.e. }t\in [0,1].
\end{equation}
The function $u$ is called \emph{control}. Furthermore, given an admissible curve $\gamma$, there exists $\bar u=(\bar u_1,\dots, \bar u_k): [0,1]\to \R^k$ such that 
\begin{equation}
\label{eq:minimal_control}
    \dot\gamma(t) = \sum_{i=1}^k \bar u_i(t) X_i(\gamma(t)),
\qquad\text{and}\qquad
    \norm{\dot\gamma(t)}_{\gamma(t)} = \norm{\bar u(t)},\qquad\text{for a.e. }t\in[0,1].
\end{equation}
The function $\bar u$ is called \emph{minimal control}, and it belongs to $\Ltwo$, cf. \cite[Lem. 3.12]{ABB-srgeom}. We define the \emph{length} of an admissible curve:
\begin{equation}
    \ell(\gamma):=\int_0^1 \norm{\dot\gamma(t)}_{\gamma(t)} \de t\in[0,\infty).
\end{equation}
We can rewrite the length of a curve as the $L^1$-norm of the associated minimal control, indeed by \eqref{eq:minimal_control},
\begin{equation}
\label{eq:length_fun_control}
    \ell(\gamma)=\int_0^1\norm{\bar u(t)}\de t=\norm{\bar u}_{L^1([0,1];(\R^k,\|\cdot\|))}.
\end{equation}
For every couple of points $q_0,q_1\in M$, define the \emph{\sF distance} between them as
\begin{equation*}
    \di_{SF} (q_0,q_1)= \inf \left\{\ell(\gamma)\, :\, \gamma \text{ admissible, } \gamma(0)=q_0 \text{ and }\gamma(1)=q_1\right\}.
\end{equation*}
Since every norm on $\R^k$ is equivalent to the standard scalar product on $\R^k$, it follows that the \sr structure on $M$ given by $(\xi,\scal{\cdot}{\cdot})$ induces an equivalent distance. Namely, denoting by $\di_{SR}$ the induced \sr distance, there exist constants $C>c>0$ such that
\begin{equation}
\label{eq:equivalent_sr_distance}
    c\,\di_{SR}\leq \di_{SF} \leq C\di_{SR},\qquad\text{on }M\times M.
\end{equation}
Thus, as a consequence of the classical Chow--Rashevskii Theorem in \sr geometry, we obtain the following.
\begin{prop}[Chow--Rashevskii]
    Let $M$ be a \sF manifold. The \sF distance is finite, continuous on $M\times M$ and the induced topology is the manifold one. 
\end{prop}

\noindent From this proposition, we get that $(M,\di_{SF})$ is a locally compact metric space. The local existence of minimizers of the length functional can be obtained as in the \sr setting, in particular, one can repeat the proof of \cite[Thm. 3.43]{ABB-srgeom}. Finally, if $(M,\di_{SF})$ is complete, then it is also a geodesic metric space. 

\section{The geometry of smooth \sF manifolds}
\label{sec:sub-Finsler_geometry}

\subsection{The energy functional and the optimal control problem}

Let $\gamma :[0,1]\rightarrow M$ be an admissible curve. Then, we define the \emph{energy} of $\gamma$ as 
\begin{equation}
    J(\gamma)=\frac{1}{2} \int_0^1 \norm{\dot\gamma(t)}^2_{\gamma(t)} \de t.
\end{equation}
By definition of admissible curve, $J(\gamma)<+\infty$. In addition, a standard argument shows that $\gamma:[0,1]\rightarrow M$ is a minimum for the energy functional if and only if it is a minimum of the length functional with constant speed. 

\begin{remark}
    The minimum of $J$ is not invariant under reparametrization of $\gamma$, so one needs to fix the interval where the curve is defined. Here and below, we choose $[0,1]$. 
\end{remark}

\noindent The problem of finding geodesics between two points $q_0,q_1\in M$ can be formulated using the energy functional as the following costrained minimization problem:

\begin{equation}\label{eq:Ptimalcontrolproblem}\tag{P}
    \begin{cases}
    \gamma:[0,1]\to M,\quad\text{admissible},\\
    \gamma(0)=q_0 \text{ and }\gamma(1)=q_1, \\
    J(\gamma) = \displaystyle\frac{1}{2} \int_0^1 \norm{\dot\gamma(t)}^2_{\gamma(t)} \de t \to \min.
    \end{cases}
\end{equation}

The problem \eqref{eq:Ptimalcontrolproblem} can be recasted as an optimal control problem. First of all, a curve is admissible if and only if there exists a control in $\Ltwo$ satisfying \eqref{eq:admissible_curve}. Second of all, we can consider the energy as a functional on the space of controls. Indeed, as in \eqref{eq:length_fun_control}, given an admissible curve $\gamma:[0,1]\to M$, we let $\bar u\in \Ltwo$ be its minimal control, as in \eqref{eq:minimal_control}. Then, we have
\begin{equation}
\label{eq:energy_minimal_control}
    J(\gamma)=\frac12 \int_0^1\norm{\bar u(t)}^2\de t = \frac12 \norm{\bar u}^2_{L^2([0,1];(\R^k,\|\cdot\|))}.
\end{equation}
Hence, we regard the energy as a functional on $\Ltwo$, namely 
\begin{equation}
    J:\Ltwo\to\R_+;\qquad J(u):=\frac12\norm{u}^2_{L^2([0,1];(\R^k,\|\cdot\|))},
\end{equation}
and we look for a constrained minimum of it. Thus, the problem \eqref{eq:Ptimalcontrolproblem} becomes: 
\begin{equation}\label{eq:P'timalcontrolproblem}\tag{P$'$}
    \begin{cases}
    \dot\gamma(t) = \displaystyle \sum_{i=1}^k u_i(t) X_i(\gamma(t)),  \\
    \gamma(0)=q_0 \text{ and }\gamma(1)=q_1, \\
    J(u) = \displaystyle\frac{1}{2} \int_0^1 \norm{u(t)}^2 \de t \to \min.
    \end{cases}
\end{equation}
Note that, by \eqref{eq:energy_minimal_control}, the solutions of \eqref{eq:Ptimalcontrolproblem} and \eqref{eq:P'timalcontrolproblem} coincide. An application of Pontryagin Maximum Principle (see \cite[Thm. 12.10]{AS-GeometricControl}) yields necessary conditions for optimality. For every $u\in \R^k$ and $\nu\in \R$, introduce the following Hamiltonian: 
\begin{equation}
\label{eq:hamiltonian_fun}
    h_u^\nu (\lambda) := \scal{\lambda}{\xi(\pi(\lambda),u)}+ \frac \nu 2 \norm{u}^2, \qquad \forall \lambda \in T^*M.
\end{equation}
Recall that for $h\in C^1(T^*M)$, its Hamiltonian vector field $\vec h\in {\rm Vec}(T^*M)$ is defined as the unique vector field in $T^*M$ satisfying
\begin{equation}
    d_\lambda h = \sigma(\cdot, \vec h(\lambda)),\qquad\forall\,\lambda\in T^*M,
\end{equation}
where $\sigma$ is the canonical symplectic form on $T^*M$.

\begin{theorem}[Pontryagin Maximum Principle]
\label{thm:PMP}
Let $M$ be a \sF manifold and let $(\gamma,\bar u)$ be a solution of \eqref{eq:P'timalcontrolproblem}. Then, there exists $(\nu,\lambda_t)\ne 0$, where $\nu\in \R$ and $\lambda_t\in T^*_{\gamma(t)}M$ for every $t\in [0,1]$, such that
\begin{equation}\tag{H}\label{eq:Hamiltonian_system}
    \begin{cases}
        \displaystyle \dot \lambda_t = \vec h_{\bar u(t)}^\nu (\lambda_t) & \text{for a.e. }t\in [0,1],\\
        \displaystyle h_{\bar u(t)}^\nu (\lambda_t)= \max_{v\in \R^k} h_{v}^\nu (\lambda_t) & \text{for a.e. }t\in [0,1], \\
        \nu\leq 0.&
    \end{cases}
\end{equation}
\end{theorem}

\begin{definition}
If $\nu <0 $ in \eqref{eq:Hamiltonian_system}, $(\lambda_t)_{t\in [0,1]}$ is called \emph{normal extremal}. If $\nu =0 $ in \eqref{eq:Hamiltonian_system}, $(\lambda_t)_{t\in [0,1]}$ is called \emph{abnormal extremal}.
\end{definition}

\begin{remark}
By homogeneity of the Hamiltonian system, if $\nu \ne 0$ in \eqref{eq:Hamiltonian_system} we can fix $\nu=-1$.
\end{remark}

\subsection{Characterization of extremals and the exponential map}

In this section, we recall some characterizations of normal and abnormal extremal, which are well-known in \sr geometry. We include the proofs in our case, for the sake of completeness.

Recall that the annihilator $\ann(\dis)\subset T^*M$ is defined by 
\begin{equation}
    \ann(\dis)_q := \{\lambda \in T^*_qM \,:\, \scal{\lambda}{w}=0, \, \forall\, w\in \dis_q\},\qquad\forall\,q\in M.
\end{equation}

\begin{lemma}
\label{lem:abnormal_annihilator}
    Let $M$ be a \sF manifold. Let $(\gamma,\bar u)$ be a non-trivial solution to \eqref{eq:P'timalcontrolproblem} and let $(\lambda_t)_{t\in [0,1]}$ be its lift. Then, $(\lambda_t)_{t\in [0,1]}$ is an abnormal extremal if and only if $\lambda_t\neq 0$ and $\lambda_t\in\ann(\dis)_{\gamma(t)}$ for every $t\in [0,1]$, where $\gamma(t):=\pi(\lambda_t)$.
\end{lemma}

\begin{proof}
The claim is an easy consequence of the maximization property of the Hamiltonian along the dynamic. More precisely, by \eqref{eq:Hamiltonian_system}, we have 
\begin{equation}
\label{eq:maximization_aux}
    h_{\bar u(t)}^\nu(\lambda_t)=\max_{v\in\R^k}h_v^\nu(\lambda_t), \qquad\text{for a.e. }t\in [0,1],
\end{equation}
where the function $h_u^\nu$ is defined in \eqref{eq:hamiltonian_fun}. Assume that $\nu=0$, then \eqref{eq:maximization_aux} reads as 
\begin{equation}
    \scal{\lambda_t}{\xi(\gamma(t),\bar u(t))}=\max_{v\in\R^k}\,\scal{\lambda_t}{\xi(\gamma(t),v)},\qquad\text{for a.e. }t\in [0,1],
\end{equation}
with $\lambda_t\neq0$ for every $t\in [0,1]$. Now, since $\xi$ is linear in the controls, the right-hand side is $+\infty$ unless $\lambda_t\in\ann(\dis)_{\gamma(t)}$ for a.e. $t\in [0,1]$. By continuity of $t\mapsto\lambda_t$, this is true for every $t\in[0,1]$. Conversely, assume that $\lambda_t\in\ann(\dis)_{\gamma(t)}$, then the maximization condition \eqref{eq:maximization_aux} becomes
\begin{equation}
    \frac \nu 2 \|\bar u(t)\|^2=\max_{v\in\R^k}\frac \nu 2 \|v\|^2.
\end{equation}
Since $\nu\leq 0$, we may distinguish two cases, either $\nu=0$ and the extremal is abnormal, or $\nu=-1$ and the extremal is normal. In the second case, the optimal control must be $0$, so that $\dot\lambda_t=0$ and the extremal is constant and constantly equal to $\lambda_0$. Since we are assuming $(\lambda_t)_{t\in[0,1]}$ to be non-constant the latter can not happen.
\end{proof}

The \emph{\sF (or maximized) Hamiltonian} is defined as 
\begin{equation}
\label{eq:sf_hamiltonian}
    H(\lambda) := \max_{u\in \R^k} h_{u}^{-1} (\lambda_t) = \max_{u\in \R^k} \Bigg(\sum_{i=1}^k \scal{\lambda}{ u_i X_i(\pi(\lambda))} - \frac{\norm{u}^2}{2}\Bigg)
\end{equation}
The \sF Hamiltonian can be explicitly characterized in terms of the dual norm $\normdot_*$, making use of Lemma \ref{lem:banachduality}. 

\begin{lemma}
\label{lem:maximized_hamiltonian}
    Let $M$ be a smooth \sF manifold. Given $\lambda \in T^*M$, define $\hat\lambda=(\hat\lambda_i)_{i=1,\ldots,k}$ where $\hat\lambda_i:=\scal{\lambda}{X_i(\pi(\lambda))}$ for every $i=1,\ldots,k$. Then, $H^{-1}(0)=\ann(\dis)$ and 
    \begin{equation}
        H(\lambda)=\frac12\big\|\hat\lambda\big\|^2_*\,,\qquad\forall\lambda \in T^*M\setminus\ann(\dis).
    \end{equation}
    where $\normdot_*$ is the dual norm to $\normdot$ in $\R^k$. Moreover, $H\in C^\infty(T^*M\setminus\ann(\dis))\cap C^1(T^*M)$. 
\end{lemma}

\begin{proof}
    Let $q\in M$. Assume that $\lambda\in T_q^*M\setminus\ann(\dis)_q$ and set
    \begin{equation}
        F(u):= \scal{\hat\lambda}{u} - \frac{\norm{u}^2}{2},\qquad\forall u\in\R^k.
    \end{equation}
    Since $\normdot$ is smooth, its square is a $C^1$-function, thus $F\in C^1(\R^k)$. Moreover, by homogeneity of the norm, $F(u)\to-\infty$ as $\norm{u}\to\infty$, hence $F$ admits a maximum. We compute its differential:
    \begin{equation}
    \label{eq:critical_point_F}
        d_uF=\hat\lambda - \norm{u}\cdot d_u\normdot= \hat\lambda - u^*,
    \end{equation}
    according to Lemma \ref{lem:banachduality}. Therefore, $F$ has a unique critical point (which is also the unique point of maximum) given by $u=u^{**}=\hat\lambda^*$. Finally, using also \eqref{eq:dualvector}, this implies that
    \begin{equation}
        H(\lambda)=\max_{u\in\R^k} F(u) = F(\hat\lambda^*)=\scal{\hat\lambda}{\hat\lambda^*}-\frac{1}{2}\big\|\hat\lambda\big\|_*^2=\frac{1}{2}\big\|\hat\lambda\big\|_*^2.
    \end{equation}
    To conclude, observe that if $\lambda\in\ann(\dis)_q$, then $\hat\lambda=0$ and 
    \begin{equation}
        H(\lambda)=\max_{u\in \R^k}\left(-\frac{\norm{u}^2}{2}\right)=0.
    \end{equation}
    Conversely, if $H(\lambda)=0$ we must have $\lambda\in\ann(\dis)$. Indeed, if this is not the case, $\hat\lambda\neq 0$ and hence $\big\|\hat\lambda\big\|_*\neq0$, giving a contradiction. This proves that $H^{-1}(0)=\ann(\dis)$. 
    
    Finally, we prove the regularity of $H$. Note that $\normdot_*$ is a smooth norm itself. Indeed, as $\normdot$ is smooth and strongly convex, the dual map of Lemma \ref{lem:banachduality}, which is
    \begin{equation}
        v^*=\norm{v} d_v\normdot=\frac12 d_v\big(\normdot^2\big)=:N(v),
    \end{equation}
    is smooth on $\{v\neq 0\}$, invertible and with invertible differential on $\{v\neq 0\}$. Thus, by the inverse function theorem, $N^{-1}\in C^\infty(\R^k\setminus \{0\})$. But now the dual norm satisfies \eqref{eq:dualvector}, hence
    \begin{equation}
        \|\hat\lambda\|_*=\|N^{-1}(\hat\lambda)\|,
    \end{equation}
    and the claim follows. Therefore, we deduce that $H\in  C^\infty(T^*M\setminus H^{-1}(0))\cap C^1(T^*M) = C^\infty(T^*M\setminus\ann(\dis))\cap C^1(T^*M)$.
\end{proof}

\begin{corollary}
\label{cor:control_normal_extremal}
Let $(\lambda_t)_{t\in[0,1]}$ be a normal extremal for the problem \eqref{eq:Hamiltonian_system}, then the associated control is given by
\begin{equation}
\label{eq:control_normal_extremal}
        \bar u(t)=\hat\lambda_t^*, \qquad\text{for a.e. }t\in [0,1].
    \end{equation}
\end{corollary}

\begin{proof}
This is again a consequence of the maximality condition in \eqref{eq:Hamiltonian_system}, together with the characterization of the \sF Hamiltonian. In particular, we must have \begin{equation}
\label{eq:critical_point}
    h^{-1}_{\bar u(t)}(\lambda_t)=\max_{u\in\R^k} h^{-1}_u(\lambda_t)=H(\lambda_t), \qquad\forall\,t\in [0,1],
\end{equation}
From this and \eqref{eq:critical_point_F}, we deduce that the control associated with $(\lambda_t)_{t\in[0,1]}$ satisfies the identity \eqref{eq:control_normal_extremal}. 
\end{proof}

The next result relates the system \eqref{eq:Hamiltonian_system} with the Hamiltonian system associated with the \sF Hamiltonian \eqref{eq:sf_hamiltonian}. A similar statement can be found in \cite[Prop. 12.3]{AS-GeometricControl}. The main difference is the regularity of the Hamiltonian function, which in the classical statement is assumed to be smooth outside the zero section.

\begin{prop}
Let $M$ be a smooth \sF manifold. Let $H\in C^\infty(T^* M \setminus \ann(\dis))\cap C^1(T^*M)$ be the \sF Hamiltonian defined in \eqref{eq:sf_hamiltonian}. If $(\lambda_t)_{t\in[0,1]}$ is a normal extremal, and $\lambda_0\in\ann(\dis)$, then $\lambda_t\equiv\lambda_0$. If $\lambda_0\in T^* M \setminus \ann(\dis)$, then 
\begin{equation}\label{eq:maximizedsystem}
    \dot \lambda_t = \vec H(\lambda_t).
\end{equation}
Conversely, if $(\lambda_t)_{t\in[0,1]}$ is a solution of \eqref{eq:maximizedsystem} with initial condition $\lambda_0\in T^* M \setminus \ann(\dis)$, then there exists $\bar u\in \Ltwo$ such that $(\lambda_t)_{t\in[0,1]}$ is a normal extremal with control $\bar u$ (i.e. the pair $(-1,(\lambda_t))$ is a solution of \eqref{eq:Hamiltonian_system}).
\end{prop}

\begin{proof}
If $(\lambda_t)_{t\in [0,1]}$ is a normal extremal, there exists an optimal control $\bar u$ such that the pair $(-1,(\lambda_t))$ is a solution to \eqref{eq:Hamiltonian_system}. Now, if the initial covector $\lambda_0\in\ann(\dis)$, then $\lambda_t\in\ann(\dis)$ for all $t\in [0,1]$, as the \sF Hamiltonian is constant along the motion and $H(\lambda_0)=0$. Using Corollary \ref{cor:control_normal_extremal}, this implies that the control $\bar u\equiv 0$ and that $\lambda_t\equiv \lambda_0$ as claimed. If $\lambda_0\in T^*M\setminus\ann(\dis)$, we follow the blueprint of \cite[Prop. 12.3]{AS-GeometricControl}. By the definition of \sF Hamiltonian \eqref{eq:sf_hamiltonian}, we have 
\begin{equation}
    H(\lambda)\geq h_{\bar u(t)}(\lambda),\qquad\forall\,\lambda\in T^*M,\, t\in [0,1],
\end{equation}
with equality along the dynamic $t\mapsto \lambda_t$. This means that the function $T^*M\ni \lambda\mapsto H(\lambda)-h_{\bar u(t)}(\lambda)$ has a maximum at $\lambda_t$. Therefore, using that $H\in C^1(T^*M)$, we deduce that
\begin{equation}
    d_{\lambda_t}H=d_{\lambda_t}h_{\bar u(t)}, \qquad\forall\,t\in [0,1].
\end{equation}
Such an equality immediately implies that the Hamiltonian vector fields are equal along the dynamic, namely 
\begin{equation}
    \vec{H}(\lambda_t)=\vec{h}_{\bar u(t)}(\lambda_t),\qquad\forall\,t\in[0,1].
\end{equation}

For the converse implication, recall that the Hamiltonian is constant along the motion. So, if $(\lambda_t)_{t\in[0,1]}$ is a solution to \eqref{eq:maximizedsystem} with initial condition $\lambda_0\in T^*M\setminus \ann(\dis)$, then $H(\lambda_0)=H(\lambda_t)$ and $\lambda_t\in  T^*M\setminus \ann(\dis)$, for every $t\in [0,1]$. 
Since $H$ is smooth outside the annihilator bundle of $\dis$, we deduce that $(\lambda_t)_{t\in [0,1]}$ is uniquely determined by $\lambda_0$ and, repeating verbatim the argument of \cite[Prop. 12.3]{AS-GeometricControl}, we conclude the proof.
\end{proof}




Fix $t\in\R$ and consider the \emph{(reduced) flow of} $\vec H$ on $T^*M\setminus\ann(\dis)$: 
\begin{equation}
        e_{\mathrm{r}}^{t\vec H}:\mathscr{A}_t\to T^*M\setminus\ann(\dis),
    \end{equation}
where $\mathscr{A}_t\subset T^*M\setminus\ann(\dis)$ is the set of covectors such that the associated maximal solution $(\lambda_s)_{s\in I}$, with $I\subset\R$ such that $0\in I$, is defined up to time $t$. Under the assumption of completeness of $(M,\di_{SF})$, $\vec H$ is complete as a vector field on $T^*M\setminus\ann(\dis)$ (and thus $\mathscr A_t=T^*M\setminus\ann(\dis)$). We state below this result without proof as the latter is analogous to the classical \sr proof, cf. \cite[Prop.\ 8.38]{ABB-srgeom}, in view of Lemma \ref{lem:banachduality} and Lemma \ref{lem:maximized_hamiltonian}.

\begin{prop}
\label{prop:maximal_extension_normal_geod}
Let $M$ be a smooth complete \sF manifold. Then any normal extremal $t\mapsto\lambda_t=e^{t\vec H}_{\mathrm r}(\lambda_0)$, with $\lambda_0\in T^*M\setminus\ann(\dis)$, is extendable to $\R$.   
\end{prop}

\begin{definition}[Sub-Finsler exponential map]
    Let $(M,\di_{SF})$ be a complete smooth \sF manifold and let $q\in M$. Then, the \emph{\sF exponential map} at $q$ is defined as
    \begin{equation}
        \exp_q(\lambda):=
        \begin{cases}
            \pi\circ e^{\vec H}_{\rm r}(\lambda) & \text{if }\lambda\in T^*_qM\setminus \ann(\dis)_q,\\
            q & \text{if }\lambda\in\ann(\dis)_q.
        \end{cases}
    \end{equation}
\end{definition}

\begin{remark}
\label{rmk:exponential_map_not_reg}
    The exponential map is smooth in $T_q^*M\setminus\ann(\dis)_q$ and, by homogeneity, we also have $\exp_q\in C^1(T^*_qM)$. However, note that the we do not have spatial regularity. More precisely, setting $\mathcal{E}: T^*M\to M$ to be $\mathcal{E}(q,\lambda)= \exp_q(\lambda)$, then
    \begin{equation}
        \mathcal{E}\in C(T^*M)\cap C^\infty(T^*M\setminus \ann(\dis)),
    \end{equation}
    and we can not expect a better spatial regularity as the vector field $\vec H$ is only continuous on the annihilator bundle. 
\end{remark}

\subsection{The end-point map}

Let $M$ be a \sF manifold, consider $u\in \Ltwo$ and fix $t_0\in [0,1]$. Define the non-autonomous vector field
\begin{equation}
    \xi_{u(t)}(q):=\xi(q,u(t))=\sum_{i=1}^k u_i(t) X_i(q),\qquad\forall\, q\in M, t\in[0,1],
\end{equation}
and denote by $P^u_{t_0,t}: M\to M$ its flow. This means that, for $q_0 \in M$, the curve $t\mapsto P^u_{t_0,t}(q_0)$ is the unique maximal solution to the Cauchy problem 
\begin{equation*}
\begin{cases}
    \displaystyle\dot{\gamma}(t)=\sum_{i=1}^{m} u_{i}(t) X_{i}(\gamma(t)),\\
    \gamma(t_0)=q_{0}.
\end{cases}
\end{equation*}
Moreover, we denote by $\gamma_u: I\to M$, the trajectory starting at $q_0$ and corresponding to $u$, namely $\gamma_u(t):=P_{0,t}^u(q_0)$, for every $t\in I$, which is the maximal interval of definition of $\gamma_u$.

\begin{definition}[End-point map]
    Let $M$ be a \sF manifold and let $q_0\in M$. We call $\mathcal{U}_{q_0}\subset \Ltwo$ the open set of controls for which the corresponding trajectory $\gamma_u$ is defined on the interval $[0,1]$. We define the \emph{end-point map} based at $q_0$ as
\begin{equation*}
    E_{q_0} : \mathcal{U}_{q_0} \to M, \qquad E_{q_0}(u)= \gamma_u(1).
\end{equation*} 
\end{definition}

\begin{lemma}[Differential of the end-point map, {\cite[Prop. 8.5]{ABB-srgeom}}]\label{lem:diffendpoint}
Let $M$ be a \sF manifold. The end-point map is smooth on $\mathcal{U}_{q_0}$. Moreover, for every $u\in \mathcal{U}_{q_0}$ the differential $d_{u} E_{q_{0}}: \Ltwo \rightarrow T_{E_{q_0}(u)} M$ has the expression 
\begin{equation*}
    d_{u} E_{q_{0}}(v)=\int_{0}^{1}\left(P_{t, 1}^{u}\right)_{*} \xi_{v(t)}(E_{q_0}(u)) \de t, \qquad \text{for every }v \in \Ltwo.
\end{equation*}
\end{lemma}

Using the explicit expression for the differential of the end-point map we deduce the following characterization for normal and abnormal extremals in \sF geometry, see \cite[Prop. \ 8.9]{ABB-srgeom} for the analogous \sr result.

\begin{prop}
\label{prop:endpoint_map_extremals}
Let $M$ be a smooth \sF manifold and let $(\gamma,u)$ be a non-trivial solution to \eqref{eq:P'timalcontrolproblem}. Then, there exists $\lambda_1\in T_{q_1}^*M$, where $q_1= E_{q_0}(u)$, such that the curve $(\lambda_t)_{t\in[0,1]}$, with
\begin{equation}
\label{eq:extremal_hamiltonian}
    \lambda_t:=(P_{t,1}^u)^* \lambda_1\in T^*_{\gamma(t)}M, \qquad\forall\,t\in[0,1],
\end{equation}
is a solution to \eqref{eq:Hamiltonian_system}. Moreover, one of the following conditions is satisfied:
\begin{itemize}
    \item[(i)]$(\lambda_t)_{t\in[0,1]}$ is a normal extremal if and only if $u$ satisfies
\begin{equation}\label{eq:normalchar}
    \scal{\lambda_1}{d_u E_{q_0}(v)} = \scal{u^*}{v} \qquad \text{for every }v \in \Ltwo;
\end{equation}
\item[(ii)]$(\lambda_t)_{t\in[0,1]}$ is an abnormal extremal if and only if $u$ satisfies
\begin{equation}\label{eq:abnormalchar}
    \scal{\lambda_1}{d_u E_{q_0}(v)} = 0 \qquad \text{for every }v \in \Ltwo.
\end{equation}
\end{itemize}
\end{prop}

\begin{proof}
Firstly, \eqref{eq:extremal_hamiltonian} is a well-known consequence of the Pontryagin maximum principle. We prove {\slshape(i)}, the proof of {\slshape(ii)} is analogous. For every $v \in \Ltwo$, using Lemma \ref{lem:diffendpoint} we deduce that
\begin{equation}\label{eq:thecomputation}
     \begin{split}
        \scal{\lambda_1}{d_u E_{q_0}(v)} &= \int_0^1 \scal{\lambda_1}{(P^u_{t,1})_* \xi_{v(t)}(E_{q_0}(u))} \de t =\int_0^1 \scal{(P^u_{t,1})^*\lambda_1}{ \xi_{v(t)}(\gamma (t))} \de t \\
        & = \int_0^1 \scal{\lambda_t}{ \xi_{v(t)}(\gamma (t))} \de t = \int_0^1 \sum_{i=1}^k \scal{\lambda_t}{ \xi_{i}(\gamma (t))} v_i(t)\de t.
    \end{split}
\end{equation}
Assume $(\lambda_t)_{t\in [0,1]}$ is a normal extremal, then, by Corollary \ref{cor:control_normal_extremal}, the associated optimal control $u$ satisfies $\scal{\lambda_t}{\xi_i(\gamma(t))}= u^*_i(t)$ for $i=1,\dots,k$.
Therefore, using \eqref{eq:thecomputation} we deduce that for every $v\in \Ltwo$ it holds

\begin{equation*}
    \scal{\lambda_1}{ d_{u} E_{q_{0}}(v)} = \int_0^1 \sum_{i=1}^k \scal{\lambda_t}{ \xi_{i}(\gamma (t))} v_i(t)\de t = \int_0^1 \scal{u^*(t)}{v(t)}\de t = \scal{u^*}{v}.
\end{equation*}
This proves \eqref{eq:normalchar}. 

Conversely, assume that the control $u$ satisfies \eqref{eq:normalchar}. We are going to prove that $(\lambda_t)_{t\in[0,1]}$ is a normal extremal. Using \eqref{eq:normalchar} and \eqref{eq:thecomputation} we deduce that for every $v\in \Ltwo$  
    \begin{equation*}
        \scal{u^*}{v}_{L^2} = \scal{\lambda_1}{ d_{u} E_{q_{0}}(v)} = \int_0^1 \sum_{i=1}^k \scal{\lambda_t}{ \xi_{i}(\gamma (t))} v_i(t)\de t,
    \end{equation*}
    As a consequence, since $v$ is arbitrary, we conclude that $\scal{\lambda_t}{\xi_i(\gamma(t))}= u^*_i(t)$ for $i=1,\dots,k$.
\end{proof}

\section{Failure of the \texorpdfstring{$\cd(K,N)$}{CD(K,N)} condition in smooth \sF manifolds}\label{sec:smoothSF}

In this section, we prove our main result regarding smooth \sF manifolds, cf. Theorem \ref{thm:intro1}. One of the crucial ingredients for the proof is the construction of a geodesic, enjoying good regularity properties, cf. Theorem \ref{thm:existence_good_geod}. 

\subsection{Construction of a geodesic without abnormal sub-segments}

This section is devoted to the construction of a geodesic without abnormal sub-segments, in smooth \sF manifolds. The main idea is to choose a short segment of a  normal geodesic that minimizes the distance from a hypersurface without characteristic points. We recall the definition of strongly normal geodesic and of geodesic without abnormal sub-segments.

\begin{definition}
\label{def:strongly_normal_geodesic}
    Let $M$ be a \sF manifold and let $\gamma : [0, 1] \to M$ be a normal geodesic. Then, we say that $\gamma$ is
    \begin{enumerate}
        \item[(i)] \emph{left strongly normal}, if for all $s \in [0, 1]$, the restriction $\gamma|_{[0,s]}$ is not abnormal;
        \item[(ii)] \emph{right strongly normal}, if for all $s \in [0, 1]$, the restriction $\gamma|_{[s,1]}$ is not abnormal; 
        \item[(iii)] \emph{strongly normal}, if $\gamma$ is left and right strongly normal. 
    \end{enumerate}
    Finally, we say that $\gamma$ does not admit abnormal sub-segments if any restriction of $\gamma$ is strongly normal.
\end{definition}
Let $\Sigma\subset M$ be a hypersurface and let $\gamma\colon[0,T]\rightarrow M$ be a horizontal curve, parameterized with constant speed, such that $\gamma(0)\in \Sigma$, $\gamma(T) = p \in M\setminus \Sigma$. Assume $\gamma$ is a minimizer for $\di_{SF}(\cdot,\Sigma)$, that is $\ell(\gamma)=\di_{SF}(p,\Sigma)$. Then, $\gamma$ is a geodesic and any corresponding normal or abnormal lift, say $\lambda :[0,T]\to T^*M$, must satisfy the transversality conditions, cf.\ \cite[Thm 12.13]{AS-GeometricControl},
\begin{equation}
\label{eq:trcondition}
\langle \lambda_0, w\rangle=0,\qquad \forall \,w\in T_{\gamma(0)} \Sigma.
\end{equation}
Equivalently, the initial covector $\lambda_0$ must belong to the annihilator bundle $\ann(\Sigma)$ of $\Sigma$ with fiber $\ann(\Sigma)_q= \{\lambda \in T_q^*M \mid \langle \lambda, T_q \Sigma\rangle = 0\}$, for any $q\in\Sigma$. 

\begin{remark}
    In the \sr setting, the normal exponential map $E$, defined as the restriction of the exponential map to the annihilator bundle of $\Sigma$, allows to build (locally) a smooth tubular neighborhood around non-characteristic points, cf. \cite[Prop. 3.1]{FPR-sing-lapl}. This may fail in the \sF setting as $E$ is not regular at $\Sigma$, cf. Remark \ref{rmk:exponential_map_not_reg}. Nonetheless, we are able to deduce a weaker result that is enough for our construction, see Theorem \ref{thm:E_diffeo_delta_smooth}. 
\end{remark}
Recall that $q\in\Sigma$ is a \emph{characteristic point}, and we write $q\in C(\Sigma)$, if $\dis_q\subset T_q\Sigma$. As it happens in the \sr case, also in the \sF setting, minimizers of $\di_{SF}(\cdot,\Sigma)$ whose initial point is a non-characteristic point, can not be abnormal geodesics. 

\begin{lemma}
\label{lem:abn_vs_char_pts}
    Let $M$ be a smooth \sF manifold. Let $p\in M\setminus \Sigma$ and let $\gamma:[0,1]\to M$ be a horizontal curve such that
    \begin{equation}
        \gamma(0)\in\Sigma,\quad \gamma(1)=p\quad\text{and}\quad \ell(\gamma)=\di_{SF}(p,\Sigma).
    \end{equation}
    Then, $\gamma(0)\in C(\Sigma)$ if and only if $\gamma$ is an abnormal geodesic. 
\end{lemma}

\begin{proof}
    The proof is a straightforward adaptation of the analogous result in the \sr setting. We sketch here the argument for completeness. By the Pontryagin maximum principle, cf. Theorem \ref{thm:PMP}, there exists a lift $\lambda:[0,1]\to T^*M$ verifying the system \eqref{eq:Hamiltonian_system} with the additional condition \eqref{eq:trcondition}. Using the characterization of Lemma \ref{lem:abnormal_annihilator}, $\lambda$ is an abnormal lift if and only if $\lambda_0\in \ann(\dis)_{\gamma(0)}$. The latter, combined with the transversality condition concludes the proof.  
\end{proof}

From now on, we assume that $\Sigma$ is the boundary of an open set $\Omega\subset M$. As our results are local in nature, this assumption is not necessary, however it makes the presentation easier. Let $\Omega\subset M$ be a non-characteristic domain in $M$, so that $\partial\Omega$ is compact and without characteristic points. Then, there exists a never-vanishing smooth section of $\ann(\partial\Omega)$, i.e.\ a smooth map $\lambda^+:\partial\Omega\to \ann(\partial\Omega)$ such that
\begin{equation}
\label{eq:section_trivialization}
    \lambda^+(q)\in \ann_q(\partial\Omega)\qquad\text{and}\qquad 2H(\lambda^+)=1,
\end{equation}
which is uniquely determined, up to a sign. Define the \emph{normal exponential map} as the restriction of the \sF exponential map to the annihilator bundle, namely
\begin{equation}
    E:D\to M,\qquad E(q,\lambda)=\exp_q(\lambda),
\end{equation}
where $D\subset\ann(\partial\Omega)$ is the largest open sub-bundle where $E$ is defined. Furthermore, we define the \emph{distance function from} $\partial\Omega$ as 
\begin{equation}
    \delta: M\to [0,\infty),\qquad \delta(p):=\di_{SF}(p,\partial\Omega).
\end{equation} 


\begin{lemma}
\label{lem:normal_exponential_map}
    Let $M$ be a smooth \sF manifold. There exists $\epsilon>0$ such that on the sub-bundle 
    \begin{equation}
        D_\epsilon:=\{(q,\lambda)\in\ann(\partial\Omega): E(q,\lambda)\in\Omega\  \text{ and }\ 0<\sqrt{2H(\lambda)}<\epsilon\}   \subset D 
    \end{equation}
    the map $E|_{D_\epsilon}$ is injective and $E(D_\epsilon)=\{0<\delta<\epsilon\} \cap \Omega$. 
\end{lemma}

\begin{proof}
    Without loss of generality, we assume that $M$ is complete, so that $D=\ann(\partial\Omega)$. We may proceed by contradiction and assume that there does not exist a choice of $\epsilon>0$ so that $E|_{D_\epsilon}$ is injective. Hence, we can find sequences $\{(q_n,\lambda_n)\},\{(q_n',\lambda_n')\}\subset\ann(\partial\Omega)$ such that
    \begin{equation}
    \label{eq:properties_sequences}
        (q_n,\lambda_n)\neq (q_n',\lambda_n')\qquad E(q_n,\lambda_n)=E(q_n',\lambda_n'),\qquad\text{and}\qquad H(\lambda_n),\, H(\lambda_n')\to0.
    \end{equation}
    Note that, as $\partial\Omega$ has no characteristic points, the \sF Hamiltonian is a norm on the fibers of $\ann(\partial\Omega)$. Therefore, by compactness, $(q_n,\lambda_n)\to (q,0)$ and $(q_n',\lambda_n')\to(q',0)$, up to subsequences. Thus, recalling that $E$ is continuous on $D$, passing to the limit in \eqref{eq:properties_sequences}, we get that $E(q,0)=E(q',0)$, meaning that $q=q'$. As a consequence $\lambda_n,\lambda_n'\in\ann(\partial\Omega)_q$ so they are multiple of the section defined in \eqref{eq:section_trivialization}, namely
    \begin{equation}
        \lambda_n=t_n\lambda^+(q),\qquad \lambda_n'=t_n'\lambda^+(q),
    \end{equation}
    where $t_n,t_n'\to 0$ and their signs agree. Finally, recall that the length of the normal curve $[0,1]\ni t \mapsto E(q,t\lambda)$ is exactly $\sqrt{2H(\lambda)}$. This forces $t_n=t_n'$ which is a contradiction with \eqref{eq:properties_sequences}. We are left to prove the last part of the statement. Fix $\epsilon>0$ so that $E|_{D_\epsilon}$ is injective, then $E(D_\epsilon)\subset \{0<\delta<\epsilon\}\cap \Omega$. For the converse inclusion, pick $p\in \{0<\delta<\epsilon\}\cap\Omega$ and let $\gamma:[0,1]\to M$ be a geodesic joining $\gamma(0)\in\partial\Omega$ and $p=\gamma(1)$ such that $\ell(\gamma)=\delta(p)$. Then, $\gamma(0)$ is a non-characteristic point, therefore $\gamma$ is a normal geodesic, whose lift satisfies \eqref{eq:trcondition}, according to Lemma \ref{lem:abn_vs_char_pts}. Hence, there exists $0\neq\lambda\in\ann(\partial\Omega)_{\gamma(0)}$ such that $\gamma(t)=E(\gamma(0),t\lambda)$ and $\ell(\gamma)=\sqrt{2 H(\lambda)}<\epsilon$. Thus, $(q,\lambda)\in D_\epsilon$ concluding the proof. 
\end{proof}

We state here a useful lemma regarding the regularity of the distance function from a boundary. Recall that a function $f:M\to\R$ is said to be \emph{locally semiconcave} if, for every $p\in M$, there exist a coordinate chart $\varphi: U\subset M\to \R^n$, with $p\in U$, and a constant $C\in\R$ such that
\begin{equation}
\label{eq:local_semiconcavity}
    F:\R^n\to \R;\qquad F(x):=f\circ \varphi^{-1}(x) - C\frac{|x|^2}{2},
\end{equation}
is concave, where $|\cdot|$ denotes the Euclidean norm. 

\begin{lemma}
\label{lem:local_semiconcavity_delta}
    Let $M$ be a smooth \sF manifold. Let $\Omega\subset M$ be an open and bounded subset. Assume that $\partial\Omega$ is smooth and without characteristic points. Then, the distance function from $\partial\Omega$, $\delta$ is locally semiconcave in $\Omega$. 
\end{lemma}

\begin{proof}
    We do not report here a complete proof, since it follows the same arguments of \cite{MR3741391}, with the obvious modification for the \sF case. In particular, applying Lemma \ref{lem:abn_vs_char_pts}, we deduce there are no abnormal geodesic joining points of $\Omega$ to its boundary and realizing $\delta$. Thus, the proof of \cite[Thm. 3.2]{MR3741391} shows that $\delta$ is locally Lipschitz in coordinates, meaning that the function $\delta$ written in coordinates is Lipschitz with respect to the Euclidean distance. Then, using \cite[Thm. 4.1]{MR3741391} implication $(3)\Rightarrow (2)$, we conclude.
\end{proof}

Since $\delta$ is locally semiconcave, Alexandrov's theorem ensures that $\delta$ is differentiable two times $\Leb^n$-a.e. (in coordinates) and, letting $\mathcal U\subset\Omega$ be the set where $\delta$ is differentiable, the function $d\delta:\mathcal U\to T^*M$ is differentiable $\Leb^n$-a.e., cf. \cite[Thm. 6.4]{MR4294651} for the precise statement of Alexandrov's theorem. This observation, combined with Lemma \ref{lem:normal_geod_potential} below, gives us an alternative description of geodesics joining $\partial\Omega$ and differentiablity points of $\delta$ in $\{0<\delta<\epsilon\}\cap\Omega$.

\begin{lemma}
\label{lem:normal_geod_potential}
    Let $M$ be a smooth \sF manifold. Let $p,q\in M$ be distinct points and assume there is a function $\phi : M\to\R$ differentiable at $p$ and such that
    \begin{equation}
    \label{eq:ineq_distance_potential}
        \phi(p) = \frac12\di^2_{SF}(p, q)\qquad\text{and}\qquad \frac12\di^2_{SF}(z, q) \geq \phi(z),\quad\forall\,z \in M.
    \end{equation}
    Then, the geodesic joining $p$ and $q$ is unique, has a normal lift and is given by $\gamma : [0, 1] \to M$; $\gamma(t) = \exp_p(-td_p\phi)$. 
\end{lemma}

\begin{proof}
    This is a well-known result in \sr geometry, cf. \cite[Lem. 2.15]{MR3308395}. The same proof can be carried out without substantial modifications in the setting of \sF manifolds, in light of Proposition \ref{prop:endpoint_map_extremals}.
\end{proof}

\begin{corollary}
\label{cor:geod_boundary_potential}
    Let $M$ be a smooth \sF manifold and let $\Omega\subset M$ be an open and bounded subset. Assume that $\partial\Omega$ is smooth and without characteristic points. Let $p\in\{0<\delta<\epsilon\}\cap\Omega$ be a differentiability point of $\delta$. Then, the unique geodesic $\gamma:[0,1]\to M$ joining $p$ and $\partial\Omega$ and such that $\delta(p)=\ell(\gamma)$ is defined by $\gamma(t)=\exp_p\big(-\frac{t}{2}d_p\delta^2\big)$.
\end{corollary}

\begin{proof}
    Since $p\in\{0<\delta<\epsilon\}\cap\Omega$, from Lemma \ref{lem:normal_exponential_map} we know that the geodesic joining $p$ and $\partial\Omega$, and realizing $\delta$ is normal and unique. Let $q\in\partial\Omega$ be its endpoint and define 
    \begin{equation}
        \phi:M\to\R;\qquad \phi(z):=\frac12\delta^2(z).
    \end{equation}
    Note that $\phi$ is differentiable at the point $p$, $\phi(p)=\frac12\ell(\gamma)^2=\frac12\di_{SF}^2(p,q)$ and, since $q\in\partial\Omega$, it also satisfies the inequality in \eqref{eq:ineq_distance_potential}. Thus, we may apply Lemma \ref{lem:normal_geod_potential} and conclude the proof. 
    \end{proof}

Collecting all the previous results, we are in position to prove the following theorem concerning the regularity of the normal exponential map. 

\begin{theorem}
\label{thm:E_diffeo_delta_smooth}
    Let $M$ be a smooth \sF manifold. The restriction of the \sF normal exponential map to $D_\epsilon$, namely $E|_{D_\epsilon}:D_\epsilon\to \{0<\delta<\epsilon\}\cap \Omega$, defines a diffeomorphism on an open and dense subset $\mathcal O\subset D_\epsilon$. Moreover, $\delta$ is smooth on $E(\mathcal O)\subset\{0<\delta<\epsilon\}\cap \Omega$, which is open and with full-measure.
\end{theorem}

\begin{proof}
    We are going to show that $d_{(q,\lambda)}E$ is invertible for every $(q,\lambda)$ in a suitable subset of $D_\epsilon$. By Corollary \ref{cor:geod_boundary_potential}, letting $U\subset\{0<\delta<\epsilon\}\cap \Omega$ be the set where $\delta$ is twice-differentiable, the map 
    \begin{equation}
        \Phi: U\to \ann(\partial\Omega); \qquad \Phi(p)=e^{-\vec H}(d_p\delta) 
    \end{equation}
    is a right-inverse for the normal exponential map, namely $E\circ\Phi=\id_U$. Note that $\Phi(U)\subset \ann(\partial\Omega)$ by Corollary \ref{cor:geod_boundary_potential}, in combination with the transversality condition \eqref{eq:trcondition}. Moreover, recalling that the Hamiltonian is constant along the motion, we also have:
    \begin{equation}
        \sqrt{2H(\Phi(p))}=\ell(\gamma)=\delta(p)\in (0,\epsilon),
    \end{equation}
    so that $\Phi(U)\subset D_\epsilon$. But now by the choice of the set $U$, $\delta$ is twice-differentiable on this set and it has a Taylor expansion up to order $2$. Thus, expanding the identity $E\circ\Phi=\id_U$ at a point $p=E(q,\lambda)$, we deduce that $d_{(q,\lambda)}E$ must be invertible for every $(q,\lambda)\in \Phi(U)\subset D_\epsilon$, and thus $E$ is a local diffeomorphism around every point in $\Phi(U)$. Furthermore, observing that $\Phi(U)$ is dense in $D_\epsilon$, we see that $E$ is a local diffeomorphism everywhere on a open and dense subset $\mathcal{O}\subset D_\epsilon$, containing $\Phi(U)$. Hence, we conclude that $E|_{\mathcal{O}}$ is a diffeormorphism onto its image, being a local diffeomorphism that is also invertible, thanks to Lemma \ref{lem:normal_exponential_map}. Finally, in order to prove that $\delta$ is smooth on $E(\mathcal{O})\subset\{0<\delta<\epsilon\}\cap\Omega$, it is enough to observe that, by construction, 
    \begin{equation}
        \delta(E(q,\lambda))=\sqrt{2H(\lambda)},\qquad\forall\,(q,\lambda)\in D_\epsilon.
    \end{equation}
    On $D_\epsilon$, $H$ is smooth, hence we conclude that $\delta$ is smooth on $E(\mathcal{O})$. Now $U\subset E(\mathcal O)$ so that $E(\mathcal O)$ is open, dense and has full-measure in $\{0<\delta<\epsilon\}\cap \Omega$.
\end{proof}

An immediate consequence of the previous theorem is the existence of many geodesics that are strongly normal in the sense of Definition \ref{def:strongly_normal_geodesic}.

\begin{corollary}
\label{cor:existence_strongly_normal_geod}
    Let $M$ be a smooth \sF manifold and let $\Omega\subset M$ be an open and bounded subset. Assume that $\partial\Omega$ is smooth and without characteristic points. Let $p\in E(\mathcal{O})\subset\{0<\delta<\epsilon\}\cap \Omega$ and let $\gamma:[0,1]\to M$ be the unique geodesic joining $p$ and $\partial\Omega$ and realizing $\delta$. Then, $\gamma$ is strongly normal. 
\end{corollary}

\begin{proof}
    Let $q:=\gamma(0)\in\partial\Omega$ the endpoint of $\gamma$ on the boundary of $\Omega$. Then, since $q$ is non-characteristic point, Lemma \ref{lem:abn_vs_char_pts} ensures that $\gamma|_{[0,s]}$ can not have an abnormal lift. Hence, $\gamma$ is left strongly normal. In order to prove that $\gamma$ is also right strongly normal, we reason in a similar way but with $\{\delta=\delta(p)\}$ in place of $\partial\Omega$. Indeed, since $\delta$ is smooth on the open set $E(\mathcal{O})$ by Theorem \ref{thm:E_diffeo_delta_smooth} and $d_{\bar p}\delta$ is not vanishing for every $\bar p\in E(\mathcal{O})$ as a consequence of Corollary \ref{cor:geod_boundary_potential}, the set $\Sigma:=\{\delta=\delta(p)\}$ defines a smooth hypersurface in a neighborhood of the point $p$. In addition, $\delta_\Sigma(q):=\di_{SF}(q,\Sigma)=\delta(p)$ and $\gamma$ is the unique geodesic realizing $\delta_\Sigma$. Finally, applying once again Lemma \ref{lem:abn_vs_char_pts}, we also deduce that $p\notin C(\Sigma)$, so repeating the argument we did before, we conclude that $\gamma$ must be also right strongly normal. This concludes the proof. 
\end{proof}

\begin{remark}
\label{rmk:abnormal_sub_segments}
    Since $E(\mathcal O)$ has full measure in $\{0<\delta<\epsilon\}\cap \Omega$, we can find $(q,\lambda)\in\mathcal {O}$ such that, denoting by $\gamma:[0,1]\to M$ the corresponding geodesic minimizing $\delta$, we have that $\gamma(t)\in E(\mathcal{O})$ for $\Leb^1$-a.e. $t\in [0,1]$. This means that $\Leb^1$-almost every level set defines locally a hypersurface and, recalling that restrictions of abnormal geodesics are still abnormal, the proof of Corollary \ref{cor:existence_strongly_normal_geod} can be repeated to show that the curve $\gamma$ does not contain abnormal sub-segments.
\end{remark}

\begin{theorem}[Existence of strongly normal geodesics without abnormal sub-segments]
\label{thm:existence_good_geod}
    Let $M$ be a smooth \sF manifold. Then, there exists a strongly normal geodesic $\gamma:[0,1]\to M$, which does not contain abnormal sub-segments.
\end{theorem}

\begin{proof}
    Note that Theorem \ref{thm:E_diffeo_delta_smooth} was stated for a hypersurface that is the boundary of non-characteri\-stic domain $\Omega$. However, without substantial modifications, one can prove that an analogous result holds locally around a non-characteristic point of a given smooth hypersurface $\Sigma\subset M$. In particular, letting $q\in\Sigma\setminus C(\Sigma)$, there exists $V_q\subset\Sigma$ open neighborhood of $q$, and $\epsilon>0$ such that, denoting by 
    \begin{equation}
        \tilde D_\epsilon:=\{(\bar q,\lambda):\bar q\in V_q,\, 0<\sqrt{2H(\lambda)}<\epsilon\},
    \end{equation}
    the map $E|_{\tilde D_\epsilon}:\tilde D_\epsilon\to E(\tilde D_\epsilon)\subset \{0<\delta_\Sigma<\epsilon\} $ is a diffeomorphism on an open and dense subset $\mathcal{O}\subset \tilde D_\epsilon$ and $\delta_\Sigma$ is smooth on $E(\mathcal {O})$. Now, Corollary \ref{cor:existence_strongly_normal_geod} shows that there exists a point $p\in E(\mathcal O)$ such that  the unique geodesic $\gamma:[0,1]\to M$ minimizing $\delta$ is strongly normal and, also according to Remark \ref{rmk:abnormal_sub_segments}, it does not contain abnormal sub-segments. In order to conclude, we need to show that there exists a hypersurface $\Sigma$ with $\Sigma\setminus C(\Sigma)\neq \emptyset$. But this is a consequence of the H\"ormander condition, indeed if $\dis_q\subset\Sigma_q$ for every $q\in\Sigma$, then Frobenius' theorem would ensure $\dis$ be involutive and thus it would not be bracket-generating.
\end{proof}

\subsection{Regularity of the distance function}

We state below the definition of conjugate and cut loci in a \sF manifold, following the blueprint of the \sr setting, cf.\ \cite[Chap.\ 11]{ABB-srgeom} or \cite{MR3935035}. 

\begin{definition}[Conjugate point]
    Let $M$ be a smooth \sF manifold and let $\gamma : [0, 1] \to M$ be a normal geodesic with initial covector $\lambda\in T^*_pM$, that is $\gamma(t)=\exp_p(t\lambda)$. We say that $q = \exp_p(\bar t\lambda)$ is a \emph{conjugate point} to $p$ along $\gamma$ if $\bar t\lambda$ is a critical point for $\exp_p$. 
\end{definition}

\begin{definition}[Cut locus]
    Let $M$ be a smooth \sF manifold and let $p\in M$. We say that $q\in M$ is a \emph{smooth point} with respect to $p$, and write $q\in \Sigma_p$, if there exists a unique geodesic $\gamma:[0,1]\to M$ joining $p$ and $q$, which is not abnormal and such that $q$ is not conjugate along $p$. Define the \emph{cut locus} of $p\in M$ as $\cut(p):=M\setminus\Sigma_p$. Finally, the cut locus of $M$ is the set 
    \begin{equation}
        \cut(M):=\{(p,q)\in M\times M: q\in\cut(p)\}\subset M\times M.
    \end{equation}
\end{definition}

\begin{remark}
    In the \sr setting, according to \cite{MR2513150}, the set of smooth points with respect to $p$ is open and dense. However, it is an open question to understand whether its complement, that is the cut locus, is negligible.   
\end{remark}

Outside the cut locus of a \sF manifold, we can define the \emph{$t$-midpoint map}, for $t\in [0,1]$, as the map $\phi_t:M\times M \setminus \cut(M)\to M$ assigning to $(p,q)$ the $t$-midpoint of the (unique) geodesic $\gamma_{p,q}$ joining $p$ and $q$. More precisely, for every $(p,q)\in M\times M \setminus \cut(M)$,
\begin{equation}
\label{eq:t-midpoint_map}
    \phi_t(p,q):=e_t(\gamma_{p,q})=\exp_q((t-1)\lambda_p), \qquad \text{where } \lambda_p \in T^*_q M\text{ such that } p=\exp_q(-\lambda_p).
\end{equation}
Note that, by definition of cut locus, the $t$-midpoint map is well-defined since the geodesic joining $p$ and $q$ for $q\notin \cut(p)$ is unique and strictly normal, i.e.\ without abnormal lifts. 

We report a useful result relating the regularity of the squared distance function on a \sF manifold $M$ with the cut locus. Such result can be proved repeating verbatim the proof of \cite[Prop.\ 11.4]{ABB-srgeom}, in light of Proposition \ref{prop:endpoint_map_extremals} and Lemma \ref{lem:normal_geod_potential}. For every $p\in M$, let $\mathfrak f_p:=\frac12\di_{SF}^2(\cdot,p)$.

\begin{prop}
\label{prop:smoothness_implies_smoothness}
    Let $M$ be a smooth \sF manifold and let $p,q\in M$. Assume there exists an open neighborhood $\mathcal O_q\subset M$ of $q$ such that $\mathfrak f_p$ is smooth. Then, $\mathcal O_q\subset \Sigma_p$ and 
    \begin{equation}
        \phi_t(p,z)=\exp_z((t-1)d_z\mathfrak f_p),\qquad \forall\, z\in\mathcal O_q.
    \end{equation}
\end{prop}

Thanks to Proposition \ref{prop:smoothness_implies_smoothness}, the regularity of the squared distance ensures uniqueness of geodesics and smoothness of the $t$-midpoint map. Thus, it is desirable to understand where the squared distance is smooth. In this regard, \cite[Thm.\ 2.19]{MR3852258} proves the regularity of the squared distance function along left strongly normal geodesics. We refer to \cite[App.\ A]{MR3852258} for further details. 

\begin{theorem}[{\cite[Thm.\ 2.19]{MR3852258}}]
\label{thm:regularity_distance_fun}
    Let $M$ be a smooth \sF manifold and let $\gamma : [0, 1] \to M$ be a left strongly normal geodesic. Then there exists $\epsilon > 0$ and an open neighborhood $U \subset M \times M$ such that:
    \begin{enumerate}
        \item[(i)] $(\gamma(0), \gamma(t)) \in U$ for all $t \in (0, \epsilon)$;
        \item[(ii)] For any $(p, q) \in U$ there exists a unique (normal) geodesic joining $p$ and $q$, shorter than $\epsilon$;
        \item[(iii)] The squared distance function $(p,q)\mapsto \di_{SF}^2(p, q)$ is smooth on $U$.
    \end{enumerate}
\end{theorem}

\noindent The regularity of the squared distance function can be ``propagated'' along geodesics that do not admit abnormal sub-segments, applying the previous theorem for every sub-segment. 

\begin{corollary}
\label{cor:regularity_distance_fun}
    Let $M$ be a smooth \sF manifold and let $\gamma:[0,1]\to M$ be a geodesic that does not admit abnormal sub-segments. Then, for every $s\in [0,1]$, there exists $\epsilon > 0$ and an open neighborhood $U \subset M \times M$ such that:
    \begin{enumerate}
        \item[(i)] $(\gamma(s), \gamma(t)) \in U$ for all $t \in [0,1]$ such that $0<|t-s|<\epsilon$;
        \item[(ii)] For any $(p, q) \in U$ there exists a unique (normal) geodesic joining $p$ and $q$, shorter than $\epsilon$;
        \item[(iii)] The squared distance function $(p,q)\mapsto \di_{SF}^2(p, q)$ is smooth on $U$.
    \end{enumerate}
\end{corollary}




\subsection{Volume contraction rate along geodesics}

Our goal is to quantify the contraction rate of small volumes along geodesics. To do this, we combine the smoothness of the $t$-midpoint map, with a lower bound on the so-called geodesic dimension. The latter has been introduced in \cite{MR3852258} for \sr manifolds and in \cite[Def.\ 5.47]{MR3502622} for general metric measure spaces. We recall below the definition. 

Let $M$ be a smooth \sF manifold. Given a point $p\in M$ and a Borel set $\Omega \subset M \setminus \cut(p)$, we define the \emph{geodesic homothety} of $\Omega$ with center $p$ and ratio $t \in [0, 1]$ as 
 \begin{equation}
    \Omega_t^p:=\{\phi_t(p,q): q\in \Omega\}.
\end{equation}
In the sequel, we say that $\m$ is a \emph{smooth measure} if, in coordinates, is absolutely continuous with respect the Lebesgue measure of the chart with a smooth and positive density. We will consider the metric measure space $(M,\di_{SF},\m)$.

\begin{definition}
\label{def:geodesic_dimension}
     Let $M$ be a smooth \sF manifold, equipped with a smooth measure $\m$. For any $p\in M$ and $s > 0$, define
    \begin{equation}
    \label{eq:pre_geod_dim}
         C_s(p) := \sup \left\{\limsup_{t\to 0}\frac{1}{t^s} \frac{\m(\Omega_t^p)}{\m(\Omega)}: \Omega \subset M \setminus \cut(p) \text{ Borel, bounded and  }\m(\Omega)\in(0,+\infty)\right\},
    \end{equation}
    We define the \emph{geodesic dimension} of $(M, \di_{SF}, \m)$ at $p\in M$ as the non-negative real number
    \begin{equation}
         \mathcal N(p) := \inf\{s > 0 : C_s(p) = +\infty\} = \sup\{s > 0 : C_s(p) = 0\},
     \end{equation}
     with the conventions $\inf \emptyset = +\infty$ and $\sup \emptyset = 0$.
 \end{definition} 

 \begin{remark}
\label{rmk:well-posedness_geod_dim}
    In \cite{MR3502622}, the definition of geodesic dimension is given for metric measure spaces with \emph{negligible cut loci}. While, in \cite[Def.\ 4.17]{barilari2022unified}, this assumption is dropped and the definition of geodesic dimension is given by taking the supremum \eqref{eq:pre_geod_dim} over Borel sets. In our work, we modified the definition by considering the supremum \eqref{eq:pre_geod_dim} over sets $\Omega$ outside the cut locus $\cut(p)$. This may not be equivalent to neither of the previous definitions but it is tailored to our specific purposes. 
 \end{remark}

 We now prove a fundamental theorem which relates the geodesic and topological dimensions of a sub-Finsler manifold $M$. This result is a suitable adaptation of \cite[Thm.\ 4]{MR3502622} to our setting.

 \begin{prop}
 \label{prop:lower_bound_geod_dim}
    Let $M$ be a smooth \sF manifold, equipped with a smooth measure $\m$. Assume that $r(p)<n:=\dim M$ for every $p\in M$.
    Then, 
    \begin{equation}
        \mathcal N (p) \geq n+1,\qquad\forall\,p\in M.
    \end{equation}
 \end{prop}

\begin{proof}
    Let $\di_{SR}$ be a \sr distance on the manifold $M$, equivalent to $\di_{SF}$ (see \eqref{eq:equivalent_sr_distance}). The Ball-Box theorem, cf.\ \cite[Cor.\ 2.1]{MR3308372}, ensures that for every $p \in M$ there exist $n_p\geq n+1$ and a positive constant $C_p$ such that 
    \begin{equation}\label{eq:Max,box}
        \m\big(B_r^{SR}(p)\big) \leq C_p \cdot r^{n_p} \qquad \text{ for $r$ sufficiently small.}
    \end{equation}
    Since $\di_{SF}$ and $\di_{SR}$ are equivalent, up to changing the constant, the same estimate holds for \sF balls, in particular
    \begin{equation}\label{eq:kdensity}
        \limsup_{r \to 0} \frac{\m\big(B_r^{SF}(p)\big)}{r^k}=0
    \end{equation}
    for every $k<n+1$. Take any $\Omega \subset M \setminus \cut(p)$ Borel, bounded and with $\m(\Omega)\in (0,+\infty)$ and consider $R>0$ such that $\Omega \subset B_R^{SF}(p)$. Note that $\Omega_t^p \subset B_{tR}^{SF}(p)$ and thus for every $k<n+1$ we have that
    \begin{equation}
        \limsup_{t \to 0} \frac{\m(\Omega_t^p)}{t^k \m (\Omega)} \leq  \limsup_{t \to 0} \frac{\m\big(B_{tR}^{SF}(p)\big)}{t^{k} \m (\Omega)} = \limsup_{t \to 0} \frac{\m\big(B_{tR}^{SF}(p)\big)}{(tR)^k} \cdot \frac{R^k}{\m(\Omega)}=0,
    \end{equation}
    where we used \eqref{eq:kdensity} for the last equality. Since $\Omega$ was arbitrary, we deduce that $C_k(p)=0$ for every $k<n+1$ and then $\mathcal{N}(p)\geq n+1$.
\end{proof}

\begin{remark}
    For an equiregular \sF manifold, with the same proof, it is possible to improve the estimate of Proposition \ref{prop:lower_bound_geod_dim}. In fact, in this case the Ball-Box theorem provides the estimate \eqref{eq:Max,box} with $n_p$ equal to the Hausdorff dimension $\dim_{H}(M)$, for every $p$, and consequently $\mathcal{N}(p)\geq \dim_{H}(M)$, cf. \cite[Prop.\ 5.49]{MR3852258}.
\end{remark}

By construction, the geodesic dimension controls the contraction rate of volumes along geo\-desics. This information can be transferred to the $t$-midpoint map, provided that is smooth. By invoking Theorem \ref{thm:regularity_distance_fun}, we can always guarantee the smoothness of the $t$-midpoint map for a sufficiently short segment of a geodesic without abnormal sub-segments.


\begin{theorem}
\label{thm:finite_order_jacobian}
    Let $M$ be a smooth \sF manifold equipped with a smooth measure $\m$ and such that $r(p)<n:=\dim M$ for every $p\in M$. Let $\gamma:[0,1]\to M$ be a geodesic that does not admit abnormal sub-segments, with endpoints $p$ and $q$. Assume that $(p,q)$ belongs to the open set $U$, found in Theorem \ref{thm:regularity_distance_fun}. Then, either $|\det\big(d_q\phi_t(p,\cdot)\big)|$ has infinite order at $t=0$ or
    \begin{equation}
    \label{eq:order_jacobian}
        |\det\big(d_q\phi_t(p,\cdot)\big)| \sim t^{m_p},\qquad\text{as }t\to 0
    \end{equation}
    for some integer $m_p\geq\mathcal N(p)\geq n+1$.
\end{theorem}

\begin{proof}
    Since, by assumption $(p,q)\in U$, we can apply item {\slshape (iii)} of Theorem \ref{thm:regularity_distance_fun}, deducing the regularity of the distance function. Combining this with Proposition \ref{prop:smoothness_implies_smoothness} and the homogeneity of the Hamiltonian flow, there exists an open neighborhood $V\subset M$ of $q$, such that the function
    \begin{equation}
        [0,1)\times V\ni(t,z) \mapsto d_z\phi_t(p,\cdot) = d_z\left(\exp_z( (t-1) d_z\mathfrak f_p)\right)
    \end{equation}
    is smooth. Thus, we can compute the Taylor expansion of its determinant in the $t$-variable at order $N:= \lceil \mathcal N (p)\rceil-1 < \mathcal{N}(p)$, obtaining:
    \begin{equation}
    \det\big(d_z\phi_t(p,\cdot)\big)= \sum_{i=0}^N a_i(z)t^i+t^{N+1}R_N(t,z), \qquad \forall\, z\in V,
    \end{equation}
    where the functions $a_i$ and $R_N$ are smooth. Arguing by contradiction, we assume that there exists $j\leq N$ such that $a_j(q)\ne 0$ and define
    \begin{equation}
        m:= \min\{ i\leq N \, :\, \exists \,z\in V \text{ such that } a_i(z)\neq 0\}.
    \end{equation}
    Note that $m\leq j$ since $a_j(q)\neq 0$ and thus $m\leq N$.
    Without loss of generality, we can assume that $V$ and $p$ are contained in the same coordinate chart and that $a_m>0$ on an open subset $\tilde V\subset V$ with positive measure. Then, in charts, it holds that 
    \begin{equation}
        \Leb^n \big(\tilde V_t^p\big)=\int_{\tilde V}\big|\det\big(d_z\phi_t(p,\cdot)\big)\big|\de z = \int_{\tilde V} a_m(z) \de z \cdot t^m + o(t^m) \qquad \text{as }t \to 0.
    \end{equation}
    Therefore, recalling that $\m$ is a smooth measure, there exists a constant $a>0$ such that 
    \begin{equation}
        \m \big(\tilde V_t^p\big) \geq a \cdot t^m,
    \end{equation}
    for every $t$ sufficiently small. As a consequence, taking any $s\in (N, \mathcal N(p))$ we have that 
    \begin{equation}
        \limsup_{t\to 0}\frac{1}{t^s} \frac{\m(\tilde V_t^p)}{\m(\tilde V)} \geq \limsup_{t\to 0}\frac{1}{\m(\tilde V)} \frac{a \cdot t^m}{t^s} = +\infty,
    \end{equation}
    and therefore we deduce $C_s(p)=+\infty$, which in turn implies $\mathcal N (p)\leq s$, giving a contradiction.
\end{proof}

Theorem \ref{thm:finite_order_jacobian} motivates the following definition. 

\begin{definition}[Ample geodesic]
\label{def:ample_geodesic}
    Let $M$ be a smooth \sF manifold and let $\gamma:[0,1]\to M$ be a strictly normal geodesic not admitting abnormal sub-segments. We say that $\gamma$ is \emph{ample} if, for every couple of distinct points $p,q\in \gamma([0,1])$, $|\det\big(d_q\phi_t(p,\cdot)\big)|$ exists and has finite order in $t=0$. 
\end{definition}

\begin{remark}
    The concept of ample geodesic in the \sr setting has been introduced in \cite{MR3852258} and it differs from Definition \ref{def:ample_geodesic}. However, we remark that, in \sr manifolds, for ample geodesics in the sense of \cite{MR3852258}, $|\det\big(d_q\phi_t(p,\cdot)\big)|$ has finite order equal to the geodesic dimension at $p$, cf.\ \cite[Lem.\ 6.27]{MR3852258}. Thus, our definition is weaker, but enough for our purposes. 
\end{remark}

\subsection{Proof of Theorem \ref{thm:intro1}}\label{sec:argument_of_the_mago}

Let $M$ be a smooth \sF manifold and let $\phi_t$ the $t$-midpoint map, defined as in \eqref{eq:t-midpoint_map}. For ease of notation, set 
\begin{equation}
\label{eq:midpoint}
    \M(p,q) := \phi_{1/2}(p,q),\qquad \forall\,(p,q)\in M\times M\setminus \cut(M),
\end{equation}
be the $1/2$-midpoint map or simply midpoint map. Reasoning as in \cite[Prop.\ 3.1]{MR4201410}, we obtain the following result as a consequence of Corollary \ref{cor:regularity_distance_fun} and Theorem \ref{thm:finite_order_jacobian}. This argument hinges upon Theorem \ref{thm:existence_good_geod}, which establishes the existence of a geodesic without abnormal sub-segments in a \sF manifold.  

\begin{prop}\label{prop:mago,isnegnaci}
    Let $M$ be a smooth \sF manifold equipped with a smooth measure $\m$ and such that $r(p)<n:=\dim M$ for every $p\in M$.
    Let $\gamma:[0,1]\to M$ be the geodesic identified in Theorem \ref{thm:existence_good_geod} and let $\varepsilon>0$. Then, there exist $0\leq a<b\leq 1$ such that, letting $\bar p:=\gamma(a)$, $\bar q:=\gamma(b)$, the following statements hold:
    \begin{enumerate}
        \item[(i)] $\bar p\notin \cut(\bar q)$, $\bar q\notin \cut(\bar p)$ and, for every $t\in (a,b)$, we have $\bar p,\bar q\notin \cut(\gamma(t))$. Moreover, for every $t\in (a,b)$, $\mathfrak f_{\gamma(t)}$ is smooth in a neighborhood of $\bar p$ and in a neighborhood $\bar q$.
        \item[(ii)] If, in addition, $\gamma$ is ample, the midpoint map satisfies        \begin{equation}\label{eq:jacobianest}
            |\det d_{\bar q}\M(\bar p,\cdot)|\leq (1+\varepsilon)2^{-m_{\bar p}},\qquad |\det d_{\bar p}\M(\cdot,\bar q)|\leq (1+\varepsilon)2^{-m_{\bar q}}     
        \end{equation}
        where $m_{\bar p}$ and $m_{\bar q}$ are defined by \eqref{eq:order_jacobian} and $m_{\bar p}, m_{\bar q}\geq n+1$.
    \end{enumerate}
\end{prop}

%

\noindent Given $z\in M$, define the \emph{inverse geodesic map} $\I_z: M \setminus \cut(z) \to M$ as 
\begin{equation}
\label{eq:inverse_geodesic}
    \I_z(p) = \exp_z(- \lambda) \qquad \text{where } \lambda \in T^*_z M\text{ such that } p=\exp_z( \lambda).
\end{equation}
We may interpret this map as the one associating to $p$ the point $\I_z(p)$ such that $z$ is the midpoint of $x$ and $\I_z(p)$. 

We prove now the main theorem of this section, which also implies Theorem \ref{thm:intro1}. Our strategy is an adaptation to the \sF setting of the one proposed in \cite{MR4201410}. 

\begin{theorem}\label{thm:casosmooth}
    Let $M$ be a complete smooth \sF manifold equipped with a smooth measure $\m$ and such that $r(p)<n:=\dim M$ for every $p\in M$. Then, the metric measure space $(M,\di_{SF},\m)$ does not satisfy the Brunn--Minkowski inequality $\bm(K,N)$, for every $K\in\R$ and $N\in (1,\infty)$.
\end{theorem}

\begin{proof}
Fix $\varepsilon>0$, $K\in \R$ and $N\in(1,\infty)$. Let $\gamma:[0,1]\to M$ be the geodesic identified by Theorem \ref{thm:existence_good_geod} and assume it is contained in a coordinate chart with (sub-Finsler) diameter $D>0$. Up to restricting the domain of the chart and the geodesic, we can also assume that 
\begin{equation}\label{eq:mincharts}
    (1-\varepsilon)\Leb^n \leq \m \leq (1+\varepsilon)\Leb^n \qquad\text{and} \qquad
    \tau_{K,N}^{(1/2)} (\theta) \geq  \frac 12 - \varepsilon, \quad \forall \theta \leq D,
\end{equation}
where the second inequality can be fulfilled, according to Remark \ref{rmk:limit_dist_coeff}.
Moreover, let $0\leq a < b \leq 1$ be as in Proposition \ref{prop:mago,isnegnaci}. We proceed by contradiction and assume that $(M,\di_{SF},\m)$ satisfies the $\bm(K,N)$. 

First of all, suppose that $\gamma$ is not ample. According to \cite[Prop.\ 5.3]{Magnabosco-Portinale-Rossi:2022b}, the Brunn--Minkowski inequality $\bm(K,N)$ implies the $\MCP(K,N)$ condition\footnote{In that paper, the proposition is proved for essentially non-branching metric measure spaces. However, what is needed is a measurable selection of geodesics which we have locally around the curve $\gamma$.}.
Therefore, $(M,\di_{SF},\m)$ satisfies the $\MCP(K,N)$ condition and, for the moment, assume $K=0$. Set $\bar p:=\gamma(a)$ and $\bar q:=\gamma(b)$ and let $\Omega_\varrho:=B_\varrho(\bar q)$ for $\varrho>0$. From the $\MCP(0,N)$ condition we get
    \begin{equation}
    \label{eq:MCP_condition_ineq}
        \m\big(\Omega_{\varrho,t}^{\bar p}\big)\geq t^N\m(\Omega_\varrho),\qquad \forall\, t\in[0,1],\,\varrho>0. 
    \end{equation}
   If $\varrho$ is sufficiently small, then $\Omega_{\varrho,t}^{\bar p}=\phi_t(\bar p,\Omega_\varrho)$ for $t\in [0,1)$, therefore, employing the first estimate in \eqref{eq:mincharts}, the inequality \eqref{eq:MCP_condition_ineq} can be reformulated as follows:
    \begin{equation}
       \label{eq:average_int_det}
       \frac{1+\varepsilon}{1-\varepsilon} \fint_{\Omega_\varrho}|\det\big(d_z\phi_t(\bar p,\cdot)\big)|\de z\geq\frac{\m(\Omega_{\varrho,t}^{\bar p})}{\m(\Omega_\varrho)}\geq  t^N,\qquad\forall\, t\in[0,1),\,\varrho>0.
    \end{equation}
    Taking the limit as $\varrho\to0$, and then the limit as $t\to 0$, we find that the order of $\big|\det\big(d_{\bar q}\phi_t(\bar p,\cdot)\big)\big|$ should be smaller than or equal to $N$, giving a contradiction. Finally, if $K\neq 0$, observe that the behavior of the distortion coefficients, as $t\to 0$, is comparable with $t$, namely there exists a constant $C=C(K,N,\di(\bar p,\bar q))>0$ such that
    \begin{equation}
        \tau_{K,N}^{(t)}(\theta)\geq C t,\qquad\text{as }t\to 0, \qquad \forall\, \theta\in (\di(\bar p,\bar q)-\varrho,\di(\bar p,\bar q)+\varrho).
    \end{equation}
    Therefore, repeating the same argument that we did for the case $K=0$, we obtain the sought contradiction. 

Suppose instead that the geodesic $\gamma$ is ample and let $m$ be the unique midpoint between $\bar p=\gamma(a)$ and $\bar q=\gamma(b)$. According to item {\slshape (i)} of Proposition \ref{prop:mago,isnegnaci}, the map $\I_m$ is well-defined and smooth in a neighborhood of $\bar p$ and $\bar q$, moreover by definition $\I_m(\bar q)=\bar p$ and $\I_m(\bar p)=\bar q$. Note that $\I_m \circ \I_m = \id$ (where defined), thus 
\begin{equation}
  |\det(d_{\bar p} \I_m)|   \cdot |\det(d_{\bar q} \I_m)| = \big|\det \big(d_{\bar q} (\I_m \circ \I_m)  \big)\big| =1.
\end{equation}
Therefore, at least one between $|\det(d_{\bar q} \I_m)|$ and $|\det(d_{\bar p} \I_m)|$ is greater than or equal to 1, without loss of generality we assume
\begin{equation}\label{eq:detI}
    |\det(d_{\bar q} \I_m)| \geq 1.
\end{equation}
Let $B_\varrho:=B^{eu}_\varrho(\bar q)$ the (Euclidean) ball of radius $\varrho>0$ centered in $\bar q$. Introduce the function $F: B_\varrho \times B_\varrho \to M$, defined as 
\begin{equation*}
    B_\varrho \times B_\varrho\ni (x,y) \mapsto F(x,y):=\M (\I_m(x),y).
\end{equation*}
Observe that, for $\varrho$ small enough, $F$ is well-defined and by construction $F(x,x)=m$ for every $x\in B_\varrho$. Therefore, we deduce that for every vector $v\in T_{\bar q}M\cong \R^n$, the following holds: 
\begin{equation}
    0 = d_{(\bar q,\bar q)} F  (v,v) = \big(d_{\bar p}\M (\cdot,\bar q) \circ d_{\bar q}\I_m\big) \, v + d_{\bar q} \M(\bar p,\cdot) \, v.
\end{equation}
Since the former identity is true for every vector $v\in\R^n$, we can conclude that
\begin{equation}
   d_{\bar p} \M(\cdot,\bar q)\circ  d_{\bar q}\I_m + d_{\bar q} \M(\bar p,\cdot)=0,
\end{equation}
and consequently, for every $v,w\in \R^n$, we have
\begin{equation}
    d_{(\bar q,\bar q)} F (v,w) = \big(d_{\bar p} \M (\cdot,\bar q)\circ  d_{\bar q} \I_m\big) \, v + d_{\bar q} \M (\bar p,\cdot) \, w = d_{\bar q} \M (\bar p,\cdot) \, (w-v).
\end{equation}
In particular, we obtain a Taylor expansion of the function $F$ at the point $(\bar q,\bar q)$ that in coordinates takes the form:
\begin{equation*}
    \norm{F(\bar q+v,\bar q+w) - m - d_{\bar q} \M (\bar p,\cdot) \, (w-v)}_{eu} = o (\norm{v}_{eu}+\norm{w}_{eu}),\qquad \text{as }v,w\to 0.
\end{equation*}
Then, as $v$ and $w$ vary in $B^{eu}_\varrho(0)$, $v-w$ varies in $B^{eu}_{2\varrho}(0)$, and we obtain that 
\begin{equation}\label{eq:themago}
    F(B_\varrho,B_\varrho) \subseteq m + d_{\bar q} \M (\bar p,\cdot) \, \big(B^{eu}_{2\varrho}(0)\big) + B^{eu}_{\omega(\varrho)}(0),
\end{equation}
where $\omega:\R_+\to \R_+$ is such that $\omega(r)=o(r)$ when $r\to 0^+$. Now, consider $A_\varrho:= \I_m(B_\varrho)$ and note that by definition $M_{1/2}(A_\varrho,B_\varrho) = F(B_\varrho,B_\varrho)$, then using \eqref{eq:themago} we conclude that, as $\varrho\to 0$, 
\begin{equation*}
\begin{split}
    \Leb^n&\big(M_{1/2}(A_\varrho,B_\varrho)\big)=\Leb^n\big(F(B_\varrho,B_\varrho)\big) \leq \Leb^n\Big( d_{\bar q} \M (p,\cdot) \, \big(B^{eu}_{2\varrho}(0)\big)\Big) + o (\varrho^n) \\
    &= \big|\det(d_{\bar q} \M (p,\cdot))\big| \cdot \omega_n 2^n\varrho^n + o(\varrho^n) \leq (1+\varepsilon)2^{n-m_q}   \omega_n \varrho^n + o(\varrho^n) \leq \frac 12 (1+\varepsilon) \omega_n \varrho^n + o(\varrho^n)
\end{split}
\end{equation*}
where $\omega_n=\Leb^n(B^{eu}_1(0))$ and the two last inequalities follow from \eqref{eq:jacobianest} and $m_{\bar q}\geq n+1$. On the other hand, it holds that $\Leb^n (B_\varrho)= \omega_n \varrho^n $ and, as $\varrho\to0$, 
\begin{equation}
    \Leb^n (A_\varrho)= \Leb^n( \I_m(B_\varrho)) = \big( |\det(d_{\bar q} \I_m)| + O (\varrho)\big)\, \Leb^n (B_\varrho)\geq \omega_n \varrho^n +o(\varrho^n).
\end{equation}
Taking into account the first estimate of \eqref{eq:mincharts}, we deduce the following inequalities for the measure $\m$, as $\varrho\to 0$, 
\begin{equation*}
\begin{gathered}
    \m\big(M_{1/2}(A_\varrho,B_\varrho)\big) \leq \frac 12 (1+\varepsilon)^2 \omega_n \varrho^n + o(\varrho^n), \\
    \m (A_\varrho) \geq (1-\varepsilon)\omega_n  \varrho^n + o(\varrho^n)\qquad\text{and}\qquad\m (B_\varrho) \geq (1-\varepsilon) \omega_n  \varrho^n.
\end{gathered}
\end{equation*}
Finally, if $\varepsilon$ is small enough we can find $\varrho$ sufficiently small such that 
\begin{equation}
\begin{split}
     \m\big(M_{1/2}(A_\varrho,B_\varrho)\big) ^\frac{1}{N} &< \bigg(\frac 12 - \varepsilon \bigg)\,\m(A_\varrho)^\frac{1}{N} + \bigg(\frac 12 - \varepsilon \bigg)\, \m(B_\varrho)^\frac{1}{N} \\
     &\leq  \tau_{K,N}^{(1/2)} \big(\Theta(A_\varrho,B_\varrho)\big) \,\m(A_\varrho)^\frac{1}{N} + \tau_{K,N}^{(1/2)} \big(\Theta(A_\varrho,B_\varrho)\big) \,\m(B_\varrho)^\frac{1}{N}, 
\end{split}
\end{equation}
which contradicts the Brunn--Minkowski inequality $\bm(K,N)$.
\end{proof}

\begin{remark}
    Observe that the argument presented in this section is local, around the geodesic without abnormal sub-segments. Thus, repeating the same proof, we can extend Theorem \ref{thm:casosmooth} if the assumption on the rank holds on an open set $V\subset M$, namely $r(p) <n$ for every $p\in V$. 
\end{remark}

\section{Failure of the \texorpdfstring{$\cd(K,N)$}{CD(K,N)} condition in the sub-Finsler Heisenberg group}\label{sec:Heisenberg}

In this section, we disprove the curvature-dimension condition in the \sF Heisenberg group, cf. Theorem \ref{thm:intro2}. Our strategy relies on the explicit expression of geodesics in terms of convex trigonometric functions, found in  \cite{Nestogol}. 

\subsection{Convex trigonometry}
\label{sec:convex_trigtrig}

In this section, we recall the definition and main properties of the convex trigonometric functions, firstly introduced in \cite{Nestoeurogol}. 
Let $\Omega\subset \R^2$ be a convex, compact set, such that $O:=(0,0)\in \text{Int} (\Omega)$ and denote by $\Sbb$ its surface area. 

\begin{definition}
    Let $\theta\in\R$ denote a generalized angle. If $0\leq \theta < 2 \Sbb$ define $P_\theta$ as the point on the boundary of $\Omega$, such that the area of the sector of $\Omega$ between the rays $Ox$ and $OP_{\theta}$ is $\frac{1}{2}\theta$ (see Figure \ref{fig:convextrig1}). Moreover, define $\sinom(\theta)$ and $\cosom(\theta)$ as the coordinates of the point $P_\theta$, i.e.
    \begin{equation*}
        P_\theta = \big( \sinom(\theta), \cosom(\theta) \big).
    \end{equation*}
    Finally, extend these trigonometric functions outside the interval $[0,2\Sbb)$ by periodicity (of period $2 \Sbb$), so that for every $k\in \mathbb Z$.
    \begin{equation*}
        \cosom(\theta)= \cosom(\theta+2k \Sbb), \quad \sinom(\theta)= \sinom(\theta+2k \Sbb) \quad \text{and}\quad P_\theta = P_{\theta +2k\Sbb}.
    \end{equation*}
\end{definition}

\noindent Observe that by definition $\sinom(0)=0$ and that when $\Omega$ is the Euclidean closed unit ball we recover the classical trigonometric functions.

\begin{figure}[h!]
    \begin{minipage}[c]{.47\textwidth}

    \centering
    \begin{tikzpicture}[scale=0.8]

    \draw[white](1.1,3.25)--(0.92142,4.75);
    \fill[color=black!10!white](3,0)--(0,0)--(1.1,3.25)--(3.6,4);
    \fill[white](3,0) .. controls (3.6,2.8) and (2.4,3.8)..(0,3)--(0,4)--(4,4);
    \draw[->] (-4,0)--(4,0);
    \draw[->] (0,-4)--(0,4);
    \draw[very thick] (3,0) .. controls (3.6,2.8) and (2.4,3.8)..(0,3)..controls (-1.2,2.6) and (-2,2.4).. (-2.6,0)..controls (-3.4,-3) and (-2,-3.4).. (0,-3.2)..controls (1.4,-3) and (2.4,-2.4).. (3,0);
    \draw[dotted,blue,thick] (1.1,3.25) --(0,3.25);
    \draw[very thick] (1.1,3.25) --(0,0);
    \draw[very thick] (3,0) --(0,0);
    \draw[very thick,blue](0,3.25)--(0,0);
    \draw[very thick,red](1.1,0)--(0,0);
    \draw[dotted,red,thick] (1.1,3.25) --(1.1,0);
    \filldraw[black] (0,0) circle (1.5pt);
    \filldraw[blue] (0,3.25) circle (1.5pt);
    \filldraw[red] (1.1,0) circle (1.5pt);
    \filldraw[black] (1.1,3.25) circle (1.5pt);

    \node at (-0.9,2)[label=south:${\color{blue}{\sin_\Omega(\theta)}}$] {};
    \node at (0.9,0)[label=south:${\color{red}{\cos_\Omega(\theta)}}$] {};
    \node at (-0.3,0.1)[label=south:$O$] {};
    \node at (2,2)[label=south:$\frac 12 \theta$] {};
    \node at (-2,-1.8)[label=south:$\Omega$] {};
    \node at (1.2,4.2)[label=south:$P_\theta$] {};

    \end{tikzpicture}
    \caption{Values of the generalized trigonometric functions $\cosom$ and $\sinom$.}
    \label{fig:convextrig1}
    
\end{minipage}%
\hfill
\begin{minipage}[h]{.47\textwidth}

    \centering
    \begin{tikzpicture}[scale=0.8]

    \fill[color=black!10!white](3,0)--(0,0)--(1.1,3.25)--(3.6,4);
    \fill[white](3,0) .. controls (3.6,2.8) and (2.4,3.8)..(0,3)--(0,4)--(4,4);
    \draw[->] (-4,0)--(4,0);
    \draw[->] (0,-4)--(0,4);
    \draw[very thick] (3,0) .. controls (3.6,2.8) and (2.4,3.8)..(0,3)..controls (-1.2,2.6) and (-2,2.4).. (-2.6,0)..controls (-3.4,-3) and (-2,-3.4).. (0,-3.2)..controls (1.4,-3) and (2.4,-2.4).. (3,0);
    \draw[very thick] (1.1,3.25) --(0,0);
    \draw[very thick] (3,0) --(0,0);
    \filldraw[black] (0,0) circle (1.5pt);
    \filldraw[black] (1.1,3.25) circle (1.5pt);
    \draw(1.1,3.25)--(-1,3);
    \draw(1.1,3.25)--(3.2,3.5);
    \draw[very thick, ->](1.1,3.25)--(0.92142,4.75);

    \node at (-0.3,0.1)[label=south:$O$] {};
    \node at (2,2)[label=south:$\frac 12 \theta$] {};
    \node at (-2,-1.8)[label=south:$\Omega$] {};
    \node at (1.5,3.3)[label=south:$P_\theta$] {};
    \node at (1.5,4.7)[label=south:$Q_{\psi}$] {};

    \end{tikzpicture}
   \caption{Representation of the correspondence $\theta\xleftrightarrow{\Omega} \psi$.}
    \label{fig:convextrig2}
    
\end{minipage}
\end{figure}

Consider now  the polar set:
\begin{equation*}
    \Omega^\circ := \{(p,q)\in \R^2\, :\, px+qy\leq 1 \text{ for every }(x,y)\in \Omega\},
\end{equation*}
which is itself a convex, compact set such that $O\in \text{Int}(\Omega^\circ)$. Therefore, we can consider the trigonometric functions $\sinomp$ and $\cosomp$. Observe that, by definition of polar set, it holds that 
\begin{equation}
\label{eq:NO_PYTHAGOREAN_IDENTITY_COMETIPERMETTI}
    \cosom(\theta) \cosomp(\psi) + \sinom(\theta)\sinomp(\psi)\leq 1,\qquad \text{for every } \theta,\psi\in \R.
\end{equation}

\begin{definition}
    We say that two angles $\theta,\psi\in \R$ \emph{correspond} to each other and write $\theta \xleftrightarrow{\Omega} \psi$ if the vector $Q_\psi:= (\cosomp(\psi),\sinomp(\psi))$ determines a half-plane containing $\Omega$ (see Figure \ref{fig:convextrig2}).
\end{definition}

By the bipolar theorem \cite[Thm.\ 14.5]{Rockafellar+1970}, it holds that $\Omega^{\circ \circ}=\Omega$ and this allow to prove the following symmetry property for the correspondence just defined.

\begin{prop}\label{prop:correspondence}
    Let $\Omega\subset \R^2$ be a convex and compact set, with $O\in \text{Int} (\Omega)$. Given two angles $\theta,\psi\in \R$, $\theta\xleftrightarrow{\Omega} \psi$ if and only if $\psi\xleftrightarrow{\Omega^\circ} \theta$.
    Moreover, the following analogous of the Pythagorean equality holds:
    \begin{equation}\label{eq:pytagorean}
        \theta\xleftrightarrow{\Omega} \psi \qquad \text{ if and only if }\qquad \cosom(\theta) \cosomp(\psi) + \sinom(\theta)\sinomp(\psi)= 1.
    \end{equation}
\end{prop}

\noindent The correspondence $\theta\xleftrightarrow{\Omega} \psi$ is not one-to-one in general, in fact if the boundary of $\Omega$ has a corner at the point $P_\theta$, the angle $\theta$ corresponds to an interval of angles (in every period). Nonetheless, we can define a monotone multi-valued
map $C^\circ$ that maps an angle $\theta$ to the maximal closed interval containing angles corresponding to $\theta$. This function has the following periodicity property:
\begin{equation*}
    C^\circ(\theta+2\Sbb k)=  C^\circ(\theta) +2\Sbb^\circ k \qquad \text{ for every }k\in \mathbb Z,
\end{equation*}
where $\Sbb^\circ$ denotes the surface area of $\Omega^\circ$.
If $\Omega$ is strictly convex, then the map $C^\circ$ is strictly monotone, while if the boundary of $\Omega$ is $C^1$, then $C^\circ$ is a (single-valued) map from $\R$ to $\R$ and it is continuous. Analogously, we can define the map $C_\circ$ associated to the correspondence $\psi\xleftrightarrow{\Omega^\circ} \theta$. Proposition \ref{prop:correspondence} guarantees that $C_\circ \circ C^\circ = C^\circ \circ C_\circ=  \text{id}$.

\begin{prop}\label{prop:difftrig}
    Let $\Omega\subset\R^2$ as above. The trigonometric functions $\sinom$ and $\cosom$ are Lipschitz and therefore differentiable almost everywhere. At every differentiability point $\theta$ of both functions, there exists a unique angle $\psi$ corresponding to $\theta$ and it holds that 
    \begin{equation*}
        \sinom'(\theta)= \cosomp(\psi) \qquad \text{and}\qquad \cosom'(\theta)= - \sinomp(\psi).
    \end{equation*}
    Naturally, the analogous result holds for the trigonometric functions $\sinomp$ and $\cosomp$.
\end{prop}

\noindent As a corollary of the previous proposition, we obtain the following convexity properties for the trigonometric functions.

\begin{corollary}\label{cor:convexitytrig}
    The functions $\sinom$ and $\cosom$ are concave in every interval in which they are non-negative and convex in every interval in which they are non-positive.
\end{corollary}

\noindent This convexity properties of the trigonometric functions will play a small but fundamental role in Section \ref{sec:CDheisenberg} in the form of the following corollaries.

\begin{corollary}\label{cor:convexity}
    Given a non-null constant $k\in\R$ and every angle $\theta$, consider the function
    \begin{equation}\label{eq:convexfunction}
        g:\R\to \R;\qquad g(t):=\sinom(\theta)\cosom(\theta+ k t) - \cosom(\theta) \sinom(\theta+ k t).
    \end{equation}
    If $k>0$ this function is convex for positive values of $t$ and concave for negative values of $t$, locally around $0$. Vice versa, If $k<0$ it is concave for positive values of $t$ and convex for negative values of $t$, locally around $0$. 
\end{corollary}

\begin{proof}
    The function $g(t)$ can be seen as a scalar product of two vectors in $\R^2$, therefore it is invariant by rotations. In particular, we consider the rotation that sends $\theta$ to $0$: this maps $P_\theta$ to the the positive $x$-axis and the set $\Omega$ to a convex, compact set $\tilde\Omega\subset\R^2$. Then, $g(t)$ in \eqref{eq:convexfunction} is equal to the function
    \begin{equation}
         t\mapsto - \cos_{\tilde\Omega}(0) \sin_{\tilde\Omega}( k t).
    \end{equation}
    The conclusion immediately follows from Corollary \ref{cor:convexitytrig}.
\end{proof}

\begin{corollary}
\label{cor:monotonicity}
      Given a non-null constant $k\in\R$ and every angle $\psi$, the function 
    \begin{equation}\label{eq:convexfunction2}
        h:\R\to \R;\qquad h(t):= 1 - \sinomp(\psi)\sinom\big((\psi+ k t)_\circ\big) - \cosomp(\psi) \cosom\big((\psi+ k t)_\circ\big).
    \end{equation}
    is non-decreasing for positive values of $t$ and non-increasing for negative values of $t$, locally around $0$.
\end{corollary}

\begin{proof}
    Note that $h$ is the derivative of the function
    \begin{equation}\label{eq:convexfunction3}
        \R \ni t\mapsto kt +\sinomp(\psi)\cosomp(\psi+ k t) - \cosomp(\psi) \sinomp(\psi+ k t),
    \end{equation}
    divided by $k$. The thesis follows from Corollary \ref{cor:convexity}, since the function \eqref{eq:convexfunction3} is the sum of a linear function and of a function of the type \eqref{eq:convexfunction}.
\end{proof}

In the following we are going to consider the trigonometric functions associated to the closed unit ball of a strictly convex norm $\normdot$ on $\R^2$, i.e. $\Omega:= \overline B^{\norm{\cdot}}_1(0)$. In this case, the polar set $\Omega^\circ$ is the closed unit ball $\overline B^{\norm{\cdot}_*}_1(0)$ of the dual norm $\normdot_*$. Moreover, according to the Pythagorean identity \eqref{eq:pytagorean}, if $\theta\xleftrightarrow{\Omega} \psi$ $\normdot$ then $Q_\psi$ is a dual vector of $P_\theta$. In particular, if $\normdot$ is a $C^1$ norm, Lemma \ref{lem:banachduality} ensures that
\begin{equation}\label{eq:correspondence}
    (\cosomp(\psi),\sinomp(\psi))=Q_\psi = d_{P_\theta} \normdot=  d_{(\cosom(\theta),\sinom(\theta))} \normdot.
\end{equation}

\subsection{Geodesics in the Heisenberg group}\label{sec:saymyname}

We present here the \sF Heisenberg group and study its geodesics. Let us consider the Lie group $M=\R^3$, equipped with the non-commutative group law, defined by
\begin{equation}
    (x, y, z) \star (x', y', z') = \bigg(x+x',y+y',z+z'+\frac12(xy' - x'y)\bigg),\qquad\forall\,(x, y, z), (x', y', z')\in\R^3,
\end{equation}
with identity element $\e=(0,0,0)$. In the notation of Section \ref{sec:prelim_sf}, we define the following morphism of bundles 
\begin{equation}
    \xi: M\times \R^2\to TM,\qquad \xi(x,y,z;u_1,u_2)=\bigg(x,y,z;u_1,u_2,\frac12(u_2 x - u_1 y)\bigg).
\end{equation}
The associated distribution of rank $2$ is spanned by the following left-invariant vector fields:
\begin{equation}
    X_1=\partial_x-\frac{y}2\partial_z,\qquad X_2=\partial_y+\frac{x}2\partial_z,
\end{equation}
namely $\dis=\text{span}\{X_1,X_2\}$. It can be easily seen that $\dis$ is bracket-generating. Then, letting $\normdot:\R^2\to\R_+$ be a norm, the \emph{\sF Heisenberg group} $\hei$ is the Lie group $M$ equipped with the \sF structure $(\xi,\normdot)$. By construction, also the resulting norm on the distribution is left-invariant, so that the left-translations defined by
\begin{equation}
\label{eq:left_translations}
    L_p:\hei\to\hei;\qquad L_p(q):=p\star q,
\end{equation}
are isometries for every $p\in \hei$.

In this setting, the geodesics were originally studied in \cite{Busemann} and \cite{Bereszynski} for the three-dimensional case and in \cite{Nestogol} for general left-invariant structures on higher-dimensional Heisenberg groups. We recall below the main results of \cite{Bereszynski} for strictly convex norms.

\begin{theorem}[{\cite[Thm.\ 1, Thm.\ 2]{Bereszynski}}]
\label{thm:geod_Heisenberg}
    Let $\hei$ be the \sF Heisenberg group, equipped with a strictly convex norm. Then, the following statements hold: 
    \begin{enumerate}
        \item[(i)] for any $q\in \hei\setminus\{x=y=0\}$, there exists a unique geodesic $\gamma:[0,1]\to\hei$ joining the origin and $q$. 
        \item[(ii)] $\gamma:[0,T]\to\hei$ is a geodesic starting at the origin if and only if it satisfies the Pontryagin's maximum principle for the time-optimal control problem:
        \begin{equation}\label{eq:time_optimal_problem}
    \begin{cases}
    \dot\gamma(t) = \displaystyle u_1(t) X_1(\gamma(t))+u_2(t) X_2(\gamma(t)),  \\
    \displaystyle u(t)\in \overline B^{\norm{\cdot}}_1(0),\quad\gamma(0)=q_0,\quad\text{and}\quad\gamma(T)=q_1, \\
      T\to \min.
    \end{cases}
\end{equation}
    \end{enumerate} 
\end{theorem}

\begin{remark}
    Note that the geodesics in \cite{Bereszynski} are found solving the Pontryagin maximum principle of Theorem \ref{thm:PMP} for the time-optimal problem.
    The latter is an equivalent formulation of \eqref{eq:Ptimalcontrolproblem}, however it produces arc-length parameterized geodesics. 
\end{remark}
    
The next step is to compute explicitly the exponential map. In \cite{Nestogol}, the author provides an explicit expression for geodesics starting at the origin, using the convex trigonometric functions functions presented in Section \ref{sec:convex_trigtrig}. Since therein the author solves the time-optimal problem, we prefer to solve explicitly the Hamiltonian system \eqref{eq:Hamiltonian_system}, in the case of the \sF Heisenberg group.

\begin{prop}[{\cite[Thm.\ 4]{Nestogol}}] 
\label{prop:geodesics_nestogol}
Let $\hei$ be the \sF Heisenberg group, equipped with a strictly convex norm. Let $\gamma : [0, 1] \to \hei$ be the projection of a (non-trivial) normal extremal $(\lambda_t)_{t\in [0,1]}$ starting at the origin, then $\gamma(t) = (x(t), y(t), z(t))$, with
\begin{equation}
    \begin{cases}
    \begin{aligned}
        x(t) &= \frac{r}{\omega}\left(\sinomp(\phi+\omega t ) - \sinomp(\phi)\right),\\
        y(t) &= -\frac{r}{\omega}\left(\cosomp(\phi+\omega t ) - \cosomp(\phi)\right),\\
        z(t) &= \frac{r^2}{2\omega^2}\left(\omega t + \cosomp(\phi+\omega t ) \sinomp(\phi) - \sinomp(\phi+\omega t ) \cosomp(\phi)\right),
    \end{aligned}
    \end{cases}
\end{equation}
for some $\phi\in [0, 2\mathbb S^\circ)$, $\omega \in \R \setminus \{0\}$ and $r > 0$. If $\omega = 0$, then 
\begin{equation}
\label{eq:straight_lines}
    \begin{cases}
    \begin{aligned}
        x(t) &= \left(r\cosom(\phi_\circ)\right) t,\\
        y(t) &= \left(r\sinom(\phi_\circ)\right) t,\\
        z(t) &= 0.
    \end{aligned}
    \end{cases}
\end{equation}
\end{prop}

\begin{proof}
    Firstly, we characterize the \sF Hamiltonian in the \sF Heisenberg group. Note that, without assuming additional regularity on $\normdot$, we can not apply directly Lemma \ref{lem:maximized_hamiltonian}. Nevertheless, we can still obtain an analogous result by means of convex trigonometry. Indeed, let $h_i(\lambda):=\langle\lambda,X_i(\pi(\lambda))\rangle$ for $i=1,2$, then 
    \begin{equation}
        H(\lambda):=\max_{u\in\R^2}\left(\sum_{i=1}^2 u_ih_i(\lambda)-\frac{\norm{u}}{2}\right),\qquad\forall\, \lambda\in\ T^*\hei.
    \end{equation}
    We introduce polar coordinates on $\R^2$ associated with $\normdot$ and its dual norm $\normdot_*$, namely $(u_1,u_2)\mapsto (\rho,\theta)$ and $(h_1,h_2)\mapsto(\zeta,\psi)$ where 
    \begin{equation}
    \label{eq:finally_some_coordinates}
        \begin{cases}
            u_1 =\rho\cosomp(\theta),\\
            u_2 =\rho\sinomp(\theta),
        \end{cases}
        \qquad\text{and}\qquad
        \begin{cases}
            h_1 =\zeta\cosomp(\psi),\\
            h_2 =\zeta\sinomp(\psi).
        \end{cases} 
    \end{equation}
    Hence, the \sF Hamiltonian becomes 
    \begin{equation}
    \label{eq:quante_righe_servono}
    \begin{split}
         H(\lambda)&=\max_{u\in\R^2}\left(\sum_{i=1}^2 u_ih_i(\lambda)-\frac{\norm{u}}{2}\right)\\
         &=\max_{\substack{\theta\in[0,2\Sbb)\\ \rho>0}}  \left(\rho\zeta \left(\cosom(\theta)\cosomp(\psi)+\sinom(\theta)\sinomp(\psi)\right)-\frac{\rho^2}{2}\right) \leq \max_{\rho>0} \left(\rho\zeta-\frac{\rho^2}{2}\right) = \frac{\zeta^2}{2},
    \end{split}
    \end{equation}
    where the last inequality is a consequence of \eqref{eq:NO_PYTHAGOREAN_IDENTITY_COMETIPERMETTI}. Moreover, we attain the equality in \eqref{eq:quante_righe_servono} if and only if $\rho=\zeta$ and $\psi=C^\circ (\theta)$. Therefore, since $\zeta=\|\hat\lambda\|_*$ with $\hat\lambda=(h_1(\lambda),h_2(\lambda))$, we conclude that
    \begin{equation}
        H(\lambda)=\frac12\big\|\hat\lambda\big\|^2_*,\qquad\forall\,\lambda\in T^*\hei\setminus\ann(\dis),
    \end{equation}
    and the maximum is attained at the control $u=\hat\lambda^*$.
    Furthermore, $H\in C^1(T^*M)$ by strict convexity of $\normdot$, cf. Proposition \ref{prop:propunderduality}. We write the system \eqref{eq:maximizedsystem} in coordinates $(x,y,z;h_1,h_2,h_3)$ for the cotangent bundle, where $h_3(\lambda):=\langle\lambda,\partial_z\rangle$. The vertical part of \eqref{eq:maximizedsystem} becomes
    \begin{equation}
    \label{eq:aux_ham_sys_Heisenberg}
    \begin{cases}
        \displaystyle \dot h_1(t) = \|\hat\lambda_t\|_*d_{\hat\lambda_t}\normdot_*\cdot\left(0,-h_3(t)\right),\\
        \displaystyle \dot h_2(t) = \|\hat\lambda_t\|_*d_{\hat\lambda_t}\normdot_*\cdot\left( h_3(t),0\right),\\
        \displaystyle \dot h_3(t) = 0.
    \end{cases}   
    \end{equation}
    Let $(\lambda_t)_{t\in [0,1]}$ be a normal extremal with associated maximal control given by $t\mapsto u(t)$, then we use Lemma \ref{lem:banachduality} to deduce that $\|\hat\lambda_t\|_*d_{\hat\lambda_t}\normdot_*=\hat\lambda_t^*= u(t)$. Therefore, letting $h_3(t)\equiv\omega\in\R$, we may rewrite \eqref{eq:aux_ham_sys_Heisenberg} as 
    \begin{equation}
    \label{eq:aux2_ham_sys_Heisenberg}
    \begin{cases}
        \displaystyle \dot h_1(t) = -  \omega\, u_2(t),\\
        \displaystyle \dot h_2(t) =  \omega\, u_1(t).
    \end{cases}   
    \end{equation}
    To solve this system, we use the polar coordinates \eqref{eq:finally_some_coordinates}: letting $t\mapsto (\rho(t),\psi(t))$ be the curve representing $\hat\lambda_t=(h_1(t),h_2(t))$, we deduce that $\rho(t)$ and $\psi(t)$ are absolutely continuous and satisfy
    \begin{equation}
        \rho(t)=\big\|\hat\lambda_t\big\|_*,\qquad \dot\psi(t)=\frac{h_1(t)\dot h_2(t)-\dot h_1(t) h_2(t)}{\rho^2(t)}.
    \end{equation}
    We may compute explicitly $\dot \rho(t)$ and $\dot\psi(t)$, using once again Lemma \ref{lem:banachduality}, the system \eqref{eq:aux2_ham_sys_Heisenberg} and identity \eqref{eq:pytagorean}:
\begin{equation}
    \dot \rho(t) = d_{\hat\lambda_t}\normdot_*\cdot(\dot h_1(t),\dot h_2(t))= \frac{\omega}{\|\hat\lambda_t\|_*}u(t)\cdot (-u_2(t),u_1(t))=0, \qquad    \dot\psi(t) = \omega.
\end{equation}
Thus, integrating the above identities, we obtain $\rho(t)\equiv r$ and $\psi(t)=\omega t +\phi$ for some $r>0$ and $\phi\in [0,2\mathbb S^\circ)$. Finally, we find an explicit expression for the maximal control:
\begin{equation}
    u(t)=\left(r\cosom(C_\circ (\phi+\omega t )),r\sinom(C_\circ (\phi+\omega t ))\right).
\end{equation} 
From this, we may explicitly integrate the horizontal part of the Hamiltonian system, obtaining the desired expression. In particular, if $\omega=0$ we immediately obtain \eqref{eq:straight_lines}. If $\omega\neq 0$, we may employ Proposition \ref{prop:difftrig} to conclude. 
\end{proof}

As $(\hei,\di_{SF})$ is complete, normal extremals can be extended to $\R$, according to Proposition \ref{prop:maximal_extension_normal_geod}. Thus, we may define the (extended) exponential map at the origin on the whole $T_0^*\hei \times \R$: 
\begin{equation}
\label{eq:extended_exponential_map}
\begin{split}
     G: \big([0,2\mathbb S^\circ)\times \R\times [0,\infty)\big)\times\R&\longrightarrow\hei,\\
     (\phi,\omega,r;t)&\longmapsto\big(x(\phi,\omega,r;t), y(\phi,\omega,r;t),  z(\phi,\omega,r;t)\big),
\end{split}
\end{equation}
where $(x(\phi,\omega,r;t), y(\phi,\omega,r;t),  z(\phi,\omega,r;t))$ correspond to the curve $(x(t),y(t),z(t))$ defined by Proposition \ref{prop:geodesics_nestogol} with initial datum $(\phi,\omega,r)$ and with the understanding that $G(\phi,\omega,0;t)\equiv 0$. By the properties of the convex trigonometric functions, $G$ is a $C^1$ map for $\omega\neq 0$. Moreover, thanks to Theorem \ref{thm:geod_Heisenberg}, for every initial datum $(\phi,\omega,r)$, the curve $t\mapsto G(\phi,\omega,r;t)$ is a geodesic between its endpoints for sufficiently small times. More precisely, it is minimal for $|t|<t^*=t^*(\phi,\omega,r)$, where $t^*>0$ is the first positive time such that $G(\phi,\omega,r;t^*)$ lies on the $z$-axis. In particular, a direct computation shows that
\begin{equation}
    t^*=
    \begin{cases}
    \displaystyle\frac{2\mathbb S^\circ}{|\omega|}, &\text{if }\omega\neq 0,\\
    \infty, &\text{if }\omega =0.
    \end{cases}
\end{equation}
We conclude this section by highlighting a property of geodesics in the Heisenberg group that will be relevant in our analysis. For the sake of notation, denote by $ \Omega^\circ_{(\phi,\omega,r)}$ the following transformation of $\Omega^\circ=\overline B^{\norm{\cdot}_*}_1(0)$:   
\begin{equation}
    \Omega^\circ_{(\phi,\omega,r)}:= R_{-\pi/2}\left[\frac{r}{\omega}(\Omega^\circ-(\cosomp(\phi),\sinomp(\phi)))\right],
\end{equation}
where $R_{-\pi/2}$ is counter-clockwise rotation in the plane of angle $-\pi/2$. 

\begin{prop}[{\cite[Thm.\ 1]{Bereszynski}}]
\label{prop:z=area}
    Let $\hei$ be the \sF Heisenberg group, equipped with a strictly convex norm and let $\gamma:[0,1]\to\hei$ be a geodesic starting at the origin, with $\gamma(t)=(x(t),y(t),z(t))$. Then, the curve $t\mapsto (x(t),y(t))$ is either a straight line or belongs to the boundary of $\Omega^\circ_{(\phi,\omega,r)}$. Moreover, for every $t\in[0,1]$, $z(t)$ equals the oriented area that is swept by the vector joining $(0,0)$ with $(x(s),y(s))$, for $s\in [0,t]$.
\end{prop}

\subsection{Failure of the \texorpdfstring{$\cd(K,N)$}{CD(K,N)} condition for \texorpdfstring{$C^{1,1}$}{C1,1}-norms} \label{sec:CDheisenberg}


In this section we contradict the validity of the $\cd(K,N)$ condition in the \sF Heisenberg group, equipped with a strictly convex and $C^{1,1}$ norm and with a smooth measure. The strategy follows the blueprint of the one presented in Section \ref{sec:argument_of_the_mago}. The main issue we have to address here is the low regularity (cf. Remark \ref{rmk:flexiamo}) of the midpoint and inverse geodesic maps of \eqref{eq:t-midpoint_map} and \eqref{eq:inverse_geodesic}. Nevertheless, using the explicit expression of geodesics presented in Proposition \ref{prop:geodesics_nestogol}, we successfully overcome these challenges through a series of technical lemmas, culminating in Corollary \ref{cor:finalmente}, Proposition \ref{prop:dopotantafatica} and Theorem \ref{thm:noCDC11}.

Let $\hei$ be the \sF Heisenberg group, equipped with a $C^{1,1}$ and strictly convex norm $\normdot$. According to Proposition \ref{prop:propunderduality}, the dual norm $\normdot_*$ is $C^1$ and strongly convex. Thus, in the notations of Section \ref{sec:convex_trigtrig}, the correspondences $C^\circ$ and $C_\circ$ are continuous functions. In order to ease the notation, in this section we sometimes use the shorthands:
\begin{equation*}
    \theta^\circ = C^\circ(\theta) \quad \text{and} \quad \psi_\circ = C_\circ(\psi), \qquad \forall\, \theta, \psi\in \R.
\end{equation*}
Alexandrov's theorem ensures that the dual norm $\normdot_*$ has a second derivative and a second-order Taylor expansion almost everywhere, we call $D_*\subset \R^2$ the set of twice differentiability of it.

\begin{prop}\label{prop:diffCcirc}
    Let $\psi\in [0,2\Sbb^\circ)$ be an angle such that $Q_\psi\in D_*$, then the function $C_\circ$ is differentiable at $\psi$ with positive derivative.
\end{prop}

\begin{proof}
    Consider a vector $v\in \R^2$ orthogonal to $d_{Q_\psi} \normdot_*$ such that $\norm{v}_{eu}=1$. Then, since $Q_\psi\in D_*$, there exists a constant $C\in\R$ such that 
    \begin{equation}\label{eq:normalongtangent}
        \norm{Q_\psi+ s v}_* = 1 + C s^2 + o(s^2), \qquad \text{as }s \to 0.
    \end{equation}
    Observe that, since the norm $\normdot_*$ is strongly convex, the constant $C$ is strictly positive. Consider the curve 
    \begin{equation*}
        s \mapsto x(s):= \frac{Q_\psi+ s v}{\norm{Q_\psi+ s v}_*},
    \end{equation*}
    which by definition is a parametrization of an arc of the unit sphere $S^{\norm{\cdot}_*}_1(0)=\partial\Omega^\circ$. Call $A(s)$ the signed area of the sector of $\Omega^\circ$ between the rays $O Q_\psi$ and $O x(s)$ (see Figure \ref{fig:diffCcirc1}). As a consequence of \eqref{eq:normalongtangent}, we deduce that 
    \begin{equation*}
        A(s)= \frac 12 k s + o(s^2), \qquad \text{as }s \to 0,
    \end{equation*}
    where $k$ is the scalar product between $Q_\psi$ and $v^\perp$, that is the vector obtained by rotating $v$ with an angle of $-\frac \pi 2$. In fact, the first-order term $\frac12 k s$ is the area of the triangle of vertices $O$, $Q_\psi$ and $Q_\psi+sv$, while the error term is controlled by the area of the triangle of vertices $x(s)$, $Q_\psi$ and $Q_\psi+sv$. The latter is an $o(s^2)$ as $s\to 0$, thanks to \eqref{eq:normalongtangent}. In particular, letting $\psi(s)$ be the angle such that $x(s)= Q_{\psi(s)}$, by definition of generalized angles, it holds that 
    \begin{equation*}
        \psi(s)-\psi= 2 A(s)= k s + o(s^2), \qquad \text{as }s \to 0.
    \end{equation*}
    Up to substituting the vector $v$ with $-v$, we can assume $k>0$.
    Then, in order to conclude, it is enough to prove that the function $s \mapsto C_\circ (\psi(s))$ is differentiable in $s=0$ with positive derivative. 
    
    \begin{figure}[h]

    \begin{minipage}[c]{.47\textwidth}

    \centering
    \begin{tikzpicture}

        \fill[color=blue!10!white] (1.75,2.5)--(-2,0)--(1.3375,4);
        \fill[color=white] (2,0) .. controls (1.9,3.5) and (1.4,4) .. (-1,5)--(2,5);

        \draw[ thick] (2,0) .. controls (1.9,3.5) and (1.4,4) .. (-1,5);
        \filldraw[black] (-2,0) circle (1.5pt);
        \filldraw[black] (1.75,2.5) circle (1.5pt);
        \filldraw[black] (1.3375,4) circle (1.5pt);
        \filldraw[black] (1.12,3.75) circle (1.5pt);
        \draw (1.3375,4)--(-2,0);
        \draw[very thick] (1.12,3.75)--(-2,0);
        \draw[very thick] (1.75,2.5) -- (-2,0);
        \draw[very thick, ->] (1.75,2.5) -- (1.2,4.5);
        \draw[very thick, ->] (1.75,2.5) -- (3.75,3.1);
        \node at (1.4,5)[label=south:$v$] {};
        \node at (-0.7,4.9)[label=south:$S^{||\cdot||_*}_1(0)$] {};
        \node at (2.2,4.2)[label=south:$Q_\psi + sv$] {};
        \node at (0.6,4.2)[label=south:$x(s)$] {};
        \node at (1.57,2.4)[label=south:$Q_\psi$] {};
        \node at (3,2.9)[label=south:$v ^\perp$] {};
        \node at (1,3.25)[label=south:${\color{blue}{A(s)}}$] {};
        \node at (-2.1,0.7)[label=south:$O$] {};
        
    \end{tikzpicture}
    \caption{Definition of $A(s)$.}
    \label{fig:diffCcirc1}
    
\end{minipage}
\hfill
\begin{minipage}[c]{.47\textwidth}

    \centering
    \begin{tikzpicture}

        \fill[color=red!10!white](1.8,1) --(-2,0)--(1,2.65)--(2,2.6);
        \fill[color=white] (2,0) .. controls (1.7,2) and (1.2,3) .. (-1,4)--(2,4);

        \draw[ thick] (2,0) .. controls (1.7,2) and (1.2,3) .. (-1,4);
        \filldraw[black] (-2,0) circle (1.5pt);
        \filldraw[black] (1,2.65) circle (1.5pt);
        \filldraw[black] (1.8,1) circle (1.5pt);
        \draw[very thick] (1.8,1)--(-2,0);
        \draw[very thick] (1,2.65) -- (-2,0);
        \draw[very thick,->](1.8,1)--(1.3,3.2);
        \node at (-0.7,4.9)[label=south:$S^{||\cdot||}_1(0)$] {};
        \node at (2,4)[label=south:${a= \text{Hess}_{Q_\psi}(||\cdot||_*)(v)}$] {};
        \node at (0.4,3.1)[label=south:$y(s)$] {};
        \node at (2.3,1.4)[label=south:$y(0)$] {};
        \node at (0.9,1.9)[label=south:${\color{red}{B(s)}}$] {};
        \node at (-2.1,0.7)[label=south:$O$] {};
        
    \end{tikzpicture}
    \caption{Definition of $B(s)$.}
    \label{fig:diffCcirc2}
    
\end{minipage}

\end{figure}

    First of all, by our choice of $k>0$, $s \mapsto C_\circ (\psi(s))$ is monotone non-decreasing close to $s=0$, being a composition of monotone non-decreasing functions. Second of all, we can show that it has a first-order expansion. To this aim, note that the curve
    \begin{equation}
        s \mapsto y(s):=d_{x(s)} \normdot_* 
    \end{equation}
    is a parametrization of an arc of the sphere $S^{\norm{\cdot}}_1(0)= \partial \Omega$ (cf. Lemma \ref{lem:banachduality}). Moreover, recalling that $Q_\psi\in D_*$ and using the homogeneity of the norm, we have that 
    \begin{equation}\label{eq:dualarea}
        y(s) = d_{Q_\psi+sv} \normdot_* =  d_{Q_\psi} \normdot_* + a \, s + o (s), \qquad \text{as }s \to 0,
    \end{equation}
    where $a:= \text{Hess}_{Q_\psi}(\normdot_*)(v)$. Observe that $a\neq 0$ because $\normdot_*$ is strongly convex and $\norm{Q_\psi}_*=1$. Then, call $B(s)$ the (signed) area of the sector of $\Omega$ between the rays $O y(0)$ and $O y(s)$ (see Figure \ref{fig:diffCcirc2}). Reasoning as we did for $A(s)$, from \eqref{eq:dualarea} we deduce that 
    \begin{equation}
    \label{eq:area_B(s)}
        B(s) = \frac 12 \scal{y(0)}{a^\perp} s + o(s^2), \qquad \text{as }s \to 0.
    \end{equation}
    On the other hand, by definition 
    \begin{equation}
    \label{eq:C_circ=area_B(s)}
        C_\circ (\psi(s)) - C_\circ(\psi(0)) = 2 B(s)=\scal{y(0)}{a^\perp} s + o(s^2), \qquad \text{as }s \to 0.
    \end{equation}
    This shows that the function $s \mapsto C_\circ (\psi(s))$ is differentiable in $s=0$ with derivative $\scal{y(0)}{a^\perp}$. In addition, since $C_\circ\circ\psi$ is non-decreasing close to $s=0$, \eqref{eq:C_circ=area_B(s)} also implies that $\scal{y(0)}{a^\perp}\geq 0$. We are left to show that $\scal{y(0)}{a^\perp}$ is strictly positive. If $\scal{y(0)}{a^\perp}=0$ then $a$ is parallel to $y(0)$, however, according to \eqref{eq:dualarea}, the vector $a$ is tangent to the sphere $S^{\norm{\cdot}}_1(0)$ at $y(0)$ and therefore we obtain a contradiction.
\end{proof}

\begin{remark}
\label{rmk:diff_C_full_measure}
    Since the norm is invariant by homotheties, then also $D_*$ is so, thus the set of angles $\psi$ such that $Q_\psi\not\in D_*$ has null $\Leb^1$-measure. In particular, the function $C_\circ$ is differentiable with positive derivative $\Leb^1$-almost everywhere, as a consequence of Proposition \ref{prop:diffCcirc}.
\end{remark}


As already mentioned, the strategy to prove the main theorem of this section is the same of Section \ref{sec:argument_of_the_mago}. In particular, it is fundamental to prove estimates on the volume contraction along geodesic homotheties. To this aim, we consider the Jacobian determinant of the exponential map \eqref{eq:extended_exponential_map}:
\begin{equation}
  J(\phi,\omega,r;t):= \left|\, \det  \begin{pNiceMatrix}
\frac{\partial x}{\partial r}  & \frac{\partial x}{\partial \phi}  & \frac{\partial x}{\partial \omega} \\
\frac{\partial y}{\partial r}  & \frac{\partial y}{\partial \phi}  & \frac{\partial y}{\partial \omega} \\
\frac{\partial z}{\partial r}  & \frac{\partial z}{\partial \phi}  & \frac{\partial z}{\partial \omega} 
\end{pNiceMatrix} (\phi,\omega,r;t) \right|
\end{equation}
where we recall $x(\phi,\omega,r;t),\, y(\phi,\omega,r;t),\, z(\phi,\omega,r;t)$ are defined in Proposition \ref{prop:geodesics_nestogol}. In order to study this, we will use the following formulation:
\begin{equation}\label{eq:J}
 J(\phi,\omega,r;t) = \left| \frac{\partial z}{\partial \omega}(\phi,\omega,r;t) \det(M_1) - \frac{\partial z}{\partial \phi}(\phi,\omega,r;t) \det(M_2) + \frac{\partial z}{\partial r} (\phi,\omega,r;t)\det(M_3)  \right|
\end{equation}
where
\begin{equation*}
    M_1:= \begin{pNiceMatrix}
\frac{\partial x}{\partial r}  & \frac{\partial x}{\partial \phi}  \\
\frac{\partial y}{\partial r}  & \frac{\partial y}{\partial \phi} \end{pNiceMatrix}(\phi,\omega,r;t), \quad M_2:= \begin{pNiceMatrix}
\frac{\partial x}{\partial r}  & \frac{\partial x}{\partial \omega}  \\
\frac{\partial y}{\partial r}  & \frac{\partial y}{\partial \omega} \end{pNiceMatrix}(\phi,\omega,r;t), \quad M_3:= \begin{pNiceMatrix}
\frac{\partial x}{\partial \phi}  & \frac{\partial x}{\partial \omega}  \\
\frac{\partial y}{\partial \phi}  & \frac{\partial y}{\partial \omega} \end{pNiceMatrix}(\phi,\omega,r;t).
\end{equation*}

 We are particularly interested in studying the behaviour of $ J(\phi,\omega,r;t)$ as $t\to0$. In the following lemmas we estimate the behaviour of every term in \eqref{eq:J} as $t\to0$.

\begin{notation}
    Let $I\subset\R$ be an interval containing $0$. Given a function $f:I \to \R$ and $n\in \N$, we write 
    \begin{equation}
        f(t) \sim  t^n, \qquad \text{as }t \to 0,
    \end{equation}
    if there exists a constant $C\neq 0$ such that $f(t)= C t^n + o (t^n)$, as $t\to 0$.
\end{notation}

\begin{lemma}\label{lem:Jestimate1}
    Let $\phi\in[0,2\Sbb^\circ)$ be a differentiability point for the map $C_\circ$, $r>0$ and $\omega\neq 0$, then
    \begin{equation}\label{eq:expansion1}
        \det \big( M_1  (\phi,\omega,r;t) \big) \sim t^2, \qquad \text{as }t \to 0,
    \end{equation}
    while 
    \begin{equation}\label{eq:expansion2}
        \det \big( M_2  (\phi,\omega,r;t) \big),\, \det \big( M_3  (\phi,\omega,r;t) \big) = O(t^3), \qquad \text{as }t \to 0.
    \end{equation}
\end{lemma}

\begin{proof}
    Let us begin by proving \eqref{eq:expansion1}. Firstly, since the function $C_\circ$ is differentiable at $\phi$, we can compute the following Taylor expansions as $t\to0$, using Proposition \ref{prop:difftrig}:
\begin{equation}\label{eq:taylor1}
\begin{split}
    \cosom\big( (\phi+\omega t)_\circ\big) &= \cosom\big( \phi_\circ\big) - t \omega C'_\circ (\phi) \sinomp(\phi) + o(t),\\
    \sinom\big( (\phi+\omega  t)_\circ\big) &= \sinom\big( \phi_\circ\big) + t \omega C'_\circ (\phi) \cosomp(\phi) + o(t).
\end{split}
\end{equation}
    Therefore, we may expand the entries of $M_1$ as $t\to 0$: 
\begin{equation}
\label{eq:taylor1000mila}
\begin{split}
    \frac{\partial x}{\partial r} (\phi,\omega,r;t) &= \frac 1 \omega\left(\sinomp(\phi+\omega t ) - \sinomp(\phi)\right)=\cosom(\phi_\circ) t+o(t),\\
    \frac{\partial y}{\partial r} (\phi,\omega,r;t) &= -\frac 1 \omega\left(\cosomp(\phi+\omega t ) - \cosomp(\phi)\right)=\sinom(\phi_\circ) t+o(t),\\
    \frac{\partial x}{\partial \phi} (\phi,\omega,r;t) &= \frac{r}{\omega} \left( \cosom\big( (\phi+\omega t )_\circ\big) - \cosom\big( \phi_\circ\big) \right)= -t rC'_\circ (\phi) \sinomp(\phi) + o(t),\\    
    \frac{\partial y}{\partial \phi} (\phi,\omega,r;t) &= \frac{r}{\omega} \left( \sinom\big( (\phi+\omega t )_\circ\big) - \sinom\big( \phi_\circ\big) \right) = tr C'_\circ (\phi) \cosomp(\phi) + o(t)
,
\end{split}
\end{equation}
where we used once again Proposition \ref{prop:difftrig}. Finally, the determinant has the following Taylor expansion as $t\to0$: 
\begin{equation}\label{eq:detM1}
\begin{split}
    \det (M_1) &= \frac{\partial x}{\partial r} \frac{\partial y}{\partial \phi} - \frac{\partial y}{\partial r}\frac{\partial x}{\partial \phi}  = t^2 r  C'_\circ (\phi) \big( \sinomp(\phi) \sinom\big( \phi_\circ\big) + \cosomp(\phi)\cosom\big( \phi_\circ\big) \big)+o(t^2)\\
    &= t^2 r  C'_\circ (\phi)+o(t^2),   
\end{split}
\end{equation}
where, in the last equality, we used Proposition \ref{prop:correspondence}. This proves \eqref{eq:expansion1}, keeping in mind Proposition \ref{prop:diffCcirc}, which guarantees that $C_\circ'(\phi)>0$.

Now we prove \eqref{eq:expansion2} for $\det(M_2)$, the proof for $\det(M_3)$ is analogous. As a first step, reasoning as before, we can Taylor expand at second-order the following quantities, as $t\to0$:
\begin{equation}
\label{eq:taylor3}
    \begin{split}
        \cosomp (\phi+\omega  t) &= \cosomp( \phi)  - \omega t \sinom(\phi_\circ) - \frac 12 (\omega t)^2 C_\circ'(\phi)\cosomp( \phi) + o (t^2), \\
        \sinomp (\phi+\omega  t) &= \sinomp( \phi)  + \omega t \cosom(\phi_\circ) - \frac 12 (\omega t)^2 C_\circ'(\phi)\sinomp( \phi) + o (t^2) .
    \end{split}
\end{equation}
Hence, we deduce the expansion for the derivative of $x$ in the $\omega$ direction, as $t\to 0$: 
\begin{equation}
\label{eq:redelleespansionizioporcone}
\begin{split}
    \frac{\partial x}{\partial \omega}& (\phi,\omega,r;t) = -\frac{r}{\omega^2} \big( \sinomp (\phi+\omega t) - \sinomp( \phi) \big) + \frac{rt}{\omega} \cosom\big( (\phi+ \omega  t)_\circ\big)\\
    &=-\frac{r}{\omega}\big( t \cosom(\phi_\circ) - \frac 12 \omega t^2 C_\circ'(\phi)\sinomp( \phi)\big)+\frac{rt}{\omega}\big(\cosom\big( \phi_\circ\big) - t \omega C'_\circ (\phi) \sinomp(\phi)\big)+ o (t^2) \\
    &=-\frac{1}{2} r t^2 C_\circ'(\phi)\sinomp( \phi)+ o (t^2). 
\end{split}
\end{equation}
An analogous computation shows that the derivative of $y$ in $\omega$ has the ensuing expansion as $t\to 0$:
\begin{equation}
\label{eq:reginadelleespansioniziaporcona}
    \frac{\partial y}{\partial \omega} (\phi,\omega,r;t) =\frac{1}{2} r t^2 C_\circ'(\phi)\cosomp( \phi)+ o (t^2).
\end{equation}
Note that, on the one hand, \eqref{eq:redelleespansionizioporcone} and \eqref{eq:reginadelleespansioniziaporcona} imply that 
\begin{equation}\label{eq:behaviour1}
    \frac{\partial x}{\partial \omega}=O(t^2)\qquad\text{and}\qquad\frac{\partial y}{\partial \omega}=O(t^2),
\end{equation}
as $t\to 0$. On the other hand, by \eqref{eq:taylor1000mila}, we can deduce the following behavior, as $t\to0$: 
\begin{equation}\label{eq:behaviour2}
    \frac{\partial x}{\partial r}=O(t)\qquad\text{and}\qquad\frac{\partial y}{\partial  r}=O(t).
\end{equation}
Thus, \eqref{eq:behaviour1} and \eqref{eq:behaviour2} prove the claimed behavior of $\det(M_2)$ as $t\to 0$, since
\begin{equation}
    \det (M_2) = \frac{\partial x}{\partial r} \frac{\partial y}{\partial \omega}- \frac{\partial y}{\partial r}\frac{\partial x}{\partial \omega}.
\end{equation}

\end{proof}

In the next lemmas, we study the derivatives of $z$. These are the most delicate to estimate, since the second-order Taylor polynomial of $z$ is zero and higher-order derivatives may not exist.



\begin{notation}
\label{notation:brilliant_notation}
    Let $g:\R\to\R$ be a function. We write 
        \begin{equation}
        g(t)= C(1+O(\varepsilon)) f(t), \qquad \forall\,t\in [-\rho,\rho].
    \end{equation}
    if there exists a constant $K>0$ and a function $f:\R\to\R$ such that, for every $\varepsilon>0$, there exist positive constants $C=C(\varepsilon),\rho=\rho(\varepsilon)>0$ for which the following holds
    \begin{equation}
        C(1-K\varepsilon) f(t)<g(t)< C(1+K\varepsilon) f(t), \qquad \forall\, t\in [-\rho,\rho].
    \end{equation}
\end{notation}

\begin{lemma}\label{lem:Jestimate2}
    Given $\varepsilon>0$ sufficiently small, for $\Leb^1$-almost every $\phi\in [0,2\Sbb^\circ)$, every $r>0$ and $\omega\neq 0$, there exist two positive constants $k=k(r)$ and $\rho=\rho (\phi,\omega,r)$ such that
    \begin{equation}\label{eq:difficultestimate}
        \frac{\partial z}{\partial \omega}(\phi,\omega,r;t)=(1+O(\varepsilon)) k t^3, \qquad \qquad \forall\,t\in [-\rho,\rho].
    \end{equation}
\end{lemma}

\begin{proof}
    First of all, we compute that 
    \begin{equation}\label{eq:difficultpartiald}
    \begin{split}
         \frac{\partial z}{\partial \omega}(\phi,\omega,r;t) &= \frac{r^2t}{2\omega^2} \big ( 1 - \sinomp(\phi)\sinom\big((\phi+\omega  t)_\circ\big) - \cosomp(\phi) \cosom\big((\phi+\omega  t)_\circ\big) \big) \\
         &\quad - \frac{r^2}{\omega^3} \big ( \omega t + \sinomp(\phi)\cosomp(\phi+\omega t) - \cosomp(\phi) \sinomp(\phi+\omega t) \big).
    \end{split}
    \end{equation}
    In order to evaluate this quantity, fix an angle $\psi\in[0,2\Sbb^\circ)$, for which Proposition \ref{prop:diffCcirc} holds, and consider the function $f_\psi$, defined as
    \begin{equation}
      s \mapsto f_\psi(s) :=  1 - \sinomp(\psi+s)\sinom(\psi_\circ) - \cosomp(\psi+s) \cosom(\psi_\circ).
    \end{equation}
    Notice that \eqref{eq:pytagorean} ensures that $f_\psi(0)=0$, moreover direct computations show that 
    \begin{equation}
        f'_\psi(0)=0 \qquad \text{and} \qquad f''_\psi(0) = C'_\circ (\psi) >0.
    \end{equation}
    Consequently, it holds that 
    \begin{equation}\label{eq:businessisbusiness}
        f_\psi(s)= \frac12 C'_\circ (\psi) \cdot s^2 + o(s^2), \qquad \text{as }s \to 0.
    \end{equation}
    For every $n\in \mathbb Z$ and $m\in\N$ define the set of angles
    \begin{equation}
        E_{n,m} := \left\{\psi \in [0,2\Sbb^\circ) \, :\, (1+\varepsilon)^{n-1} s^2 < f_\psi (s) < (1+\varepsilon)^{n+1} s^2, \text{ for every }s\in\Big[-\frac1m,\frac1m\Big]\right\}.
    \end{equation}
    Observe that, by \eqref{eq:businessisbusiness}, we have that $ E:=\bigcup_{n\in \mathbb Z,m\in \N} E_{n,m}$ covers all differentiability points of $C_\circ$, in particular $E$ has full $\Leb^1$-measure, cf. Remark \ref{rmk:diff_C_full_measure}.

     Now, fix $\omega> 0$, $r>0$ and take $\phi\in [0,2\Sbb^\circ)$ to be a density point\footnote{We say that $r\in\R$ is a \emph{density point} for a measurable set $J\subset \R$ if 
\begin{equation} 
\lim_{s\to 0^+} \frac{\Leb^1(J \cap [r-s,r+s])}{2s} = 1.
\end{equation} 
} for the set $E_{n,m}$, for some $n\in \mathbb Z$, $m\in \N$. We are going to prove the statement \eqref{eq:difficultestimate} for our choice of parameters and for positive times. The cases with $\omega<0$ and negative times are completely analogous. Let $0<\rho(\phi,\omega,r)< \frac{1}{2 \omega m}$ be sufficiently small such that for every $t\in(0,\rho]$
    \begin{equation}\label{eq:predensity}
        \Leb^1(E_{n,m} \cap [\phi-2\omega  t,\phi+2\omega  t])> 4 \omega t (1-\varepsilon/4).
    \end{equation}
    Introduce the set 
    \begin{equation}
         F_{n,m}:= \{s \in \R \, :\,\phi+\omega s \in E_{n,m}\} .
    \end{equation}
    Observe that, from \eqref{eq:predensity}, we can deduce that for every $t\in(0,\rho]$,
    \begin{equation}\label{eq:density}
        \Leb^1(F_{n,m} \cap [-2 t,2 t])> 4  t (1-\varepsilon/4).
    \end{equation}
    Now, given every $t\in(0,\rho]$, \eqref{eq:density} ensures that there exists $\bar s\in [t(1-\varepsilon),t]$ such that $\bar s\in F_{n,m}$. Then, thanks to Corollary \ref{cor:monotonicity}, we obtain that 
    \begin{equation}\label{eq:nonnepossopiu}
        \begin{split}
             1 - \sinomp(\phi)\sinom&\big((\phi+\omega  t)_\circ\big)- \cosomp(\phi) \cosom\big((\phi+\omega  t)_\circ\big) \\
             & \geq 1 - \sinomp(\phi)\sinom\big((\phi+\omega  \bar s)_\circ\big)- \cosomp(\phi) \cosom\big((\phi+\omega  \bar s)_\circ\big)\\
            &= f_{\phi+\omega  \bar s}(-\omega  \bar s) \geq (1+\varepsilon)^{n-1} (\omega  \bar s)^2 \geq (1-\varepsilon)^2 (1+\varepsilon)^{n-1} (\omega  t)^2,
        \end{split}
    \end{equation}
    where the second to last inequality holds by our choice of the parameter $\rho$ and because $\phi+\omega  \bar s \in E_{n,m}$.
    With an analogous argument, we can find an element in $ [t,t(1+\varepsilon)]\cap F_{n,m}$ and deduce the estimate:
    \begin{equation}\label{eq:nonnepossopiu2}
        1 - \sinomp(\phi)\sinom \big((\phi+\omega  t)_\circ\big)- \cosomp(\phi) \cosom\big((\phi+\omega  t)_\circ\big) \leq (1+\varepsilon)^2 (1+\varepsilon)^{n+1} (\omega  t)^2.
    \end{equation}
    Combining \eqref{eq:nonnepossopiu} and \eqref{eq:nonnepossopiu2}, we conclude that, on $(0,\rho]$, the following holds
    \begin{equation}\label{eq:sonostancodimetterelabel}
         1 - \sinomp(\phi)\sinom \big((\phi+\omega t)_\circ\big)- \cosomp(\phi) \cosom\big((\phi+\omega t)_\circ\big) = (1+ O(\varepsilon))(1+\varepsilon)^{n} (\omega  t)^2,
    \end{equation}
    in the Notation \ref{notation:brilliant_notation}. Consequently, we deduce:
    \begin{equation}\label{eq:stimapazza1}
        \frac{r^2t}{2\omega^2} \left ( 1 - \sinomp(\phi)\sinom\big((\phi+\omega  t)_\circ\big) - \cosomp(\phi) \cosom\big((\phi+\omega t)_\circ\big) \right) = (1+ O(\varepsilon))(1+\varepsilon)^{n}\frac{r^2 t^3}{2},
    \end{equation}
    for $t\in (0,\rho]$. To estimate the second term in \eqref{eq:difficultpartiald}, observe that 
    \begin{equation}
    \begin{split}
          \frac{\partial}{\partial s}\big ( \omega  s + \sinomp(\phi)\cosomp&(\phi+\omega  s) - \cosomp(\phi) \sinomp(\phi+\omega  s) \big) \\
          &= \omega  \big ( 1 - \sinomp(\phi)\sinom\big((\phi+\omega  s)_\circ\big) - \cosomp(\phi) \cosom\big((\phi+\omega  s)_\circ\big) \big).
    \end{split}
    \end{equation}
    In particular, since in $s=0$ this quantity is equal to $0$, we have that for $t\in (0,\rho]$
    \begin{equation}\label{eq:pezzotosto1}
        \begin{split}
             \omega  t &+ \sinomp(\phi)\cosomp(\phi+\omega  t) - \cosomp(\phi) \sinomp(\phi+\omega  t) \\
             &=\omega  \int_0^t \big ( 1 - \sinomp(\phi)\sinom\big((\phi+\omega  s)_\circ\big) - \cosomp(\phi) \cosom\big((\phi+\omega  s)_\circ\big) \big) \de s \\
             &= \omega  \int_0^t (1+ O(\varepsilon))(1+\varepsilon)^{n}(\omega  t)^2 \de s = (1+ O(\varepsilon))(1+\varepsilon)^{n} \frac{(\omega  t)^3 }{3},
        \end{split}
    \end{equation}
    where the second equality follows from \eqref{eq:sonostancodimetterelabel}. Then, we obtain that: 
    \begin{equation}\label{eq:stimapazza2}
        \frac{r^2}{\omega^3} \big ( \omega  t + \sinomp(\phi)\cosomp(\phi+\omega  t) - \cosomp(\phi) \sinomp(\phi+\omega  t) \big)= (1+ O(\varepsilon))(1+\varepsilon)^{n}\frac{r^2t^3}{3}.
    \end{equation}
    Finally, putting together \eqref{eq:stimapazza1} and \eqref{eq:stimapazza2}, we conclude that
    \begin{equation}
         \frac{\partial z}{\partial \omega}(\phi,\omega,r;t) = (1+ O(\varepsilon))(1+\varepsilon)^{n}\frac{r^2t^3}{6},\qquad\forall\, t\in (0,\rho],
    \end{equation}
    that is \eqref{eq:difficultestimate} with $k=k(r) :=(1+\varepsilon)^{n}\frac{r^2}{6}$.
   To conclude, observe that we proved the statement for (every $r>0$, $\omega\neq 0$ and) every $\phi\in [0,2\Sbb^\circ)$ which is a density point of some $E_{n,m}$ and the set of such angles has full $\Leb^1$-measure in  $[0,2\Sbb^\circ)$. Indeed, $E=\bigcup_{n\in \mathbb Z,m\in \N} E_{n,m}$ has full $\Leb^1$-measure in  $[0,2\Sbb^\circ)$ and almost every point of a measurable set is a density point.
\end{proof}

\begin{lemma}\label{lem:Jestimate3}
    Let $\phi\in [0,2\Sbb^\circ)$ be a differentiability point for the map $C_\circ$, $r>0$ and $\omega\neq 0$, then
    \begin{equation*}
        \frac{\partial z}{\partial r}(\phi,\omega,r;t),\, \frac{\partial z}{\partial \phi} (\phi,\omega,r;t) = o (t^2), \qquad \text{as }t\to 0. 
    \end{equation*}
\end{lemma}

\begin{proof}
     We start by proving the statement for $\frac{\partial z}{\partial r}$. We have that
    \begin{equation}
         \frac{\partial z}{\partial r} (\phi,\omega,r;t) = \frac{r}{\omega^2}\left(\omega t + \cosomp(\phi+\omega t ) \sinomp(\phi) - \sinomp(\phi+\omega t ) \cosomp(\phi)\right).
    \end{equation}
    Direct computations show that, on the one hand,
    \begin{multline}
         \frac{\partial}{\partial t}\bigg|_{t=0} \big(\omega t + \cosomp(\phi+\omega t ) \sinomp(\phi) - \sinomp(\phi+\omega t ) \cosomp(\phi)\big) \\
        = \omega - \omega\sinom(\phi_\circ)\sinomp(\phi) -  \omega \cosom(\phi_\circ) \cosomp(\phi) =0,
    \end{multline}
    where we applied Proposition \ref{prop:correspondence}, and, on the other hand,
    \begin{multline}
        \frac{\partial^2}{\partial t^2}\bigg|_{t=0} \big(\omega t + \cosomp(\phi+\omega t ) \sinomp(\phi)- \sinomp(\phi+\omega t ) \cosomp(\phi)\big) \\
        =- \omega^2\cosomp(\phi)\sinomp(\phi) C'_\circ(\phi) + \omega^2  \sinomp(\phi)\cosomp(\phi) C'_\circ(\phi) =0.
    \end{multline}
Consequently, we conclude the proof of the first part of the statement:
    \begin{equation*}
         \frac{\partial z}{\partial r} (\phi,\omega,r;t) = \frac{r}{\omega^2} \cdot o(t^2) = o(t^2),\qquad\text{as }t\to 0.
    \end{equation*}
 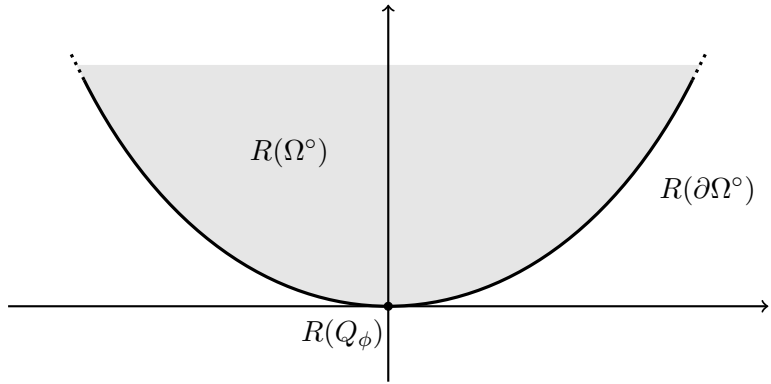
\begin{figure}[b!]
    \centering
        \begin{tikzpicture}
    \fill[color=black!10!white] (-4,3) .. controls (-2,-1) and (2,-1) .. (4,3);
    \fill[color=black!10!white] (-4,3)--(-4.1,3.2)--(4.1,3.2)-- (4,3);
    \draw[thick,->] (0,-1)--(0,4);
    \draw[thick,->] (-5,0)--(5,0);
    \draw[very thick, dotted](-4,3)--(-4.2,3.4);
    \draw[very thick, dotted](4,3)--(4.2,3.4);
    \node at (-0.6,0.1)[label=south:$R(Q_\phi)$] {};
    \node at (4.2,2)[label=south:$R(\partial\Omega^\circ)$] {};
    \node at (-1.3,2.5)[label=south:$R(\Omega^\circ)$] {};
    \filldraw[black] (0,0) circle (1.5pt);
    \draw[very thick] (-4,3) .. controls (-2,-1) and (2,-1) .. (4,3);
    \end{tikzpicture}
    \caption{Image of $Q_\phi$ and $\Omega^\circ$ through the rigid transformation $R$.}
    \label{fig:R}
\end{figure}

    In order to prove the statement for $\frac{\partial z}{\partial \omega}$, we use a geometric argument based on Proposition \ref{prop:z=area}. First of all, recall that $d_{Q_\phi}\normdot_*$ identifies a half-plane tangent at $Q_\phi$ and containing $\Omega^\circ$. Thus, we can find a rigid transformation $R:\R^2\to \R^2$, such that $R(Q_\phi)= (0,0)$ and $R(\Omega^\circ)$ is contained in $\{y\geq0\}\subset\R^2$, see Figure \ref{fig:R}. Then, as $\normdot_*$ is $C^{1,1}$, the image of the unit sphere $R\big(\partial\Omega^\circ\big)$, can be described (locally around $O$) as the graph of a non-negative function $f\in C^{1,1}(\R)$ with $f(0)=0$. In addition, by our choice of $\phi\in[0,2\Sbb^\circ)$, $f$ is twice differentiable in $0$ with strictly positive second derivative $f''(0):=c>0$. Now consider the function $p$ defined in a neighborhood of $\phi$ as 
    \begin{equation*}
        p(\psi):= \p_x \big( R(Q_{\psi})\big),
    \end{equation*}
    where $\p_x:\R^2\to \R$ denotes the projection on the $x$-axis, i.e.\ $\p_x(a,b)=a$.
    
    Second of all, for $s_1,s_2\in\R$, call $F(s_1,s_2)$ the signed area between the segment connecting $(s_1,f(s_1))$ and $(s_2,f(s_2))$ and the graph of $f$ (intended positive if $s_1< s_2$ and negative if $s_1>s_2$), see Figure \ref{fig:F(s_1,s_2)}. Proposition \ref{prop:z=area} ensures that for $\psi$ in a neighborhood of $\phi$ it holds that
    \begin{equation*}
        z (\psi,\omega,r;t) = \frac{r^2}{\omega^2}F \big(p(\psi), p(\psi+\omega  t)\big).
    \end{equation*}
    In particular, we obtain that 
    \begin{equation}\label{eq:derivativeofF}
        \frac{\omega^2}{r^2}\,\frac{\partial z}{\partial \phi} (\phi,\omega,r;t) = \frac{\partial}{\partial s_1} F(0, p(\phi+\omega  t))\cdot p'(\phi) +  \frac{\partial}{\partial s_2} F(0,p(\phi+\omega  t))\cdot p'(\phi+ \omega  t).
    \end{equation}

\begin{figure}[h!]

\centering

\begin{minipage}[c]{.5\textwidth}
  \begin{center}
    \begin{tikzpicture}[scale=0.75]

\fill[color=blue!10!white](-2,0.62) -- (3,1.5)--(3,0)--(-2,0);
\fill[white] (-4,3) .. controls (-2,-1) and (2,-1) .. (4,3)--(4,0)--(-4,0);

\draw[thick,->] (0,-2)--(0,5);
\draw[thick,->] (-5,0)--(5,0);
\draw[very thick, dotted](-4,3)--(-4.2,3.4);
\draw[very thick, dotted](4,3)--(4.2,3.4);
\draw[thick,dotted](-2,0.6)--(-2,0);
\draw[thick,dotted](3,1.5)--(3,0);
\filldraw[black] (-2,0.62) circle (1.5pt);
\filldraw[black] (3,1.5) circle (1.5pt);
\draw[very thick](-2,0.62)--(3,1.5);

\node at (-2,0.1)[label=south:$s_1$] {};
\node at (3,0.1)[label=south:$s_2$] {};
\node at (1.2,2.45)[label=south:${\color{blue}{F(s_1,s_2)}}$] {};
\node at (3.8,2)[label=south:$f$] {};
\draw[very thick] (-4,3) .. controls (-2,-1) and (2,-1) .. (4,3);

\end{tikzpicture} 

  \caption{Definition of the function $F(s_1,s_2)$.}
  \label{fig:F(s_1,s_2)}
      
  \end{center}
\end{minipage}%
\begin{minipage}[c]{.5\textwidth}

  \centering

   \begin{tikzpicture}[scale=0.75]

\fill[color=red!10!white](-1.5,3) -- (0,0)--(2,0)--(2,0.62);
\fill[white] (-4,3) .. controls (-2,-1) and (2,-1) .. (4,3)--(4,0)--(-4,0);
\draw[very thick](-1.5,3) -- (0,0);
\draw[very thick](2,0.62)--(-1.5,3);
\draw[thick,->] (0,-2)--(0,5);
\draw[thick,->] (-5,0)--(5,0);
\draw[very thick, dotted](-4,3)--(-4.2,3.4);
\draw[very thick, dotted](4,3)--(4.2,3.4);
\filldraw[black] (2,0.62) circle (1.5pt);
\filldraw[black] (0,0) circle (1.5pt);
\filldraw[black] (-1.5,3) circle (1.5pt);
\draw[thick,dotted](2,0.6)--(2,0);
\draw[very thick] (-4,3) .. controls (-2,-1) and (2,-1) .. (4,3);
\node at (2,0.1)[label=south:$q$] {};
\node at (-1.5,2.9)[label=north:${R(O)}$] {};
\node at (1.1,2.5)[label=south:${\color{red}{A(q)}}$] {};

\end{tikzpicture}

  \caption{Definition of the function $A(q)$.}
  \label{fig:guardaunpo':A}
\end{minipage}

\end{figure}

    We now proceed to compute the terms in the last formula, starting from the ones involving $p'$. To this aim, consider the point $(x_0,y_0):= R(O)$ and, for every $q$ in a neighborhood of $0$, call $A(q)$ the signed area inside $R\big(\partial\Omega^\circ\big)$ between the segments $(x_0,y_0)O$ and $(x_0,y_0)(q,f(q))$. Observe that 
    \begin{equation*}
        A'(q)= \frac 12 \scal{(1,f'(q))}{(y_0-f(q), q-x_0)} = \frac 12 y_0 + O(q),\qquad\text{as }q\to 0.
    \end{equation*}
    Note that, in the last equality, we have used that $f(0)=f'(0)=0$ and $f\in C^{1,1}(\R)$. Consequently, since $A(0)=0$, we have that
    \begin{equation}
    \label{eq:quante_labelss}
        A(q) = \frac 1 2 y_0 q + O(q^2),\qquad\text{as }q\to 0.
    \end{equation}
    On the other hand, by the definition of angle it holds that $2A(p(\phi+\vartheta))=\vartheta$ for every $\vartheta$ sufficiently small and therefore, invoking \eqref{eq:quante_labelss} and observing that $p\in C^1$, we obtain that
    \begin{equation}\label{eq:treddisperato}
        p(\phi+ \vartheta) = \frac{1}{y_0} \vartheta + o(\vartheta) \quad \text{and} \quad p'(\phi+ \vartheta) = \frac{1}{y_0} + o(1),\qquad\text{as }\vartheta\to 0.
    \end{equation}
    Now we compute the partial derivatives of the function $F$. Observe that $F$ can be calculated in the following way 
    \begin{equation}
         F(s_1,s_2) = \frac 12 (f(s_1)+f(s_2)) (s_2-s_1) - \int_{s_1}^{s_2} f(x) \de x.
    \end{equation}
     As a consequence, we compute that
    \begin{equation}
     \begin{split}
          \frac{\partial}{\partial s_1} F(s_1,s_2) &= \frac 12 f'(s_1) (s_2-s_1) + \frac 12 \big( f(s_1) - f (s_2)\big)\\
              \frac{\partial}{\partial s_2} F(s_1,s_2) &= \frac 12 f'(s_2) (s_2-s_1) + \frac 12 \big( f(s_1) - f (s_2)\big).
    \end{split}
    \end{equation}
    Combining these two relations with \eqref{eq:derivativeofF}, we conclude that 
    \begin{multline}
        \frac{\omega^2}{r^2}\frac{\partial z}{\partial \phi} (\phi,\omega,r;t) = - \frac 12 f (p(\phi+\omega t)) \cdot p'(\phi), \\
        + \frac 12 [ f'(p(\phi+\omega t)) p(\phi+\omega t) -   f (p(\phi+\omega t))]\cdot p'(\phi+ \omega t).
    \end{multline}
    Now, recall that $f$ is twice differentiable in $0$ with positive second derivative $c$, therefore we have that 
    \begin{equation}
       f(x)= \frac{1}{2} c x^2 + o(x^2)\qquad \text{and} \qquad f'(x)=  c x + o(x).
     \end{equation}
     Using these relations, together with \eqref{eq:treddisperato}, we can conclude that 
    \begin{equation}
    \begin{split}
         \frac{\omega^2}{r^2} \,\frac{\partial z}{\partial \phi} (\phi,\omega,r;t) &= \frac 12  f'(p(\phi+\omega t)) p(\phi+\omega t)  p'(\phi+ \omega t)- \frac 12 f (p(\phi+\omega t)) [p'(\phi) + p'(\phi+\omega t)] \\
         &=  \frac 12 [c p(\phi+\omega t) + o (p(\phi+\omega t)) ] \,p(\phi+\omega t)  p'(\phi+ \omega t) \\
         &\quad - \frac 14 [c p(\phi+\omega t)^2 + o (p(\phi+\omega t)^2)] [p'(\phi) + p'(\phi+\omega t)] \\
         &=\frac{c}{2 y_0^3} (\omega t)^2 - \frac{c}{2y_0^3} (\omega t)^2 + o(t^2)= o(t^2).
    \end{split}
    \end{equation}
    This concludes the proof. 
\end{proof}

As a consequence of the these lemmas, we obtain the following estimate of the quantity $J(\phi,\omega,r;t)$, as $t\to 0$.

\begin{corollary}\label{cor:finalmente}
Given $\varepsilon>0$ sufficiently small, for $\Leb^1$-almost every $\phi\in [0,2\Sbb^\circ)$, every $r>0$ and $\omega\neq 0$, there exist two positive constants $C=C(\phi,\omega,r)$ and $\rho=\rho(\phi,\omega,r)$ such that
    \begin{equation*}
        J(\phi,\omega,r;t) = C(1 + O(\varepsilon)) |t|^5,\qquad \forall\, t\in[-\rho,\rho],
    \end{equation*}
in the Notation \ref{notation:brilliant_notation}.
\end{corollary}

\begin{proof}
    Let $\phi\in [0,2\Sbb^\circ)$ be a differentiability point for the map $C_\circ$ and such that the conclusion of Lemma \ref{lem:Jestimate2} holds, and fix $r>0$ and $\omega\neq 0$.
    Observe that, on the one hand, as a consequence of Lemma \ref{lem:Jestimate1} and Lemma \ref{lem:Jestimate3}, we have that
    \begin{equation}
        \left|\frac{\partial z}{\partial \phi} \det(M_2) \right|(\phi,\omega,r;t), \, \left| \frac{\partial z}{\partial r} \det(M_3)\right|(\phi,\omega,r;t) = o(t^5), \qquad \text{as }t \to 0.
    \end{equation}
    On the other hand, Lemma \ref{lem:Jestimate1} and Lemma \ref{lem:Jestimate2} ensure that there exist positive constants $C=C(\phi,\omega,r),\rho=\rho(\phi,\omega,r)>0$ such that
    \begin{equation}
        \left|\frac{\partial z}{\partial \omega} \det(M_1) \right|(\phi,\omega,r;t)= C(1 + O(\varepsilon)) |t|^5,\qquad \forall\, t\in[-\rho,\rho]
    \end{equation}
    where, in particular, $\rho$ has to be smaller than the constant identified by Lemma \ref{lem:Jestimate2}.
    Up to taking a smaller $\rho$ and keeping in mind \eqref{eq:J}, we may conclude that 
    \begin{equation}
          J(\phi,\omega,r;t) = C(1 + O(\varepsilon)) |t|^5, \qquad \forall\, t\in[-\rho,\rho].
    \end{equation}
\end{proof}

\begin{remark}
    Note that, in the \sr Heisenberg group, the contraction rate of volumes along geodesic is exactly $t^5$, cf. \cite{MR3852258}. In our setting, we are able to highlight the same behavior for the Jacobian determinant of the exponential map $J(\phi,\omega,r;t)$, as $t\to0$. 
\end{remark}

Now that we know the behaviour of $J(\phi,\omega,r;t)$ as $t\to 0$, in the next proposition, we obtain a statement similar to Proposition \ref{prop:mago,isnegnaci}, which will allow us to disprove the $\cd(K,N)$ condition in the Heisenberg group. In particular, the proof of the following proposition uses Corollary \ref{cor:finalmente} and some ideas developed in \cite[Prop.\ 3.1]{MR4201410}.

In our setting, we define the \emph{midpoint map} as: 
\begin{equation}\label{eq:midpoint_hei}
    \M(p,q):=e_{\frac12}\left(\gamma_{pq}\right),\qquad \text{if } p\star q^{-1}\notin\{x=y=0\},
\end{equation}
where $\gamma_{pq}:[0,1]\to\hei$ is the unique geodesic joining $p$ and $q$, given by Theorem \ref{thm:geod_Heisenberg}. Similarly, we define the \emph{inverse geodesic map} $I_m$ (with respect to $m\in\hei$) as:
\begin{equation}\label{eq:inverse_geodesic_hei}
    I_m(q)=p, \qquad \text{if there exists }x\in\hei\text{ such that }\M(p,q)=m.
\end{equation}

\begin{remark}
    Recall the definition of midpoint map in \eqref{eq:midpoint} and inverse geodesic map in \eqref{eq:inverse_geodesic}. Both maps were defined using the differential structure of a smooth \sF manifold, however they are characterized by the metric structure of the space. In particular, if the norm is sufficently regular, they coincide with \eqref{eq:midpoint_hei} and \eqref{eq:inverse_geodesic_hei}.
\end{remark}

\begin{prop}\label{prop:dopotantafatica}
    Let $\hei$ be the \sF Heisenberg group, equipped with a strictly convex and $C^{1,1}$ norm. For $\Leb^1$-almost every $\phi\in[0,2\Sbb^\circ)$, every $r>0$ and $\omega\neq 0$, there exists a positive constant $\rho=\rho(\phi,\omega,r)$ such that for every $t\in[-\rho,\rho]$:
    \begin{enumerate}
        \item[(i)] the inverse geodesic map $I_\e$ is well-defined and $C^1$ in a neighborhood of $G(\phi,\omega,r;t)$;
        \item[(ii)] the midpoint map $\M$ is well-defined and $C^1$ in a neighborhood of $(\e,G(\phi,\omega,r;t))$, moreover
        \begin{equation}
        \label{eq:suitable_jacobian_estimate}
           \big|\det d_{G(\phi,\omega,r;t)} \M(\e,\cdot) \big| \leq \frac 1{2^4}.
        \end{equation}
    \end{enumerate}
\end{prop}

\begin{proof}
    Take $\varepsilon$ sufficiently small, let $\phi$ be an angle for which the conclusion of Corollary \ref{cor:finalmente} holds. Fix $r>0$ and $\omega\neq 0$, and let $\rho=\rho(\phi,\omega,r)$ be the (positive) constant identified by Corollary \ref{cor:finalmente}. Let $t\in [-\rho,\rho]$ and consider the map $E_t:T^*_\e\hei\to\hei$ defined as 
    \begin{equation}\label{eq:esponenzialepazzo}
        E_t(\phi,\omega,r) :=  G(\phi,\omega,r;t)= \big(x(\phi,\omega,r;t), y(\phi,\omega,r;t),  z(\phi,\omega,r;t)\big),
    \end{equation}
    where $G$ is defined in \eqref{eq:extended_exponential_map}. 
    Note that $J(\phi,\omega,r;t)$ is the Jacobian of $ E_t(\phi,\omega,r)$ and in particular, since $t\in [-\rho,\rho]$, Corollary \ref{cor:finalmente} ensures that $J(\phi,\omega,r;t)>0$. Then, from the inverse function theorem, we deduce that $E_t$ is locally invertible in a neighborhood $B_t\subset \hei$ of $E_t(\phi,\omega,r)$ with $C^1$ inverse $E_t^{-1}:B_t \to T^*_\e\hei$. Then, according to Theorem \ref{thm:geod_Heisenberg} and Proposition \ref{prop:geodesics_nestogol}, the curve $[-t,t] \ni s \mapsto G(\phi,\omega,r;s)$ is the unique geodesic connecting $G(\phi,\omega,r;-t)$ and $G(\phi,\omega,r;-t)$, and such that $G(\phi,\omega,r;0)=\e$, provided that $\rho$ is sufficiently small. Hence, we can write the map $I_\e:B_t \to \R^3$ as
    \begin{equation}
        I_\e(q)= E_{-t}(E_t^{-1}(q)),\qquad\forall\,q\in B_t.
    \end{equation}
    Therefore, the map $I_\e$ is $C^1$ on $B_t$, being a composition of $C^1$ functions, proving item {\slshape (i)}. 
    
    With an analogous argument, the midpoint map (with first entry $\e$), $\M_\e(\cdot):=\M(\e,\cdot):B_t \to \R^3$, can be written as 
    \begin{equation}\label{eq:boh}
        \M_\e(q)= E_{t/2}(E_t^{-1}(q)),\qquad\forall\,q\in B_t.
    \end{equation}
    As before, we deduce this map is well-defined and $C^1$. To infer regularity of the midpoint map in a neighborhood of $(\e,G(\phi,\omega,r;t))$, we take advantage of the underline group structure, in particular of the left-translations \eqref{eq:left_translations}, which are isometries. Indeed, note that 
    \begin{equation}
        \M(p,q)= L_p\left(\M_\e( L_{p^{-1}} (q))\right), \qquad\forall\, p,q \in\hei,
    \end{equation}
    and, for every $(p,q)$ in a suitable neighborhood of $(\e,G(\phi,\omega,r;t))$, we have $L_{p^{-1}} (q)\in B_t$, therefore $\M$ is well-defined and $C^1$. Finally, keeping in mind \eqref{eq:boh} and applying Corollary \ref{cor:finalmente}, we deduce that 
    \begin{equation}
    \begin{split}
        \big|\det d_{G(\phi,\omega,r;t)} \M_\e(\cdot) \big|&= \big|\det d_{(\phi,\omega,r)} E_{t/2} \big| \cdot \big|\det d_{(\phi,\omega,r)} E_t \big|^{-1} \\
        &= J(\phi,\omega,r;t/2) \cdot J(\phi,\omega,r;t)^{-1}\\
        &= \frac{C(1 + O(\varepsilon)) |t/2|^5}{C(1 + O(\varepsilon)) |t|^5}= \frac{1}{2^5}\,(1 + O(\varepsilon)) \leq \frac{1}{2^4},
    \end{split}
    \end{equation}
    where the last inequality is true for $\varepsilon$ sufficiently small. This concludes the proof of item {\slshape (ii)}.
\end{proof}

\begin{theorem}\label{thm:noCDC11}
    Let $\hei$ be the \sF Heisenberg group, equipped with a strictly convex and $C^{1,1}$ norm and with a smooth measure $\m$. Then, the metric measure space $(\hei,\di_{SF},\m)$ does not satisfy the Brunn--Minkowski inequality $\bm(K,N)$, for every $K\in\R$ and $N\in (1,\infty)$.
\end{theorem}

\begin{proof}
    Take an angle $\phi$ for which the conclusion of Proposition \ref{prop:dopotantafatica} holds, fix $r>0$, $\omega\neq 0$ and call $\gamma$ the curve 
    \begin{equation}
        \R \ni s \mapsto \gamma(s):= G(\phi,\omega,r;s).
    \end{equation}
    Fix $t\in (0,\rho]$, where $\rho=\rho(\phi,\omega,r)$ is the positive constant identified by Proposition \ref{prop:dopotantafatica}. Recall the map $E_t$ (see \eqref{eq:esponenzialepazzo}) from the proof of Proposition \ref{prop:dopotantafatica}. $E_t$ is invertible, with $C^1$ inverse, in a neighborhood $B_t\subset \hei$ of $E_t(\phi,\omega,r)= \gamma(t)$. Consider the function 
    \begin{equation}
        s \mapsto \Phi(s) :=  \p_1 \big[ E_t^{-1} \big(L_{\gamma(s)^{-1}}( \gamma(t+s))\big)\big],
    \end{equation}
    where $\p_1$ denotes the projection onto the first coordinate. Observe that, for $s$ sufficiently small, $L_{\gamma(s)^{-1}}( \gamma(t+s)) \in B_t$, thus, $\Phi$ is well-defined and $C^1$ (being composition of $C^1$ functions) in an open interval $I\subset\R$ containing $0$. Moreover, note that $\Phi(s)$ is the initial angle for the geodesic joining $\e$ and $\gamma(s)^{-1}\star\gamma(t+s)$.  
    Now, we want to prove that there exists an interval $\tilde I \subset I$ such that, for $\Leb^1$-almost every $s\in \tilde I$, $\Phi(s)$ is an angle for which the conclusion of Proposition \ref{prop:dopotantafatica} holds. We have two cases, either $\Phi' \equiv 0$ in $I$ or there is $\bar s\in I$ such that $\Phi'(\bar s) \neq 0$. In the first case, since by definition $\Phi(0)=\phi$, we deduce that $\Phi(s)\equiv \phi$, thus the claim is true. In the second case, since $\Phi$ is $C^1$, we can find an interval $\tilde I \subset I$ such that $\Phi'(s) \neq 0$ for every $s\in \tilde I$. Then, consider 
    \begin{equation}
        J:=\{\psi\in \Phi(\tilde I): \psi \text{ is an angle for which Proposition \ref{prop:dopotantafatica} holds}\}\subset\Phi(\tilde I)
    \end{equation}
    and observe that $J$ has full $\Leb^1$-measure in $\Phi(\tilde I)$. Therefore, the set $\tilde J:=\Phi^{-1}(J)\subset \tilde I$ has full $\Leb^1$-measure in $\tilde I$, it being the image of $J$ through a $C^1$ function with non-null derivative. Thus the claim is true also in this second case. 
    
    At this point, let $\bar s\in \tilde I$ such that $\Phi(\bar s)$ is an angle for which the conclusion of Proposition \ref{prop:dopotantafatica} holds and consider 
    \begin{equation}
        \bar\rho:=\rho\big( E_t^{-1} \big(L_{\gamma(\bar s)^{-1}}( \gamma(t+\bar s))\big)\big)>0.
    \end{equation} 
    For every $s \in [-\bar\rho,\bar\rho] \setminus\{0\}$, from Proposition \ref{prop:dopotantafatica}, we deduce that the inverse geodesic map $I_\e$ and the midpoint map $\M$ are well-defined and $C^1$ in a neighborhood of $G\big(E_t^{-1} \big(L_{\gamma(\bar s)^{-1}}( \gamma(t+\bar s))\big);s\big)$ and $\big(\e, G\big(E_t^{-1} \big(L_{\gamma(\bar s)^{-1}}( \gamma(t+\bar s))\big);s\big)\big)$, respectively. Moreover, we have that 
    \begin{equation}
           \big|\det d_{G(E_t^{-1} (L_{\gamma(\bar s)^{-1}}( \gamma(t+\bar s)));s)} \M(\e,\cdot) \big| \leq \frac 1{2^4}.
    \end{equation}
    Observe that, since the left-translations are smooth isometries, the inverse geodesic map $I_{\gamma(\bar s)}$ is well-defined and $C^1$ in a neighborhood of $\gamma(\bar s + s)$, in fact it can be written as
    \begin{equation}
        I_{\gamma(\bar s)} (p) = L_{\gamma(\bar s)} \big[I_\e \big(L_{\gamma(\bar s)^{-1}} (p)\big)\big],
    \end{equation}
    and $L_{\gamma(\bar s)^{-1}} \big(\gamma(\bar s + s)\big)=G\big(E_t^{-1} \big(L_{\gamma(\bar s)^{-1}}( \gamma(t+\bar s))\big);s\big)$.
    Similarly, we can prove that the midpoint map is well-defined and $C^1$ in a neighborhood of $(\gamma(\bar s),\gamma(\bar s + s))$, with
    \begin{equation}
        \big|\det d_{\gamma(\bar s + s)} \M(\gamma(\bar s),\cdot) \big| \leq \frac 1{2^4}.
    \end{equation}

    In conclusion, up to restriction and reparametrization, we can find a geodesic $\eta:[0,1]\to \hei$ with the property that, for $\Leb^1$-almost every $\bar s\in [0,1]$, there exists $\lambda(\bar s)>0$ such that, for every $s\in [\bar s -\lambda(\bar s), \bar s + \lambda(\bar s)]\cap [0,1]\setminus\{\bar s\}$, the inverse geodesic map $I_{\eta(\bar s)}$ and the midpoint map $\M$ are well-defined and $C^1$ in a neighborhood of $\eta (s)$ and $(\eta(\bar s),\eta(s))$ respectively, and in addition
    \begin{equation}
        \big|\det d_{\eta( s)} \M(\eta(\bar s),\cdot) \big| \leq \frac 1{2^4}.
    \end{equation}
    Set $\lambda(s)=0$ on the (null) set where this property is not satisfied and consider the set 
    \begin{equation}
        T:=\left\{(s,t)\in [0,1]^2\, : \, t\in [s-\lambda(s),s + \lambda(s)]\right\}. 
    \end{equation}
    Observe that, introducing for every $\epsilon>0$ the set
    \begin{equation}
        D_\epsilon:= \{(s,t)\in [0,1]^2\, : \, |t-s|<\epsilon\},
    \end{equation}
    we have that 
    \begin{equation}\label{eq:quasinoncicredo}
        \frac{\Leb^2(T \cap D_\epsilon)}{\Leb^2( D_\epsilon)} =\frac{\Leb^2(T \cap D_\epsilon)}{2\epsilon-\epsilon^2} \to 1, \qquad \text{as }\,\epsilon \to 0.
    \end{equation}
    On the other hand, we can find $\delta>0$ such that the set $
    \Lambda_\delta :=\{s\in [0,1] \, :\, \lambda(s)>\delta\}$ satisfies $\Leb^1(\Lambda_\delta) > \frac 34$. In particular, for every $\epsilon<\delta$ sufficiently small we have that 
    \begin{equation}\label{eq:quasinoncicredo2}
        \Leb^2\left(\left\{(s,t)\in [0,1]^2\,:\, \frac{s+t}{2}\not\in  \Lambda_\delta \right\} \cap D_\epsilon \right) < \frac{1}{2} \epsilon.
    \end{equation}
    Therefore, putting together \eqref{eq:quasinoncicredo} and \eqref{eq:quasinoncicredo2}, we can find $\epsilon<\delta$ sufficiently small such that 
    \begin{equation}
        \Leb^2\left(T \cap D_\epsilon \cap \left\{(s,t)\in [0,1]^2\,:\, \frac{s+t}{2}\in  \Lambda_\delta \right\} \right) > \frac 12 \Leb^2( D_\epsilon).
    \end{equation}
    Then, since the set $D_\epsilon$ is symmetric with respect to the diagonal $\{s=t\}$, we can find $\bar s\neq\bar t$ such that 
    \begin{equation}
        (\bar s,\bar t),(\bar t,\bar s) \in T \cap D_\epsilon \cap \left\{(s,t)\in [0,1]^2\,:\, \frac{s+t}{2}\in  \Lambda_\delta \right\} .
    \end{equation}
    In particular, this tells us that: 
    \begin{itemize}
        \item[(i)] $\bar t\in [\bar s-\lambda(\bar s),\bar s + \lambda(\bar s)]$ and $\bar s\in [\bar  t-\lambda(\bar  t),\bar t + \lambda(\bar  t)]$;
        \item[(ii)] $|\bar t-\bar s|< \epsilon <\delta$;
        \item[(iii)] $ \frac{\bar s+\bar t}{2}\in  \Lambda_\delta$.
    \end{itemize}
    Now, on the one hand, (i) ensures that the midpoint map $\M$ is well-defined and $C^1$ in a neighborhood of $(\eta(\bar s), \eta(\bar t))$ with  
    \begin{equation}
        \big|\det d_{\eta( \bar t)} \M(\eta(\bar s),\cdot) \big| \leq \frac 1{2^4} \qquad \text{and} \qquad \big|\det d_{\eta( \bar s)} \M(\cdot, \eta(\bar t)) \big| \leq \frac 1{2^4}.
    \end{equation}
    While, on the other hand, the combination of (ii) and (iii) guarantees that the inverse geodesic map $I_{\eta(\frac{\bar s+\bar t}{2})}$ is well-defined and $C^1$ in a neighborhood of $\eta(\bar s)$ and in a neighborhood of $ \eta(\bar t)$ respectively. Indeed, we have:
    \begin{equation}
        \bar s,\bar t\in \left[\frac{\bar s+\bar t}{2}- \delta,\frac{\bar s+\bar t}{2}+ \delta\right] \subset \left[\frac{\bar s+\bar t}{2}- \lambda\left(\frac{\bar s+\bar t}{2}\right),\frac{\bar s+\bar t}{2}+ \lambda\left(\frac{\bar s+\bar t}{2}\right)\right],
    \end{equation}
    and, by the very definition of $\lambda(\cdot)$, we obtain the claimed regularity of the inverse geodesic map.
    
    Once we have these properties, we can repeat the same strategy used in the second part of the proof of Theorem \ref{thm:casosmooth} and contradict the Brunn--Minkowski inequality $\bm(K,N)$ for every $K\in \R$ and every $N\in (1,\infty)$.
\end{proof}

\begin{remark}
\label{rmk:flexiamo}
    If we want to replicate the strategy of Theorem \ref{thm:casosmooth}, we ought to find a short geodesic $\gamma:[0,1]\to\hei$ such that 
    \begin{itemize}
        \item[(i)] the midpoint map $\M$ is $C^1$ around $(\gamma(0),\gamma(1))$ and satisfies a Jacobian estimates at $\gamma(1)$ of the type \eqref{eq:suitable_jacobian_estimate};
        \item[(ii)] the midpoint map $\M$ satisfies a Jacobian estimates at $\gamma(0)$ of the type \eqref{eq:suitable_jacobian_estimate};
        \item[(iii)] the inverse geodesic map $\I_{\gamma(1/2)}$, with respect to $\gamma(1/2)$, is $C^1$ around $\gamma(0)$ and $\gamma(1)$. 
    \end{itemize}
    Proposition \ref{prop:dopotantafatica} guarantees the existence of a large set $\mathscr A\subset T_{\gamma(0)}^*\hei$ of initial covectors for which the corresponding geodesic $\gamma$ satisfies (i). The problem arises as the set $\mathscr A$ of ``good'' covectors depends on the base point and is large only in a measure-theoretic sense. A simple ``shortening'' argument, mimicking the strategy of the smooth case, is sufficient to address (ii). However, once the geodesic is fixed, we have no way of ensuring that (iii) is satisfied. In particular, it may happen that the map $\I_{\gamma(1/2)}$ does not fit within the framework of Proposition \ref{prop:dopotantafatica} item {\slshape (i)}, as the corresponding initial covector may fall outside the hypothesis.  To overcome such a difficulty, we use a density-type argument to choose \emph{simultaneously} an initial point and an initial covector in such a way that (i)--(iii) are satisfied.
\end{remark}

\subsection{Failure of the \texorpdfstring{$\mathsf{MCP}(K,N)$}{MCP(K,N)} condition for singular norms} \label{sec:noMCP}
In this section we prove Theorem \ref{thm:intro3}, showing that the measure contraction property (see Definition \ref{def:mcp}) can not hold in a \sF Heisenberg group, equipped with a strictly convex, singular norm. Our strategy is based on the observation that, in this setting, geodesics exhibit a branching behavior, despite being unique (at least for small times).

\begin{theorem}\label{thm:noCDnonC1}
    Let $\hei$ be the \sF Heisenberg group, equipped with a strictly convex norm $\normdot$ which is not $C^1$, and let $\m$ be a smooth measure on $\hei$. Then, the metric measure space $(\hei, \di_{SF}, \m)$ does not satisfy the measure contraction property $\MCP(K,N)$ for every $K\in \R$ and $N\in (1,\infty)$.
\end{theorem}


\begin{proof}
For simplicity, we assume $\m=\Leb^3$. As it is apparent from the proof, the same argument can be carried out in the general case.

According to Proposition \ref{prop:propunderduality}, since $\normdot$ is not $C^1$, its dual norm $\normdot_*$ is not strictly convex. In particular, there exists a straight segment contained in the sphere $S^{\norm{\cdot}_*}_1(0)=\partial\Omega^\circ$. Since the differential structure of the Heisenberg group is invariant under rotations around the $z$-axis, we can assume without losing generality that this segment is vertical in $\R^2\cap\{x>0\}$, i.e. there exists $\bar x \in \R$ and an interval $I:=[y_0,y_1]\subset \R$ such that 
\begin{equation*}
    \{\bar x\}\times I \subset \partial \Omega^\circ.
\end{equation*}
Moreover, we can take the interval $I$ to be maximal, namely for every $y\not\in I$ we have $(\bar x,y)\not\in \Omega^\circ$ (see Figure \ref{fig:MCP1}). Let $\psi_0 \in [0,2\Sbb^\circ)$ be such that $Q_{\psi_0}=(\bar x,y_0)$, then 
it holds that
\begin{equation}\label{eq:sincosflat}
    (\bar x, y)= Q_{\psi_0 + (y- y_0)\bar x}, \qquad \text{for every }y\in I.
\end{equation}
As a consequence, we have that 
\begin{equation}\label{eq:120.5}
   \cosomp(\psi_0 + (y-y_0)\bar x)= \bar x \quad\text{and}\quad \sinomp(\psi_0 + (y-y_0)\bar x)= y, \qquad \text{for }y\in I.
\end{equation}
Let $y_2= \frac 12 (y_0+y_1)$ and $\phi_0= \psi_0 + \frac{1}{2}(y_1-y_0) \bar x$, so that $(\bar x, y_2)= Q_{\phi_0}$ by \eqref{eq:sincosflat}. 
Moreover, take $\phi_1>\psi_1:= \psi_0+(y_1-y_0)\bar x$ sufficiently close to $\psi_1$ (so that $Q_{\phi_1}$ is not in the flat part of $\partial\Omega^\circ$) and call $\bar r= \phi_1-\phi_0>0$. We are now going to prove that there exists a suitably small neighborhood $\mathscr A\subset T_0^*\hei\cong [0,2\Sbb^\circ)\times \R\times [0,\infty)$ of the point $(\phi_0,\bar r,\bar r)$\footnote{ Here the angle $\phi_0$ has to be intended modulo $2\Sbb^\circ$.} such that 
\begin{equation}\label{eq:positiveset}
    \Leb^3\big( G(\mathscr A;1)\big) >0.
\end{equation}
For proving this claim, one could argue directly by computing the Jacobian of the map $G(\cdot,1)$ at the point $(\phi_0,\bar r,\bar r)$, however the computations are rather involved and do not display the geometrical features of the space. Thus, we instead prefer to present a different strategy, which highlights the interesting behaviour of geodesics.

Consider the map 
\begin{equation*}
    F(\phi,\omega,r):= \big(x(\phi,\omega,r;1),y(\phi,\omega,r;1)\big),
\end{equation*}
where $x(\phi,\omega,r;t),y(\phi,\omega,r;t)$ are defined as in \eqref{eq:extended_exponential_map}, and observe that
\begin{equation*}
    F(\phi_0,\bar r,\bar r)= (\sinomp(\phi_1) - \sinomp(\phi_0), \cosomp(\phi_0) - \cosomp(\phi_1))= (\sinomp(\phi_1) - y_2, \bar x -\cosomp(\phi_1)).
\end{equation*}
Proceeding with hindsight, let $\varepsilon>0$ such that $\varepsilon<\min\{\frac12(\phi_1-\psi_1),\frac 14 (\psi_1-\phi_0)\}$ and
 consider the intervals $I_\phi=[\phi_0-\varepsilon,\phi_0+\varepsilon]$ and $I_r=[\bar r-\varepsilon,\bar r + \varepsilon]$, then the set $F(I_\phi\times I_r\times I_r)$ is a neighborhood of $F(\phi_0,\bar r,\bar r)$. Indeed, due to our choice of $\phi_1$ the set 
\begin{equation}\label{eq:insiemebuffo}
    \{F(\phi_0,r,r) \, :\, r \in [\bar r - \varepsilon/2, \bar r + \varepsilon/2] \}\subset \R^2
\end{equation} 
is a curve that is not parallel to the $x$-axis. Moreover, for every small $\delta$ such that $|\delta| < \psi_1-\phi_0 = \phi_0-\psi_0$ and every $r\in [\bar r - \varepsilon/2, \bar r + \varepsilon/2]$, the equalities in \eqref{eq:120.5} imply the following relation:
\begin{equation*}
\begin{split}
    F(\phi_0+ \delta,r-\delta, r-\delta)&= (\sinomp(\phi_0+ r) - \sinomp(\phi_0+\delta), \cosomp(\phi_0+\delta) - \cosomp(\phi_0+r)) \\
    &= (\sinomp(\phi_0+ r) - \sinomp(\phi_0)- \delta/\bar x, \bar x - \cosomp(\phi_0+r) ) \\ 
    &= (-\delta/\bar x,0) + F(\phi_0, r, r).
\end{split}
\end{equation*}
This shows that $F(I_\phi\times I_r\times I_r)$ contains all the sufficiently small horizontal translation of the set in \eqref{eq:insiemebuffo} (see Figure \ref{fig:MCP2}), so it is a neighborhood of $F(\phi_0,\bar r,\bar r)$. In particular $\Leb^2(F(I_\phi\times I_r\times I_r))>0$. 

\begin{figure}[h]

    \begin{minipage}[c]{.5\textwidth}

    \centering
    \begin{tikzpicture}
        \filldraw[black] (-2,-0.5) circle (1.5pt);
        \filldraw[black] (2,-1.5) circle (1.5pt);
        \filldraw[black] (2,-0.4) circle (1.5pt);
        \filldraw[black] (2,0.7) circle (1.5pt);
        \filldraw[black] (0.6,1.95) circle (1.5pt);
        
        \draw[very thick] (1.7,-2) .. controls (1.8,-2) and (1.95,-1.9) ..(2,-1.5) -- (2,0.7).. controls (1.95,1.7) and (1.7,1.9) ..(0,2);
        \draw [thick](2.4,-1.5)--(2.4,0.7);
        \draw [thick](2.3,-1.5)--(2.5,-1.5);
        \draw [thick](2.3,0.7)--(2.5,0.7);
        \node at (-1.7,0.2)[label=south:$O$] {};
        \node at (-0.8,2.5)[label=south:$\partial\Omega^\circ$] {};
        \node at (0.8,1.12)[label=south:${(\bar x, y_1)= Q_{\psi_1}}$] {};
        \node at (0.8,0.02)[label=south:${(\bar x, y_2)= Q_{\phi_0}}$] {};
        \node at (0.8,-1.06)[label=south:${(\bar x, y_0)= Q_{\psi_0}}$] {};
        \node at (1.1,2.7)[label=south:$Q_{\phi_1}$] {};
        \node at (2.3,-0.4)[label=east:$I$] {};
        
    \end{tikzpicture}
    \caption{The flat part of $\partial\Omega^\circ$.}
    \label{fig:MCP1}
    
\end{minipage}%
\begin{minipage}[c]{.5\textwidth}

    \centering
    \begin{tikzpicture}[scale=1.3]
        \fill[color=blue!10!white](1.55,-0.6) ..controls (1.52,-0.9) ..(1.45,-1.2)--(2.25,-1.2)..controls (2.32,-0.9) ..(2.35,-0.6)--cycle;
         \draw[ thick] (-2,-1.7) .. controls (-2,-1.8) and (-1.9,-1.95) ..(-1.5,-2) -- (0.7,-2).. controls (1.7,-1.95) and (1.9,-1.7) ..(2,0);
         \draw[very thick, blue](1.55,-0.6) ..controls (1.52,-0.9) ..(1.45,-1.2)--(2.25,-1.2)..controls (2.32,-0.9) ..(2.35,-0.6)--cycle;
         \filldraw[red] (1.95,-0.6) circle (1.5pt);
         \filldraw[red] (1.85,-1.2) circle (1.5pt);
         \draw[very thick, red](1.95,-0.6) ..controls (1.92,-0.9) ..(1.85,-1.2);
         \node at (-0.4,0)[label=south:${\color{red}{\{F(\phi_0,r,r) \, :\, r \in [\bar r - \varepsilon/2, \bar r + \varepsilon/2] \}}}$] {};
         \node at (0.25,-0.6)[label=south:${\color{blue}{F(I_\phi\times I_r\times I_r) \supset}}$] {};
         \filldraw[black] (-0.4,-2) circle (1.5pt);
         \node at (-0.4,-1.4)[label=south:$O$] {};
        \filldraw[white] (0,-2.5) circle (1.5pt);
        \filldraw[white] (0,1.10) circle (1.5pt);
    \end{tikzpicture}
   \caption{Estimate of the set $F(I_\phi\times I_r\times I_r)$.}
    \label{fig:MCP2}
    
\end{minipage}
\end{figure}

Now we claim that, for every point $(\tilde x,\tilde y,\tilde z)= G(\tilde \psi,\tilde \omega,\tilde r;1)$ with $\tilde \psi \in I_\phi$, $\tilde \omega \in I_r$ and $\tilde r \in I_r$, there exists an interval $J_z\ni \tilde z$ (depending on $\tilde x$ and $\tilde y$) such that 
\begin{equation}
\label{eq:claim_interval_included}
    \{(\tilde x, \tilde y, z)\,:\, z\in J_z\} \subset G([\tilde \psi - \varepsilon, \tilde \psi + \varepsilon], [\tilde\omega - \varepsilon, \tilde\omega + \varepsilon], [\tilde r - \varepsilon, \tilde r + \varepsilon];1).
\end{equation}
 This is enough to prove \eqref{eq:positiveset}, indeed, on the one hand, \eqref{eq:claim_interval_included} implies that
\begin{equation}\label{eq:Jay-Z}
    \{(\tilde x, \tilde y, z)\,:\, z\in J_z\} \subset G(I'_{\psi}\times  I'_r\times I'_r;1),
\end{equation}
where $I'_{\psi}=[ \phi_0 - 2\varepsilon,  \phi_0 + 2\varepsilon]$,  and $I'_{r}=[\bar r - 2\varepsilon, \bar r + 2\varepsilon]$.
On the other hand, since \eqref{eq:Jay-Z} holds for every point $(\tilde x, \tilde y)\in F(I_\phi\times I_r\times I_r)$, we deduce that
 \begin{equation}
     \Leb^3\big(G(I'_{\psi}\times I'_r\times I'_r;1)\big)\geq \int_{F(I_\phi\times I_r\times I_r)}\Leb^1(J_z(\tilde x,\tilde y))\de \tilde x\de \tilde y >0.
 \end{equation}
which implies \eqref{eq:positiveset} with $\mathscr A=I'_{\psi}\times I'_r\times I'_r$.
%
%
%

We proceed to the proof of claim \eqref{eq:claim_interval_included}: let $(\tilde x,\tilde y,\tilde z)=G(\tilde \psi,\tilde \omega,\tilde r;1)$ with $\tilde \psi \in I_\phi$, $\tilde \omega \in I_r$ and $\tilde r \in I_r$ and consider the family of parallel lines 
\begin{equation}
\label{eq:parallel_lines}
\big\{ l(s)=\{y=s + kx\} \,:\, s \in \R\big\}
\end{equation}
in $\R^2$, following the direction identified by the vector $(\tilde x,\tilde y)$, see Figure \ref{fig:MCP3}. Call $S'\subset \R^2$ the sphere $\partial\Omega^\circ$ dilated by $\frac{\tilde r}{\tilde w}$ and rotated by $-\frac{\pi}{2}$. Then, there exists $\bar s \in \R$ such that $l(\bar s)$ intersects $S'$ in the points 
\begin{equation}
\frac{\tilde r}{\tilde w}(\sinomp(\tilde \psi), -\cosomp(\tilde \psi))\qquad\text{ and }\qquad\frac{\tilde r}{\tilde w}(\sinomp(\tilde \psi+ \tilde r), -\cosomp(\tilde \psi+\tilde r)).
\end{equation}

\begin{figure}[ht]

    \begin{minipage}[c]{.5\textwidth}

    \centering
    \begin{tikzpicture}[scale=1.5]

        \draw[very thick] (-2,-1.7) .. controls (-2,-1.8) and (-1.9,-1.95) ..(-1.5,-2) -- (0.7,-2).. controls (1.7,-1.95) and (1.9,-1.7) ..(2,0);
        \draw (-0.5,-2.2)--(2.2,-0.4);
        \draw[red] (-0.8,-2.2)--(2.2,-0.2);
        \draw (-0.2,-2.2)--(2.2,-0.6);
        \draw (-1.1,-2.2)--(2.2,-0);
        \filldraw[red] (-0.5,-2) circle (1pt);
        \draw[very thick, red,->] (-0.5,-2) -- (1.97,-0.35);
        \node at (0.5,-0.5)[label=south:${\color{red}{(\tilde x, \tilde y)}}$] {};
        \node at (2.5,0.1)[label=south:${\color{red}{l(\bar s)}}$] {};
        \node at (-1.6,-1.4)[label=south:$S'$] {};
    \end{tikzpicture}
    \caption{The line $l(\bar s)$ identifies $(\tilde x, \tilde y)$.}
    \label{fig:MCP3}
    
\end{minipage}
\hfill
\begin{minipage}[c]{.5\textwidth}

    \centering
    \begin{tikzpicture}[scale=1.5]
         \fill[blue!10!white](-0.5,-2) -- (1.97,-0.35)--(2,-2);
         \fill[white] (-2,-1.7) .. controls (-2,-1.8) and (-1.9,-1.95) ..(-1.5,-2) -- (0.7,-2).. controls (1.7,-1.95) and (1.9,-1.7) ..(2,0)--(2,-2)--(-0.5,-2);

        \draw[very thick] (-2,-1.7) .. controls (-2,-1.8) and (-1.9,-1.95) ..(-1.5,-2) -- (0.7,-2).. controls (1.7,-1.95) and (1.9,-1.7) ..(2,0);
        \draw (-0.8,-2.2)--(2.2,-0.2);
        \filldraw[red] (-0.5,-2) circle (1pt);
        \filldraw[red] (1.97,-0.35) circle (1pt);
        \draw[very thick, red] (-0.5,-2) -- (1.97,-0.35);
        \node at (0.5,-0.6)[label=south:${\color{red}{d(s)}}$] {};
        \node at (1,-1.3)[label=south:${\color{blue}{a(s)}}$] {};
        \node at (2.5,0.1)[label=south:${l( s)}$] {};
        \node at (-1.6,-1.4)[label=south:$S'$] {};
    \end{tikzpicture}
   \caption{Definition of $d(s)$ and $a(s)$.}
    \label{fig:MCP4}
    
\end{minipage}
\end{figure}

Let $a(s)$ be the function that associates to $s$ the area inside $S'$ and below $l(s)$ and let $d(s)$ be the function that associates to $s$ the (Euclidean) distance between the two intersections of $l(s)$ with $S'$ (see Figure \ref{fig:MCP4}). In particular, by our choice of $\bar s $, we have $d(\bar s)=\norm{(\tilde x, \tilde y)}_{eu}$ and, according to Proposition \ref{prop:z=area}, $a(\bar s)=\tilde z$. Moreover, note that, by Lemma \ref{lem:facilecontodianalisi1}, the function 
\begin{equation}\label{eq:increasingratio}
    s \mapsto \frac{a(s)}{d(s)^2} \text{ is strictly increasing.}
\end{equation}
Now, for every $s$ close enough to $\bar s$, the line $l(s)$ intersects $S'$ in the points
\begin{equation}
	\frac{\tilde r}{\tilde \omega}(\sinomp( \psi(s)), -\cosomp(\psi(s)))\qquad\text{ and }\qquad\frac{\tilde r}{\tilde \omega}(\sinomp( \psi(s)+  r(s)), -\cosomp( \psi(s)+r(s))),
\end{equation}
with $\psi(s)\in [\tilde \psi - \varepsilon, \tilde \psi + \varepsilon]$ and $r(s)\in [\tilde r - \varepsilon/2, \tilde r + \varepsilon/2]$. By Proposition \ref{prop:z=area} and our choice of parallel lines in \eqref{eq:parallel_lines}, we deduce that 
\begin{equation*}
    G\bigg(\psi(s), r(s), r(s) \frac{\norm{(\tilde x, \tilde y)}_{eu}}{d(s)};1\bigg)= \bigg(\tilde x, \tilde y, \frac{\norm{(\tilde x, \tilde y)}_{eu}^2}{d(s)^2} \cdot a(s)\bigg).
\end{equation*}
Observe that, since $d$ is a continuous function and $d(\bar s)= \norm{(\tilde x, \tilde y)}_{eu}$, for every $s$ sufficiently close to $\bar s$ we have 
\begin{equation*}
    r(s) \in [\tilde r-\varepsilon,\tilde r+\varepsilon]\subset I'_r \qquad \text{and} \qquad r(s) \frac{d(s)}{\norm{(\tilde x, \tilde y)}_{eu}} \in [\tilde r - \varepsilon, \tilde r + \varepsilon] \subset I'_r.
\end{equation*}
Then, \eqref{eq:increasingratio} is sufficient to conclude the existence of an interval $J_z\subset\R$ as in \eqref{eq:claim_interval_included}. This concludes the proof of claim \eqref{eq:positiveset} with the choice $\mathscr A=I'_{\psi}\times I'_r\times I'_r$.  

Finally, we are ready to disprove the measure contraction property $\MCP(K,N)$, taking as marginals
\begin{equation}
\mu_0:=\delta_\e\qquad\text{and}\qquad\mu_1:= \frac{1}{\Leb^3(G(\mathscr A;1))}\, \Leb^3|_{G(\mathscr A;1)}.
\end{equation}
Note that, thanks to our construction of the set $\mathscr A$, the curve $t\mapsto G(\lambda;t)$, with $\lambda\in \mathscr A$, is the unique geodesic joining the origin and $G(\lambda;1)$ (cf. Theorem \ref{thm:geod_Heisenberg}). Therefore, according to Remark \ref{rmk:SIUUUUUUU}, it is enough to contradict \eqref{eq:tj_pantaloncini} with $ A'=A= G(\mathscr A;1)$. 
In particular, we prove that there exists $t_0\in(0,1)$ such that
\begin{equation}
\label{eq:very_thin_support}
    M_t(\{\e\},A) \subset \{y=0,z=0\}, \qquad \forall\,t<t_0.
\end{equation}
 To this aim, fix any $(\phi,\omega,r)\in \mathscr A$ and note that, for every $t<\frac{\psi_1-\phi}{\omega}$, \eqref{eq:120.5} implies that
\begin{equation*}
    \cosomp(\phi+\omega t ) = \bar x \qquad\text{and}\qquad \sinomp(\phi+\omega t )= \sinomp(\phi) + \frac{\omega t}{\bar x}.
\end{equation*}
From these relations, it follows immediately that 
\begin{equation*}
    y(\phi,\omega,r;t)=0\qquad\text{and}\qquad z(\phi,\omega,r;t) = 0,
\end{equation*}
for every $t<\frac{\psi_1-\phi}{\omega }$. Observe that, by our choice of $\varepsilon$ small enough, $\frac{\psi_1-\phi}{\omega}$ is bounded from below by a positive constant uniformly as $\phi\in I'_\phi$ and $\omega \in I'_r$, thus ensuring the existence of a constant $t_0\in (0,1)$ for which \eqref{eq:very_thin_support} holds.
\end{proof}

\begin{remark}
    In the last step of the proof of the preceding theorem, we established the existence of a family of branching geodesics: namely those corresponding to a flat part of $\partial\Omega^\circ$. In particular, when $\hei$ is equipped with a strictly convex and singular norm, geodesics can branch, although they are unique. This is remarkable as examples of branching spaces usually occur when geodesics are not unique. 
\end{remark}

\begin{lemma}\label{lem:facilecontodianalisi1}
    Let $f:\R\to\R$ be a concave and $C^1$ function. Assume that there exist $\alpha_0<\beta_0$ such that
    \begin{equation}
    \label{eq:lemegaassunzioni}
        f(\alpha_0)=f(\beta_0)=0\qquad\text{and}\qquad f>0\, \ \text{on}\ \,(\alpha_0,\beta_0).
    \end{equation}
    For every $s\in [0,\max f)$, define $\alpha(s)<\beta(s)$ such that
    \begin{equation}
        \{y=s\}\cap {\rm Graph}(f)=\{\left(\alpha(s),s\right);\left(\beta(s),s\right)\}.
    \end{equation}
    Denote by $a(s)$ the area enclosed by the line $\{y=s\}$ and the graph of $f$, and by $d(s):=\beta(s)-\alpha(s)$ (see Figure \ref{fig:byconvexity}). Then, 
    \begin{equation}
        [0,\max f)\ni s\mapsto \frac{a(s)}{d^2(s)}\qquad\text{is strictly decreasing}.
    \end{equation}
\end{lemma}

  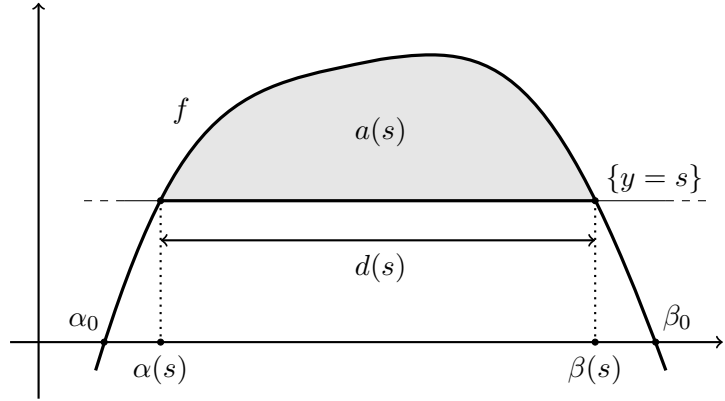
\begin{figure}[h]
      \centering
       \begin{tikzpicture}[scale=0.75]
\fill[color=black!10!white](3.76,2.5) --(-3.86,2.5)--(-3.86, 5.4)--(3.76,5.4)--cycle;
\fill[white] (-5,-0.5)..controls (-3.5,4.4) and (-2.5,4.5) .. (0,5)..controls (2,5.3) and (3,5).. (5,-0.5)--(5,5.5)--(-5,5.5)--cycle;
\draw[very thick] (-5,-0.5)..controls (-3.5,4.4) and (-2.5,4.5) .. (0,5)..controls (2,5.3) and (3,5).. (5,-0.5);
\draw[thick,->] (-6,-1)--(-6,6);
\draw[thick,->] (-6.5,0)--(6,0);
\filldraw[black] (3.76,2.5) circle (1.5pt);
\filldraw[black] (-3.86,2.5) circle (1.5pt);
\draw[very thick](3.76,2.5) --(-3.86,2.5);
\draw[thick,dotted](3.76,2.5) --(3.76,0);
\filldraw[black] (3.76,0) circle (1.5pt);
\draw[thick,dotted](-3.86,2.5) --(-3.86,0);
\filldraw[black] (-3.86,0) circle (1.5pt);
\draw (-4.5,2.5)-- (5,2.5);
\draw[dashed](5,2.5)--(5.8,2.5);
\draw[dashed](-4.5,2.5)--(-5.3,2.5);
\draw[<->, thick](-3.86,1.8)--(3.76,1.8);

\node at (0,3.1)[label=north:$a(s)$] {};
\node at (0,0.7)[label=north:$d(s)$] {};
\node at (-3.5,3.5)[label=north:$f$] {};
\node at (4.8,3.5)[label=south:${\{y=s\}}$] {};
\node at (-3.86,-1.1)[label=north:$\alpha(s)$] {};
\node at (-5.2,-0.15)[label=north:$\alpha_0$] {};
\node at (5.2,-0.15)[label=north:$\beta_0$] {};
\node at (3.76,-1.1)[label=north:$\beta(s)$] {};
\filldraw[black] (-4.85,0) circle (1.5pt);
\filldraw[black] (4.82,0) circle (1.5pt);

\end{tikzpicture}
\caption{Representation of the quantities $\alpha(s)$, $\beta(s)$, $a(s)$ and $d(s)$. }
    \label{fig:byconvexity}
  \end{figure}

\begin{proof}
    Fix $0\leq s_1< s_2 < \max f$, then it is sufficient to prove that 
    \begin{equation}\label{eq:strictlydecreasing}
        A_1:= a(s_1) > \frac{d^2(s_1)}{d^2(s_2)} a (s_2) =: A_2.
    \end{equation}
    Observe that, by definition, $\displaystyle a(s)=\int_{\alpha(s)}^{\beta(s)}(f(t)-s)\de t$, therefore:
    \begin{equation}
        A_1= \int_{\alpha(s_1)}^{\beta(s_1)} \big(f(t) - s_1 \big) \de t, \qquad
        A_2=\frac{d(s_1)}{d(s_2)}\int_{\alpha(s_1)}^{\beta(s_1)}\left[f\left(\alpha(s_2)+\frac{d(s_2)}{d(s_1)}(t-\alpha(s_1))\right)-s_2\right]\de t,
    \end{equation}
    where, for $A_2$, we used the change of variables $t\mapsto \alpha(s_1)+\frac{d(s_1)}{d(s_2)}(t-\alpha(s_2))$. For ease of notation, set $g$ to be the integrand of $A_2$, namely:
    \begin{equation}
        g(t):=\frac{d(s_1)}{d(s_2)}\left[f\left(\alpha(s_2)+\frac{d(s_2)}{d(s_1)}(t-\alpha(s_1))\right)-s_2\right],\qquad\forall\,t\in\R.
    \end{equation}
    Now, let $\tilde t\in [\alpha(s_1),\beta(s_1)]$ be such that
    \begin{equation}
        \tilde t = \alpha(s_2)+\frac{d(s_2)}{d(s_1)}(\tilde t-\alpha(s_1)).
    \end{equation}
    Note that, by linearity, $t\leq \tilde t$ if and only if $t\leq \alpha(s_2)+\frac{d(s_2)}{d(s_1)}(t-\alpha(s_1))$. Thus, for every $t\leq \tilde t$, the concavity of $f$ yields that 
    \begin{equation}
    \label{eq:ineq_first_der}
        g'(t) = f'\bigg(\alpha(s_2)+\frac{d(s_2)}{d(s_1)}(t-\alpha(s_1))\bigg) \leq f'(t).
    \end{equation}
    Therefore, observing that $g(\alpha(s_1))=0$, we deduce that, for every $t\leq \tilde t$,
    \begin{equation}
    \label{eq:ineq_integrands}
        g(t) = \int_{\alpha(s_1)}^{t} g'(r) \de r \leq \int_{\alpha(s_1)}^{t} f'(r) \de r = f(t) - s_1.
    \end{equation}
    The same inequality can be proved for every $t\geq \tilde t$, proceeding in a symmetric way. Thus, integrating both sides of \eqref{eq:ineq_integrands}, we obtain $A_1\geq A_2$. 
    Finally, observe that if $A_1=A_2$ then also \eqref{eq:ineq_first_der} is an equality, for every $t< \tilde t$. By concavity, this implies that $f'(t)\equiv c_1$ for every $t< \tilde t$. Analogously, $f'(t)\equiv c_2$ for every $t > \tilde t$ and \eqref{eq:lemegaassunzioni} implies that we must have $c_1\neq c_2$. Thus, $f$ is linear on $(\alpha(s_1)),\tilde t)$ and $(\tilde t,\beta(s_1)))$ and not differentiable at $\tilde t$. But, this contradicts that $f\in C^1(\R)$, proving claim \eqref{eq:strictlydecreasing}.
\end{proof}

\bibliography{biblio-carnot}

\begin{thebibliography}{BBLDS17}

\bibitem[ABB20]{ABB-srgeom}
Andrei Agrachev, Davide Barilari, and Ugo Boscain.
\newblock {\em A comprehensive introduction to sub-{R}iemannian geometry},
  volume 181 of {\em Cambridge Studies in Advanced Mathematics}.
\newblock Cambridge University Press, Cambridge, 2020.
\newblock From the Hamiltonian viewpoint, With an appendix by Igor Zelenko.

\bibitem[ABR18]{MR3852258}
A.~Agrachev, D.~Barilari, and L.~Rizzi.
\newblock Curvature: a variational approach.
\newblock {\em Mem. Amer. Math. Soc.}, 256(1225):v+142, 2018.

\bibitem[ABS21]{MR4294651}
Luigi Ambrosio, Elia Bru\'{e}, and Daniele Semola.
\newblock {\em Lectures on optimal transport}, volume 130 of {\em Unitext}.
\newblock Springer, Cham, [2021] \copyright 2021.
\newblock La Matematica per il 3+2.

\bibitem[ACS18]{MR3741391}
Paolo Albano, Piermarco Cannarsa, and Teresa Scarinci.
\newblock Regularity results for the minimum time function with {H}\"{o}rmander
  vector fields.
\newblock {\em J. Differential Equations}, 264(5):3312--3335, 2018.

\bibitem[Agr09]{MR2513150}
A.~A. Agrachev.
\newblock Any sub-{R}iemannian metric has points of smoothness.
\newblock {\em Dokl. Akad. Nauk}, 424(3):295--298, 2009.

\bibitem[AGS14]{AmbrosioGigliSavare11-2}
Luigi Ambrosio, Nicola Gigli, and Giuseppe Savar{\'e}.
\newblock Metric measure spaces with {R}iemannian {R}icci curvature bounded
  from below.
\newblock {\em Duke Math. J.}, 163(7):1405--1490, 2014.

\bibitem[AGS15]{AmbrosioGigliSavare12}
Luigi Ambrosio, Nicola Gigli, and Giuseppe Savar{\'e}.
\newblock Bakry-\'{E}mery curvature-dimension condition and {R}iemannian
  {R}icci curvature bounds.
\newblock {\em The Annals of Probability}, 43(1):339--404, 2015.

\bibitem[AP95]{MR1623472}
Dominique Az\'{e} and Jean-Paul Penot.
\newblock Uniformly convex and uniformly smooth convex functions.
\newblock {\em Ann. Fac. Sci. Toulouse Math. (6)}, 4(4):705--730, 1995.

\bibitem[AS04]{AS-GeometricControl}
Andrei Agrachev and Yuri~L. Sachkov.
\newblock {\em Control theory from the geometric viewpoint}, volume~87 of {\em
  Encyclopaedia of Mathematical Sciences}.
\newblock Springer-Verlag, Berlin, 2004.
\newblock Control Theory and Optimization, II.

\bibitem[AS20]{MR4061989}
Luigi Ambrosio and Giorgio Stefani.
\newblock Heat and entropy flows in {C}arnot groups.
\newblock {\em Rev. Mat. Iberoam.}, 36(1):257--290, 2020.

\bibitem[Bat22]{MR4506771}
David Bate.
\newblock Characterising rectifiable metric spaces using tangent spaces.
\newblock {\em Invent. Math.}, 230(3):995--1070, 2022.

\bibitem[BBLDS17]{MR3657277}
Davide Barilari, Ugo Boscain, Enrico Le~Donne, and Mario Sigalotti.
\newblock Sub-{F}insler structures from the time-optimal control viewpoint for
  some nilpotent distributions.
\newblock {\em J. Dyn. Control Syst.}, 23(3):547--575, 2017.

\bibitem[Ber94]{Bereszynski}
V.~N. Berestovski\u{\i}.
\newblock Geodesics of nonholonomic left-invariant inner metrics on the
  {H}eisenberg group and isoperimetrics of the {M}inkowski plane.
\newblock {\em Sibirsk. Mat. Zh.}, 35(1):3--11, i, 1994.

\bibitem[BGL14]{MR3155209}
Dominique Bakry, Ivan Gentil, and Michel Ledoux.
\newblock {\em Analysis and geometry of {M}arkov diffusion operators}, volume
  348 of {\em Grundlehren der mathematischen Wissenschaften [Fundamental
  Principles of Mathematical Sciences]}.
\newblock Springer, Cham, 2014.

\bibitem[BMR22]{barilari2022unified}
Davide Barilari, Andrea Mondino, and Luca Rizzi.
\newblock {U}nified synthetic {R}icci curvature lower bounds for {R}iemannian
  and sub-{R}iemannian structures, 2022.

\bibitem[BR18]{MR3848070}
Davide Barilari and Luca Rizzi.
\newblock Sharp measure contraction property for generalized {H}-type {C}arnot
  groups.
\newblock {\em Commun. Contemp. Math.}, 20(6):1750081, 24, 2018.

\bibitem[BR19]{MR3935035}
Davide Barilari and Luca Rizzi.
\newblock Sub-{R}iemannian interpolation inequalities.
\newblock {\em Invent. Math.}, 215(3):977--1038, 2019.

\bibitem[BR20]{MR4245620}
Zeinab Badreddine and Ludovic Rifford.
\newblock Measure contraction properties for two-step analytic sub-{R}iemannian
  structures and {L}ipschitz {C}arnot groups.
\newblock {\em Ann. Inst. Fourier (Grenoble)}, 70(6):2303--2330, 2020.

\bibitem[BS10]{MR2610378}
Kathrin Bacher and Karl-Theodor Sturm.
\newblock Localization and tensorization properties of the curvature-dimension
  condition for metric measure spaces.
\newblock {\em J. Funct. Anal.}, 259(1):28--56, 2010.

\bibitem[BT23]{borzatashiro}
Samuël Borza and Kenshiro Tashiro.
\newblock Measure contraction property, curvature exponent and geodesic
  dimension of sub-{F}insler {H}eisenberg groups, 2023.

\bibitem[Bus47]{Busemann}
Herbert Busemann.
\newblock The isoperimetric problem in the {M}inkowski plane.
\newblock {\em Amer. J. Math.}, 69:863--871, 1947.

\bibitem[Cio90]{MR1079061}
Ioana Cioranescu.
\newblock {\em Geometry of {B}anach spaces, duality mappings and nonlinear
  problems}, volume~62 of {\em Mathematics and its Applications}.
\newblock Kluwer Academic Publishers Group, Dordrecht, 1990.

\bibitem[CM17]{MR3648975}
Fabio Cavalletti and Andrea Mondino.
\newblock Sharp and rigid isoperimetric inequalities in metric-measure spaces
  with lower {R}icci curvature bounds.
\newblock {\em Invent. Math.}, 208(3):803--849, 2017.

\bibitem[CM21]{MR4309491}
Fabio Cavalletti and Emanuel Milman.
\newblock The globalization theorem for the curvature-dimension condition.
\newblock {\em Invent. Math.}, 226(1):1--137, 2021.

\bibitem[FPR20]{FPR-sing-lapl}
Valentina Franceschi, Dario Prandi, and Luca Rizzi.
\newblock On the essential self-adjointness of singular sub-{L}aplacians.
\newblock {\em Potential Anal.}, 53(1):89--112, 2020.

\bibitem[Gig15]{Gigli12}
Nicola Gigli.
\newblock On the differential structure of metric measure spaces and
  applications.
\newblock {\em Mem. Amer. Math. Soc.}, 236(1113):vi+91, 2015.

\bibitem[GMS13]{GigMonSav}
Nicola Gigli, Andrea Mondino, and Giuseppe Savar{\'e}.
\newblock Convergence of pointed non-compact metric measure spaces and
  stability of ricci curvature bounds and heat flows.
\newblock {\em Proceedings of the London Mathematical Society}, 2013.

\bibitem[Jea14]{MR3308372}
Fr\'{e}d\'{e}ric Jean.
\newblock {\em Control of nonholonomic systems: from sub-{R}iemannian geometry
  to motion planning}.
\newblock SpringerBriefs in Mathematics. Springer, Cham, 2014.

\bibitem[Jui09]{MR2520783}
Nicolas Juillet.
\newblock Geometric inequalities and generalized {R}icci bounds in the
  {H}eisenberg group.
\newblock {\em Int. Math. Res. Not. IMRN}, 2009(13):2347--2373, 2009.

\bibitem[Jui21]{MR4201410}
Nicolas Juillet.
\newblock Sub-{R}iemannian structures do not satisfy {R}iemannian
  {B}runn-{M}inkowski inequalities.
\newblock {\em Rev. Mat. Iberoam.}, 37(1):177--188, 2021.

\bibitem[KdP21]{kerdreux2021local}
Thomas Kerdreux, Alexandre d'Aspremont, and Sebastian Pokutta.
\newblock Local and global uniform convexity conditions, 2021.

\bibitem[LD11]{MR2865538}
Enrico Le~Donne.
\newblock Metric spaces with unique tangents.
\newblock {\em Ann. Acad. Sci. Fenn. Math.}, 36(2):683--694, 2011.

\bibitem[LD15]{MR3283670}
Enrico Le~Donne.
\newblock A metric characterization of {C}arnot groups.
\newblock {\em Proc. Amer. Math. Soc.}, 143(2):845--849, 2015.

\bibitem[Lok19]{Nestoeurogol}
L.~V. Lokutsievski\u{\i}.
\newblock Convex trigonometry with applications to sub-{F}insler geometry.
\newblock {\em Mat. Sb.}, 210(8):120--148, 2019.

\bibitem[Lok21]{Nestogol}
L.~V. Lokutsievskiy.
\newblock Explicit formulae for geodesics in left-invariant sub-{F}insler
  problems on {H}eisenberg groups via convex trigonometry.
\newblock {\em J. Dyn. Control Syst.}, 27(4):661--681, 2021.

\bibitem[LV09]{MR2480619}
John Lott and C\'{e}dric Villani.
\newblock Ricci curvature for metric-measure spaces via optimal transport.
\newblock {\em Ann. of Math. (2)}, 169(3):903--991, 2009.

\bibitem[Mil21]{MR4373164}
Emanuel Milman.
\newblock The quasi curvature-dimension condition with applications to
  sub-{R}iemannian manifolds.
\newblock {\em Comm. Pure Appl. Math.}, 74(12):2628--2674, 2021.

\bibitem[MPR22a]{seminalpaper}
Mattia Magnabosco, Lorenzo Portinale, and Tommaso Rossi.
\newblock The {B}runn--{M}inkowski inequality implies the {$\cd$} condition in
  weighted {R}iemannian manifolds.
\newblock {\em arXiv preprint arXiv:2209.13424}, 2022.

\bibitem[MPR22b]{Magnabosco-Portinale-Rossi:2022b}
Mattia Magnabosco, Lorenzo Portinale, and Tommaso Rossi.
\newblock The strong {B}runn--{M}inkowski inequality and its equivalence with
  the {$\cd$} condition.
\newblock {\em arXiv preprint arXiv:2210.01494}, 2022.

\bibitem[MR23]{MR4562156}
Mattia Magnabosco and Tommaso Rossi.
\newblock Almost-{R}iemannian manifolds do not satisfy the curvature-dimension
  condition.
\newblock {\em Calc. Var. Partial Differential Equations}, 62(4):Paper No. 123,
  2023.

\bibitem[Oht07]{MR2341840}
Shin-ichi Ohta.
\newblock On the measure contraction property of metric measure spaces.
\newblock {\em Comment. Math. Helv.}, 82(4):805--828, 2007.

\bibitem[Oht09]{MR2546027}
Shin-ichi Ohta.
\newblock Finsler interpolation inequalities.
\newblock {\em Calc. Var. Partial Differential Equations}, 36(2):211--249,
  2009.

\bibitem[OS12]{MR2917125}
Shin-ichi Ohta and Karl-Theodor Sturm.
\newblock Non-contraction of heat flow on {M}inkowski spaces.
\newblock {\em Arch. Ration. Mech. Anal.}, 204(3):917--944, 2012.

\bibitem[Rif13]{MR3110060}
Ludovic Rifford.
\newblock Ricci curvatures in {C}arnot groups.
\newblock {\em Math. Control Relat. Fields}, 3(4):467--487, 2013.

\bibitem[Rif14]{MR3308395}
Ludovic Rifford.
\newblock {\em Sub-{R}iemannian geometry and optimal transport}.
\newblock SpringerBriefs in Mathematics. Springer, Cham, 2014.

\bibitem[Riz16]{MR3502622}
Luca Rizzi.
\newblock Measure contraction properties of {C}arnot groups.
\newblock {\em Calc. Var. Partial Differential Equations}, 55(3):Art. 60, 20,
  2016.

\bibitem[Roc70]{Rockafellar+1970}
Ralph~Tyrell Rockafellar.
\newblock {\em Convex Analysis}.
\newblock Princeton University Press, Princeton, 1970.

\bibitem[RS23]{rizzi2023failure}
Luca Rizzi and Giorgio Stefani.
\newblock Failure of curvature-dimension conditions on sub-{R}iemannian
  manifolds via tangent isometries, 2023.

\bibitem[Stu06a]{MR2237206}
Karl-Theodor Sturm.
\newblock On the geometry of metric measure spaces. {I}.
\newblock {\em Acta Math.}, 196(1):65--131, 2006.

\bibitem[Stu06b]{MR2237207}
Karl-Theodor Sturm.
\newblock On the geometry of metric measure spaces. {II}.
\newblock {\em Acta Math.}, 196(1):133--177, 2006.

\bibitem[vRS05]{MR2142879}
Max-K. von Renesse and Karl-Theodor Sturm.
\newblock Transport inequalities, gradient estimates, entropy, and {R}icci
  curvature.
\newblock {\em Comm. Pure Appl. Math.}, 58(7):923--940, 2005.

\end{thebibliography}

\bibliographystyle{alpha} 

\end{document}